\documentclass{amsart}
\usepackage{a4wide,mathtools,amssymb,color,pict2e}

\usepackage[all]{xy}

\newcommand{\IN}{\mathbb N}
\newcommand{\IR}{\mathbb R}
\newcommand{\w}{\omega}
\newcommand{\IQ}{\mathbb Q}
\newcommand{\IZ}{\mathbb Z}
\newcommand{\defeq}{\coloneqq}
\newcommand{\updownarrows}{\;{\uparrow}\!{\downarrow}\;}
\newcommand{\dist}{\mathsf{d}}
\newcommand{\Sf}{{\mathsf S}}

\newcommand{\Ra}{\Rightarrow}
\newcommand{\dom}{\mathsf{dom}}
\newcommand{\e}{\varepsilon}
\newcommand{\lcd}{\delta}
\newcommand{\ld}{\delta}

\newcommand{\F}{\mathcal F}
\newcommand{\U}{\mathcal U}
\newcommand{\rng}{\mathsf{rng}}

\newtheorem{theorem}{Theorem}[section]
\newtheorem{corollary}[theorem]{Corollary}
\newtheorem{example}[theorem]{Example}
\newtheorem{lemma}[theorem]{Lemma}
\newtheorem{claim}[theorem]{Claim}
\newtheorem{question}[theorem]{Question}
\newtheorem{proposition}[theorem]{Proposition}
\newtheorem{problem}[theorem]{Problem}

\theoremstyle{definition}

\newtheorem{definition}[theorem]{Definition}
\newtheorem{remark}[theorem]{Remark}

\title{Banakh spaces and their geometry}
\author{Taras Banakh}
\address{Ivan Franko National University of Lviv (Ukraine), and Jan Kochanowski University in Kielce (Poland)}
\email{t.o.banakh@gmail.com}

\begin{document}
\begin{abstract}
Following Will Brian, we define a metric space $X$ to be {\em Banakh} if all nonempty spheres of positive radius $r$ in $X$ have cardinality $2$ and diameter $2r$. Standard examples of Banakh spaces are subgroups of the real line. In this paper we study the geometry of Banakh spaces, characterize Banakh spaces which are isometric to subgroups of the  real line, and also construct  Banakh spaces $X$ which do not embed into the real line and have a prescribed distance set $\dist[X^2]$.
 \end{abstract}
 \maketitle

\tableofcontents
\newpage

\section{Introduction}

This paper is motivated by the problem of metric characterization of the real line, posed by the author at {\tt MathOverflow} \cite{TB}. He suggested that a metric space $(X,\dist)$ is isometric to the real line if and only if for every $c\in X$ and every positive real number $r$, the sphere $$\Sf(c;r)\defeq\{x\in X:\dist(x,c)=r\}$$ contains exactly two points and has diameter $2r$. Will Brian proved in \cite{WB} that this characterization is indeed true for complete metric spaces and suggested to give the name ``Banakh space'' to any metric space $X$ whose  nonempty spheres consist of two points at maximal possible distance. 
\smallskip

\begin{definition} A metric space $(X,\dist)$ is defined to be a {\em Banakh space} if for every point $c\in X$ and real number $r\in\dist[X^2]$ there exist points $x,y\in X$ such that $\Sf(c;r)=\{x,y\}$ and $\dist(x,y)=2r$.
%
\end{definition}

Standard examples of Banakh spaces are subgroups of the real line. In this paper we study geometric properties of Banakh spaces,
characterize Banakh spaces which are isometric to subgroups of the  real line, and also construct  Banakh spaces $X$ which do not embed into the real line and have a prescribed distance set $\dist[X^2]$.

\section{Preliminaries} 
For a subset $X$ of the real line, let $X_+\defeq\{x\in X:x\ge 0\}$ and $\pm X=X\cup(-X)$. In particular, $\IR_+$ equals to the closed half-line $[0,\infty)$, and $\pm(\IR_+)=\IR$. The real line $\IR$ is considered as an Abelian metric group, endowed with the operation of addition of real numbers.

We denote by $\IZ$ and $\IQ$ the sets of integer and rational numbers, respectively. The set $\IZ_+$ coincides with the set $\w$ of finite ordinals. Every ordinal is identified with the set of smaller ordinals. 

For two subsets $A,B$ of the real line, let
$$AB\defeq\{a\cdot b:a\in A,\;b\in B\},\quad
A+B\defeq\{a+b:a\in A,\;b\in B\},\quad
A-B\defeq\{a-b:a\in A,\;b\in B\}.
$$The sets $A{\cdot}B,A+B,A-B$ are subsets of the real line. We shall often write $AB$ instead of $A{\cdot}B$. If the set $B$ is a singleton $\{b\}$, then we write $Ab$ instead of $A{\cdot}\{b\}$. 

For a metric space $X$ by $\dist_X$ (or just $\dist$ if $X$ is clear from the context) we denote the metric of $X$. 
For a point $x$ of a metric space $X$ and a set $D$ of real number, the set 
$$\Sf(x;D)\defeq\{y\in X:\dist(x,y)\in D\}=\bigcup_{r\in D}\Sf(x;r)$$
is called the {\em $D$-sphere} around the point $x$.

A function $f:X\to Y$ between two metric spaces $X,Y$ is called 
\begin{itemize}
\item  an {\em isometric embedding} if $\dist_Y(f(x),f(y))=\dist_X(x,y)$ for every points $x,y\in X$;
\item an {\em isometry} if $f$ is a surjective isometric embedding.
\end{itemize}
Two metric spaces $X,Y$ are {\em isometric} if there exists an isometry $f:X\to Y$ between $X$ and $Y$.

\section{Banakh subspaces of the real line}

For warming up, in this section we characterize Banakh subspaces of the real line.

\begin{theorem}\label{t:real-Banakh} For a subspace $X$ of the real line, the following conditions are equivalent.
\begin{enumerate}
\item The metric space $X$ is Banakh.
\item For every $x,y,z\in X$ we have $\{x+y-z,x-y+z\}\subseteq X$;
\item For every $c\in X$ the set $X-c$ is a subgroup of the real line.
\end{enumerate}
\end{theorem}

\begin{proof} $(1)\Ra (2)$ Assume that $X$ is Banakh. Given any points $x,y,z\in X$, we need to show that $\{x+y-z,x-y+z\}\subseteq X$. This inclusion holds if $y=z$. So, we assume that $y\ne z$, which implies that the real number $r\defeq |y-z|\in\dist[X^2]$ is strictly positive. Since the metric space $X$ is Banakh, the sphere $\Sf(x;r)=X\cap\{x-r,x+r\}$ contains exactly two points, which implies that $\{x+y-z,x-y+z\}=\{x+|y-z|,x-|y-z\}=\{x-r,x+r\}=\Sf(x;r)\subseteq X$.
\smallskip

$(2)\Ra(3)$ Assume that $\{x+y-z,x-y+z\}\subseteq X$ for every $x,y,z\in X$. Fix any $c\in X$. To show that $X-c$ is a subgroup of the real line, it suffices to check that $(X-c)-(X-c)\subseteq X-c$. Given any points $x,y\in X-c$, we have  $\{x+c,y+c\}\subseteq X$ and hence $x-y=(x+c)-(y+c)+c-c\in X-c$.
\smallskip

$(3)\Ra(1)$ Assume that for every $c\in X$, the set $X-c$ is a subgroup of the real line. To prove that $X$ is Banakh, take any point $c\in X$ and real number $r\in\dist[X^2]=\{|x-y|:x,y\in X\}$. Since $X-c$ is a subgroup of the real line, $\{|x-y|,-|x-y|\}=\{x-y,y-x\}=\{(x-c)-(y-c),(y-x)-(x-c\}\subseteq X-c$ and hence $\{c+r,c-r\}=\{c+|x-y|,c-|x-y|\}\subseteq X$. It follows that the sphere $\Sf(c;r)$ in $X$ coincides with the doubleton $\{c+r,c-r\}$ and has diameter $2r$, witnessing that $X$ is a Banakh space.
\end{proof}

\section{Uniqueness of points in Banakh spaces}

In this section we prove that a point of a Banakh space is uniquely determined by its distances from two distinct points of the Banakh space.
 
\begin{theorem}\label{t:GPS} Two points $x,y$ of a Banakh space $X$ are equal if and only if $\dist(a,x)=\dist(a,y)$ and $\dist(b,x)=\dist(b,y)$ for some distinct points $a,b\in X$.
\end{theorem}

\begin{proof} The ``only if'' part is trivial. To prove the ``if' part, assume that $x\ne y$ but $\dist(a,x)=\dist(a,y)$ and $\dist(b,x)=\dist(b,y)$ for some distinct points $a,b\in X$. The Banakh property of $X$ implies that $\dist(x,y)=2\dist(a,x)=2\dist(b,x)$ and hence all distances $\dist(a,x)$, $\dist(b,x)$, $\dist(a,y)$, $\dist(b,y)$ are equal to some positive real number $r$. Since $X$ is a Banakh space, the equality $\dist(a,x)=r=\dist(b,x)$ implies $\dist(a,b)=\dist(a,x)+\dist(x,b)=2r$. Since $X$ is a Banakh space, the sphere $\Sf(b;2r)\ni a$  contains a point $c\in X$ such that $\dist(a,c)=4r$. The triangle inequality implies
$$3r=\dist(a,c)-\dist(a,x)\le \dist(x,c)\le\dist(x,b)+\dist(b,c)=r+2r=3r$$and hence $\dist(x,c)=3r$. By analogy we can prove that $\dist(y,c)=3r$. The Banakh property implies $\dist(x,y)=\dist(x,c)+\dist(c,y)=6r$, which contradicts the triangle inequality $\dist(x,y)\le \dist(x,a)+\dist(a,y)=2r$. 
\end{proof}

\section{Discrete lines in Banakh spaces}

In this section we shall prove that any points $x,y$ of a Banakh space $X$ are contained in a unique subset $Z\subseteq X$ which is isometric to the discrete line $\IZ{\cdot}\dist(x,y)$, see Theorem~\ref{t:main1}. This fact will be derived from the following lemma.

\begin{lemma}\label{l:1} Let $a,b$ be two points of a Banakh space $X$ such that $\dist(a,b)=1$. For every $n\in\IN$ there exists a unique sequence of points $(x_k)_{k=-n}^n$ in $X$ such that $x_0=a$, $x_1=b$ and $\dist(x_i,x_j)=|i-j|$ for all $i,j\in\{-n,\dots,n\}$.
\end{lemma}

\begin{proof} The uniqueness of the sequence $(x_k)_{k=-n}^n$ follows from Theorem~\ref{t:GPS}. The existence will be proved by induction on $n$. To start the induction, let $x_0\defeq a$, $x_1\defeq b$ and apply the Banakh property to find a point $x_{-1}\in X$ such that $\dist(x_{-1},x_0)=1=\dist(x_0,x_1)$ and $\dist(x_{-1},x_1)=2$. It is clear that the sequence $(x_k)_{k=-1}^1$ has the required property.

Now assume that for some $n\in \IN$ there exists a sequence $(x_k)_{k=-n}^n$ such that $x_0=a$, $x_1=b$ and $\dist(x_i,x_j)=|i-j|$ for all $i,j\in\{-n,\dots,n\}$.
In particular, $\dist(x_0,x_n)=n=\dist(x_0,x_{-n})$. 

By the Banakh property, there exists a point $x_{n+1}\in X$ such that $\dist(x_{n+1},x_n)=1$ and $\dist(x_{n+1},x_{n-1})=2$.
Also by the Banakh property, there exists a point $x_{n+2}\in X$ such that $\dist(x_n,x_{n+2})=2$ and $\dist(x_{n-2},x_{n+2})=4$.

\begin{claim}\label{cl:1} $\dist(x_{n+1},x_{n+2})=1$.
\end{claim}

\begin{proof} By the Banakh property, there exists a point $x_{n+2}'\in X$ such that $\dist(x_{n+1},x_{n+2}')=1$ and $\dist(x_n,x'_{n+2})=2$. Assuming that $x_{n+2}\ne x_{n+2}'$ and taking into account that $\dist(x_n,x_{n+2})=2=\dist(x_n,x_{n+2}')$, we can apply the Banakh property and conclude that $\dist(x_{n+2},x_{n+2}')=4$. Taking into account that $\dist(x_{n-2},x_{n+2})=4$ and $\dist(x_{n+2},x_{n+2}')\le \dist(x_{n+2},x_n)+\dist(x_n,x_{n+2}')=2+2<8$, we conclude that  $x_{n+2}'=x_{n-2}$ and hence $\dist(x_{n+2},x'_{n+2})=\dist(x_{n+2},x_{n-2})=4$. The triangle inequality implies that $\dist(x_{n+2},x_{n+1})\le \dist(x_{n+2},x_n)+\dist(x_n,x_{n+1})=3=\dist(x_{n+2},x_{n+2}')-\dist(x_{n+2}',x_{n+1})\le \dist(x_{n+2},x_{n+1})$ and hence $\dist(x_{n+2},x_{n+1})=3$. By analogy we can prove that $\dist(x_{n+2},x_{n-1})=3$. Since $x_{n+1}\ne x_{n-1}$, the Banakh property ensures that $\dist(x_{n-1},x_{n+1})=6$, which contradicts the choice of $x_{n+1}$. This contradiction shows that $x_{n+2}'=x_{n+2}$ and hence $\dist(x_{n+1},x_{n+2})=1$. 
\end{proof}

\begin{claim}\label{cl:2} $\dist(x_{n-1},x_{n+2})=3$.
\end{claim}

\begin{proof} The triangle inequality implies that $$\dist(x_{n-1},x_{n+2})\le \dist(x_{n-1},x_n)+\dist(x_n,x_{n+2})=3=\dist(x_{n+2},x_{n-2})-\dist(x_{n-2},x_{n-1})\le \dist(x_{n+2},x_{n-1})$$ and hence $\dist(x_{n-1},x_{n+2})=3$. 
\end{proof}

\begin{claim}\label{cl:3} For every $i\in\{0,\dots,n+2\}$ we have $\dist(x_i,x_{n+2})=n+2-i$.
\end{claim}

\begin{proof} This claim will be proved by downward induction on $i$. For $i\in\{n-2,n-1,n+1,n+2\}$ the equality $\dist(x_{n+2},x_i)=(n+2)-i$ follows from the choice of $x_{n+2}$ or from Claims~\ref{cl:1}, \ref{cl:2}. Assume that for some $i\in\{n-2,\dots,1\}$ we know that $\dist(x_i,x_{n+2})=n+2-i$. Since $2i-(n+2)\ge -n$ the point $x_{2i-(n+2)}$ is well-defined and $\dist(x_i,x_{2i-(n+2)})=n+2-i$ by the property of the sequence $(x_k)_{k=-n}^n$. Also $\dist(x_{2i-(n+2)},x_{n+2})\ge \dist(x_{2i-(n+2)},x_n)-\dist(x_{n},x_{n+2})=(n-2i+(n+2))-2=2(n-i)>0$. Then $\dist(x_{2i-(n+2)},x_{n+2})=2\dist(x_i,x_{n+2})=2\dist(x_i,x_{2i-(n+2)})=2(n+2-i)$, by the Banakh property of $X$.  
The triangle inequality implies
\begin{multline*}\dist(x_{i-1},x_{n+2})\le \dist(x_{i-1},x_i)+\dist(x_{n+2},x_i)\\=1+(n+2-i)
=(2n+4-2i)-(i-1-(2i-n-2))\\=\dist(x_{2i-(n+2)},x_{n+2})-\dist(x_{i-1},x_{2i-(n+2)})\le \dist(x_{i-1},x_{n+2})
\end{multline*}
 and hence $\dist(x_{n+2},x_{i-1})=(n+2)-(i-1)$. This completes the proof of the inductive step, and also the proof of the equality $\dist(x_i,x_{n+2})=(n+2)-i$ for all $i\in\{n+2,\dots,0\}$.
\end{proof}

\begin{claim}\label{cl:4} For every $i\in\{0,\dots,n+1\}$ we have $\dist(x_i,x_{n+1})=n+1-i$.
\end{claim}

\begin{proof}  For $i=n$ the equality $\dist(x_i,x_{n+1})=1$ follows from the choice of $x_{n+1}$. So, assume that $i\le n$. Then $\dist(x_i,x_n)=n-i$ by the choice of the sequence $(x_k)_{k=-n}^n$ and $\dist(x_i,x_{n+2})=n+2-i$ by Claim~\ref{cl:3}. Applying the triangle inequality, we obtain
$$\dist(x_i,x_{n+1})\le \dist(x_i,x_{n})+\dist(x_n,x_{n+1})=n-i+1=\dist(x_i,x_{n+2})-\dist(x_{n+2},x_{n+1})\le \dist(x_i,x_{n+1})$$and hence $\dist(x_i,x_{n+1})=n+1-i$.
\end{proof}

By analogy with Claims~\ref{cl:1}--\ref{cl:4}, we can prove

\begin{claim}\label{cl:5} There exist points points $x_{-n-2}$ and $x_{-n-1}$ such that for every $i\in\{-n-2,-n-1\}$ and $j\in\{-n-2,-n-1,\dots,0\}$ the equality $\dist(x_i,x_j)=|i-j|$ holds. 
\end{claim}

\begin{claim}\label{cl:6} $\dist(x_{-n-1},x_{n+1})=2(n+1)$.
\end{claim}

\begin{proof} It follows from $\dist(x_{-n-1},x_0)=n+1=\dist(x_0,x_{n+1})$ and the Banakh property of $X$ that either $\dist(x_{-n-1},x_{n+1})=2(n+1)$ or $x_{n+1}=x_{-n-1}$. Let us show that the equality $x_{-n-1}=x_{n+1}$ is not possible. Otherwise, by the Banakh property, there exists a  point $c\in X$ such that $\dist(c,x_0)=n+1$ and $\dist(c,x_{n+1})=2(n+1)$.
The triangle inequality implies that $$\dist(c,x_{n})\le \dist(c,x_0)+\dist(x_0,x_{n})=n+1+n=\dist(c,x_{n+1})-\dist(x_{n+1},x_{n})\le \dist(c,x_{n})$$ and hence $\dist(c,x_{n})=2n+1$. By analogy we can prove that $\dist(c,x_{-n})=2n+1$. 
Since $x_{-n}\ne x_n$, the Banakh property implies $\dist(x_{-n},x_n)=2(2n+1)$ which contradicts the equality  $\dist(x_{-n},x_n)=2n$ holding by the inductive assumption. This contradiction shows that $x_{-n-1}\ne x_{n+1}$ and hence $\dist(x_{-n-1},x_{n+1})=2(n+1)$, by the Banakh property.
\end{proof}

\begin{claim} $\dist(x_{-n-2},x_{n+2})=2(n+2)$.
\end{claim}

\begin{proof} It follows from $\dist(x_{-n-2},x_0)=n+2=\dist(x_0,x_{n+2})$ and the Banakh property that either $\dist(x_{-n-2},x_{n+2})=2(n+2)$ or $x_{n+2}=x_{-n-2}$. Let us show that the equality $x_{-n-2}=x_{n+2}$ is not possible. Otherwise, by the Banakh property, there exists a  point $c\in X$ such that $\dist(c,x_0)=n+2$ and $\dist(c,x_{n+2})=2(n+2)$.
The triangle inequality implies that $$\dist(c,x_{n+1})\le \dist(c,x_0)+\dist(x_0,x_{n+1})=n+2+n+1=\dist(c,x_{n+2})-\dist(x_{n+2},x_{n+1})\le \dist(c,x_{n+1})$$ and hence $\dist(c,x_{n})=2n+3$. By analogy we can prove that $\dist(c,x_{-n-1})=2n+3$. 
By Claim~\ref{cl:6}, $x_{-n-1}\ne x_{n+1}$. Now the Banakh property of $X$ ensures that $\dist(x_{-n-1},x_{n+1})=2(2n+3)$ which contradicts Claim~\ref{cl:6}. This contradiction shows that $x_{-n-2}\ne x_{n+2}$ and hence $\dist(x_{-n-2},x_{n+2})=2(n+2)$, by the Banakh property of $X$.
\end{proof}

\begin{claim} For every $i,j\in\{-n-2,\dots,n+2\}$ we have $\dist(x_i,x_j)=|i-j|$.
\end{claim}

\begin{proof} We lose no generality assuming that $i<j$. If $0\le i$ or $j\le 0$, then the equality $\dist(x_i,x_j)=|i-j|$ follows from Claim~\ref{cl:3}, \ref{cl:4} and \ref{cl:5}. So, we assume that $i<0<j$. Claims~\ref{cl:3}, \ref{cl:4}, \ref{cl:5} and the triangle inequality imply
\begin{multline*}
\dist(x_i,x_j)\le \dist(x_i,x_0)+\dist(x_0,x_j)\\=-i+j=2(n+2)-(n+2-j)-(i-(-n-2))\\
=\dist(x_{-n-2},x_{n+2})-\dist(x_j,x_{n+2})-\dist(x_{-n-2},x_i)\le \dist(x_i,x_j)
\end{multline*}
and hence $\dist(x_i,x_j)=j-i=|i-j|$.
\end{proof}

Now we see that the sequence $(x_i)_{i=-n-1}^{n+1}$ has the required properties, which completes the inductive step, and also completes the proof of Lemma~\ref{l:1}. 
\end{proof}

The following theorem is the principal result of this section.

\begin{theorem}\label{t:main1} For any points $a,b$ of a Banakh space $X$ and the real number $r\defeq \dist(a,b)$, there exists a unique isometric embedding $\ell_{ab}:\IZ r\to X$ such that $\ell_{ab}(0)=a$ and $\ell_{ab}(r)=b$.
\end{theorem}

\begin{proof} If $r=0$, then the constant function $\ell_{ab}:r\IZ\to\{a\}\subseteq X$ is a required unique isometry with $\ell_{ab}(0)=a$ and $\ell_{ab}(r)=b$. 

So, assume that $r\ne0$ and consider the metric $\rho:X\times X\to\IR$ defined by $\rho(x,y)\defeq\frac{\dist(x,y)}{r}$. It is easy to see that the Banakh property of the metric space $(X,\dist)$ implies the Banakh property of the metric space $(X,\rho)$. Since $\rho(a,b)=1$, we can apply Lemma~\ref{l:1} and find a
 sequence of points $(x_k)_{k\in \IZ}$ in $X$ such that $x_0=a$, $x_1=b$ and $\dist(x_i,x_j)=\rho(x_i,x_j){\cdot}r=|i-j|{\cdot}r$ for all $i,j\in\IZ$. Then the function $\ell_{ab}:r\IZ\to X$, $\ell_{ab}:rn\mapsto x_n$, is an isometric embedding with $\ell_{ab}(0)=a$ and $\ell_{ab}(r)=b$. The uniqueness of $\ell_{ab}$ follows from Theorem~\ref{t:GPS}.
\end{proof}
  
\section{$G$-spheres in Banakh spaces, isometric to subgroups of $\IQ$} 

Let us recall that for a subset $G$ of the real line and a point $c$ of a metric space $X$, the set $$\Sf(c;G)\defeq\{x\in X:\dist(x,c)\in G\}$$is called the {\em $G$-sphere}  around the point $c$.
  
Theorem~\ref{t:main1} implies the following description of $\IZ r$-spheres in Banakh spaces.

\begin{theorem}\label{t:Sf-rZ} For any point $a\in X$ of a Banakh space $X$ and any $r\in\dist[X^2]$, the exists an isometry $\ell:\IZ r\to\Sf(a;\IZ r)$ such that $\ell(0)=a$.
\end{theorem}

\begin{proof} Since $X$ is a Banakh space, there exists a point $b\in\Sf(a;r)$. By Theorem~\ref{t:main1}, there exists an isometric embedding $\ell_{a,b}:\IZ r\to X$ such that $\ell_{ab}(0)=a$ and $\ell_{ab}(r)=b$. For every positive integer number $n$ we have $\{\ell_{ab}(nr),\ell_{ab}(-nr)\}\subseteq \Sf(a;nr)$. Since $X$ is a Banakh space, the sphere $\Sf(a;nr)$ contains exactly two points and hence is equal to $\{\ell_{ab}(nr),\ell_{ab}(-nr)\}$. Then $$\ell_{ab}[\IZ r]=\bigcup_{n\in\IZ_+}\{\ell_{ab}(nr),\ell_{ab}(-nr)\}=\bigcup_{n\in\IZ_+}\Sf(a;nr)=\Sf(a;\IZ_+r)=\Sf(a;\IZ r)$$and hence $\ell\defeq\ell_{ab}$ is a required isometry between $\IZ r$ and $\Sf(a;\IZ r)$.\end{proof}

Next, we characterize $G$-spheres in Banakh spaces, which are isometric to subghroups $G$ of the group $\IQ$ of rational numbers.

\begin{theorem}\label{t:main2} For a subgroup $G$ of the group $\IQ$ and a point $c$ of a Banakh space $X$, the $G$-sphere $\Sf(c;G)$ is isometric to $G$ if and only if $G_+\subseteq\dist[X^2]$.
\end{theorem}

\begin{proof} If the $G$-sphere $\Sf(c;G)$ is isometric to $G$, then $G_+=\dist[G^2]=\dist[\Sf(c;G)^2]\subseteq\dist[X^2]$. 

Now assuming that $G_+\subseteq\dist[X^2]$, we shall prove that the $G$-sphere $\Sf(c;G)$ is isometric to the group $G$. If $G=\{0\}$, then $\Sf(c;G)=\{c\}$ is isometric to the trivial group $G=\{0\}$. So, we assume that the group $G$ is not trivial. Since the group $G\subseteq\IQ$ is countable, there exists a sequence $(G_n)_{n\in\w}$ of nontrivial finitely generated subgroups $G_n\subseteq G_{n+1}$ of $\IQ$ such that $G=\bigcup_{n\in\w}G_n$. Since $\IQ=\bigcup_{n\in\w}\IZ\frac1{n!}$, every group $G_n$ is cyclic, being a  subgroup of a suitable cyclic subgroup $\IZ\frac1{m!}$ of $\IQ$. Then $G_n=\IZ r_n$ for a unique positive rational number $r_n\in G_+\subseteq\dist[X^2]$.  

Since $r_0\in G_+\subseteq\dist[X^2]$ and the metric space $X$ is Banakh, there exist a point $x_0\in \Sf(x;r_0)$. By Theorem~\ref{t:Sf-rZ}, for every $n\in\w$, there exists an isometry $\ell_n:\IZ r_n\to \Sf(c;\IZ r_n)$ such that $\ell_n(0)=c$. Taking into account that $\ell_n$ is an isometry and $r_0\in \IZ r_0=G_0\subseteq G_n=\IZ r_n$, we conclude that $x_0\in \Sf(c;r_0)=\{\ell_n(r_0),\ell_n(-r_0)\}$. If $x_0=\ell_n(r_0)$, then let $\ell'_n\defeq\ell_n$. If $x_0=\ell_n(-r_0)$, then define the isometry $\ell'_n:\IZ r_n\to\Sf(c;\IZ r_n)$ by the formula $\ell'_n(x)=\ell_n(-x)$ for $x\in \IZ r_n$. In both cases, $\ell'_n:\IZ r_n\to\Sf(c;r_n)$ is an isometry such that $\ell'_n(r_0)=x_0$. 

We claim that $\ell'_n=\ell'_{n+1}{\restriction}_{\IZ r_n}$ for each $n\in\w$. Indeed, $\ell'_n:\IZ r_n\to\Sf(c;\IZ r_n)$ and $\ell'_{n+1}:\IZ r_{n+1}\to\Sf(x;\IZ r_{n+1})$ are two isometries with $\ell'_n(0)=c=\ell'_{n+1}(0)$ and $\ell'_n(r_0)=x_0=\ell'_{n+1}(r_0)$.
Taking into account that the $r_n$-sphere $\Sf(c;\IZ r_n)$ is isometric to $\IZ r_n$,  we conclude that $\Sf(c;r_n)\cap\Sf(x_0;r_0- r_n)$ is a singleton consisting of a unique point $s_n\defeq \ell'_n(r_n)=\ell'_{n+1}(r_n)$. Now the uniqueness of the isometry $\ell_{cs_n}$ implies that $\ell'_n=\ell_{cs_n}=\ell'_{n+1}{\restriction}_{\IZ r_n}$.

Then $\ell\defeq\bigcup_{n\in\w}\ell'_n$ is an isometry of the group $G=\bigcup_{n\in\w}G_n=\bigcup_{n\in\w}\IZ r_n$ onto the $G$-sphere $\Sf(c;G)=\bigcup_{n\in\w}\Sf(c;G_n)$.
\end{proof}

\section{Banakh spaces, isometric to subgroups of the real line}

In this section we apply Banakh spaces to characterizing metric spaces that are isometric to certain subgroups of the real line. For some other characterizations of subsets of the real line, see \cite{BBKR}.  

\begin{theorem}\label{t:main3} A metric space $X$ is isometric to a subgroup $G\subseteq\IQ$ if and only if $X$ is a Banakh space with $\dist[X^2]=G_+$.
\end{theorem}

\begin{proof} The ``only if'' part of this characterization is trivial. To prove the ``if'' part, assume that $X$ is a Banakh space with $\dist[X^2]=G_+$. It follows from $0\in G_+=\dist[X^2]$ that the Banakh space $X$ is not empty and hence contains some point $x\in X$. By Theorem~\ref{t:main2}, there exists an isometry $f:\Sf(x;G)\to G$. It follows from $\dist[X^2]=G_+$ that $\Sf(x;G)=X$ and hence $f$ is an isometry between the Banakh space $X$ and the group $G$.
\end{proof} 

The following metric characterization of the real line generalizes a characterization of Will Brian \cite{WB} who proved that a metric space $X$ is isometric to the real line if and only if $X$ is a complete Banakh space with $\dist[X^2]=\IR_+$. 

\begin{theorem}\label{t:WB} A metric space $X$ is isometric to the real line if and only if $X$ is a complete Banakh space such that $G_+\subseteq\dist[X^2]$ for some non-cyclic  subgroup $G$ of $\IQ$.
\end{theorem}

\begin{proof} The ``only if'' part of this characterization is trivial. To prove the ``if'' part, assume that $X$ is a complete Banakh space such that $G_+\subseteq\dist[X^2]$ for some non-cyclic subgroup $G$ of $\IQ$. Since $0\in G_+\subseteq\dist[X^2]$, the Banakh space $X$ is not empty and hence $X$ contains some point $c\in X$. By Theorem~\ref{t:main2}, there exists an isometry $f:\Sf(c;G)\to G$ such that $f(c)=0$. Let $\overline{\Sf(c;G)}$ be the closure of the $G$-sphere $\Sf(c;G)$ in the complete metric space $X$. Being non-cyclic, the subgroup $G$ of $\IQ$ is dense in $\IR$.
By \cite[4.3.18]{Eng}, the isometry $f$ extends to an isometry $\bar f:\overline{\Sf(c;G)}\to \overline{G}=\IR$. It remains to show that $X=\overline{\Sf(c;G)}$. Since $X$ is a Banakh space, for any point $x\in X\setminus\{c\}$, the sphere $\Sf(c;r)$ of radius $r\defeq\dist(x,c)$ contains exactly two points. On the other hand, since $\bar f:\overline{\Sf(c;G)}\to\IR$ is an isometry with $\bar f(c)=f(c)=0$, $\Sf(c;r)\supseteq\{\bar f^{-1}(r),\bar f^{-1}(-r)\}$ and hence $x\in \Sf(c;r)=\{\bar f^{-1}(r),\bar f^{-1}(-r)\}\subseteq\overline{\Sf(c;G)}$. Therefore, $X=\overline{\Sf(c;G)}$ and $\bar f$ is an isometry between the Banakh space $X$ and the real line.
\end{proof}   

\begin{remark} Theorems~\ref{t:H1} and \ref{t:H2} show that both conditions (the completeness and $G_+\subseteq\dist[X^2]$) in Theorem~\ref{t:WB} are essential.
\end{remark}
 
\section{Orientation in Banakh spaces}

In this section we introduce the notion of orientation in a Banakh space and study its properties. This notion is a crucial tool for studying the structure of $\IQ r$-spheres in Banakh spaces. 

By Theorem~\ref{t:main1}, for any points $a,b$ of a Banakh space $X$, there exists a unique isometric embedding $\ell_{ab}:\IZ{\cdot}\dist(a,b)\to X$ such that $\ell_{ab}(0)=a$ and $\ell_{ab}(\dist(a,b))=b$. 

\begin{definition} Given points $o,x,y$ in a Banakh space $X$, we write  $ox\upuparrows oy$ (resp. $ox\updownarrows oy$) if $\ell_{ox}(m\dist(o,x))=\ell_{oy}(n\dist(o,y))$ for some numbers $n,m\in\IZ$ with $nm>0$ (resp. $nm<0)$.
\end{definition}

\begin{lemma}\label{l:upup} Let $o,x,y$ be points of a Banakh space $X$. 
\begin{enumerate}
\item If  $ox\upuparrows oy$, then $\ell_{ox}(n\dist(o,x))=\ell_{oy}(m\dist(o,y))$  for any nonzero numbers $n,m\in\IZ$ with $n\dist(o,x)=m\dist(o,y)$.
\item If  $ox\updownarrows oy$, then $\ell_{ox}(n\dist(o,x))=\ell_{oy}(m\dist(o,y))$  for any nonzero numbers $n,m\in\IZ$ with $n\dist(o,x)=-m\dist(o,y)$.
\end{enumerate}
\end{lemma}

\begin{proof} Assume that $ox\upuparrows oy$ or $ox\updownarrows oy$, and let $$s\defeq\begin{cases}\;\;\,1&\mbox{if $ox\upuparrows oy$};\\
-1&\mbox{if $ox\updownarrows oy$}.
\end{cases}
$$

Since $ox\upuparrows oy$ or $ox\updownarrows oy$, there exist numbers $p,q\in\IZ$ such that $spq>0$ and $\ell_{ox}(p\dist(o,x))=\ell_{oy}(q\dist(o,y))$. Then $|p|\dist(o,x)=\dist\big(o,\ell_{ox}(p\dist(o,x))\big)=\dist\big(o,\ell_{oy}(q\dist(o,y))\big)=|q|\dist(o,y)$. Now take any nonzero integer numbers $n,m$ such that $n\dist(o,x)=sm \dist(o,y)$. 

If $\dist(o,x)=0$, then $\dist(o,y)=\frac {|p|}{|q|}\dist(o,x)=0$ and hence $$\ell_{ox}(m\dist(o,x))=\ell_{0,x}(0)=o=\ell_{oy}(0)=\ell_{oy}(n\dist(o,y)).$$
So, we assume that $\dist(o,x)\ne0$.

Since $\ell_{ox}$ and $\ell_{oy}$ are isometric embeddings with $\ell_{ox}(0)=o=\ell_{oy}(0)$, the Banakh property ensures that $\ell_{ox}(n\dist(o,x))\in\{\ell_{oy}(m\dist(o,y)),\ell_{oy}(-m\dist(0,y))\}$. Assuming that $\ell_{ox}(n\dist(o,x))=\ell_{oy}(-m\dist(0,y))$, we conclude that
\begin{multline}\label{eq:mult1}
|p-n|\dist(o,x)=\dist\big(\ell_{ox}(p\dist(0,x)),\ell_{ox}(n\dist(o,x))\big)=\dist\big(\ell_{oy}(q\dist(0,y)),\ell_{oy}(-m\dist(0,y))\big)\\=|q+m|\dist(o,y)=|q+m|s\frac {n}m\dist(o,x)=|q+m|s\frac {p}q\dist(o,x).
\end{multline}
If $q=-m$, then $|p-n|=|q+m|s\frac{n}m=0$ and hence $p=n$, which implies a contradiction: $s=\frac pq=\frac{n}{-m}=-s$. This contradiction shows that $|q+m|\ne0$. The equality (\ref{eq:mult1}) implies  
$$\frac{|p-n|}{|q+m|}=s\frac {n}m=s\frac {p}q=c$$ for a unique positive real number $c$. It follows that $n=smc$ and $p=sqc$. Then  $$|q+m|c=|p-n|=|q-m|c\quad\mbox{and}\quad |q+m|=|m-q|.$$ Resolving this equality, we obtain $q+m=m-q$ or $q+m=q-m$. Consequently, $q=0$ or $m=0$, which contradicts the choice of $q$ or $m$. This contradiction shows that $\ell_{ox}(n\dist(o,x))\ne\ell_{oy}(-m\dist(0,y))$ and hence  $\ell_{ox}(n\dist(o,x))=\ell_{oy}(m\dist(0,y))$.
\end{proof}

\begin{proposition}\label{p:updown} For any points $o,x,y,z$ of a Banakh space $X$, the following conditions hold:
\begin{enumerate}
\item $ox\upuparrows ox$.
\item If $ox\upuparrows oy$, then $oy\upuparrows ox$.
\item If $ox\updownarrows oy$, then $oy\updownarrows ox$.
\item If $ox\upuparrows oy$ and $oy\upuparrows oz$, then $ox\upuparrows oz$.
\item If $ox\updownarrows oy$ and $oy\updownarrows oz$, then $ox\upuparrows oz$.
\item If $ox\upuparrows oy$ and $oy\updownarrows oz$, then $ox\updownarrows oz$.
\end{enumerate}
\end{proposition}

\begin{proof} 1,2,3. The first three statements follow immediately from the definition of the relations $\upuparrows$ and $\updownarrows$.
\smallskip

4.  Assuming that $ox\upuparrows oy$ and $oy\upuparrows oz$, we can find integer numbers $n,m,p,q\in\IZ$ such that $\ell_{ox}(n\dist(o,x))=\ell_{oy}(m\dist(o,y))$, $\ell_{oy}(p\dist(o,y))=\ell_{oz}(q\dist(o,z))$ and  $nm>0$, $pq>0$. Since $\ell_{ox}$ and $\ell_{oy}$ are isometries with $\ell_{ox}(0)=o=\ell_{oy}(0)$, we obtain 
$$|n|\dist(o,x)=\dist(\ell_{ox}(n\dist(o,x)),o)=\dist(\ell_{oy}(m\dist(o,y)),o)=|m|\dist(o,x).$$ Since $mn>0$, the equality $|n|\dist(o,x)=|m|\dist(o,y)$ implies $n\dist(o,x)=m\dist(o,y)$. By analogy we can show that $\ell_{oy}(p\dist(o,y))=\ell_{oz}(q\dist(o,z))$ and $pq>0$ imply $p\dist(o,y)=q\dist(o,z)$. It follows from $nm>0$ and $pq>0$ that $pnqm>0$.  The equalities $n\dist(o,x)=m\dist(o,y)$ and $p\dist(o,y)=q\dist(o,z)$ imply $pn\dist(o,x)=pm\dist(o,y)=qm\dist(oz)$. Lemma~\ref{l:upup}(1) ensures that 
$$\ell_{ox}(pn\dist(0,x))=\ell_{oy}(pm\dist(o,y))=\ell_{oz}(qm\dist(oz))$$ and hence $ox\upuparrows oz$ because $pnqm>0$.
\smallskip

5.  Assuming that $ox\updownarrows oy$ and $oy\updownarrows oz$, we can find integer numbers $n,m,p,q\in\IZ$ such that $\ell_{ox}(n\dist(o,x))=\ell_{oy}(m\dist(o,y))$, $\ell_{oy}(p\dist(o,y))=\ell_{oz}(q\dist(o,z))$ and  $nm<0$, $pq<0$. Since $\ell_{ox}$ and $\ell_{oy}$ are isometries with $\ell_{ox}(0)=o=\ell_{oy}(0)$, we obtain 
$$|n|\dist(o,x)=\dist(\ell_{ox}(n\dist(o,x)),o)=\dist(\ell(m\dist(o,y)),o)=|m|\dist(o,x).$$ Since $mn<0$, the equality $|n|\dist(o,x)=|m|\dist(o,y)$ implies $n\dist(o,x)=-m\dist(o,y)$. By analogy we can show that $\ell_{oy}(p\dist(o,y))=\ell_{oz}(q\dist(o,z))$ and $pq<0$ imply $p\dist(o,y)=-q\dist(o,z)$. It follows from $nm<0$ and $pq<0$ that $pnqm>0$.  The equalities $n\dist(o,x)=-m\dist(o,y)$ and $-p\dist(o,y)=q\dist(o,z)$ imply $pn\dist(o,x)=-pm\dist(o,y)=qm\dist(oz)$. Lemma~\ref{l:upup}(2) ensures that 
$$\ell_{ox}(pn\dist(0,x))=\ell_{oy}(pm\dist(o,y))=\ell_{oz}(qm\dist(oz))$$ and hence $ox\upuparrows oz$ because $pnqm>0$.
\smallskip 
 
6.  Assuming that $ox\upuparrows oy$ and $oy\updownarrows oz$, we can find integer numbers $n,m,p,q\in\IZ$ such that $\ell_{ox}(n\dist(o,x))=\ell_{oy}(m\dist(o,y))$, $\ell_{oy}(p\dist(o,y))=\ell_{oz}(q\dist(o,z))$ and  $nm>0$, $pq<0$. Since $\ell_{ox}$ and $\ell_{oy}$ are isometries with $\ell_{ox}(0)=o=\ell_{oy}(0)$, we obtain 
$$|n|\dist(o,x)=\dist(\ell_{ox}(n\dist(o,x)),o)=\dist(\ell(m\dist(o,y)),o)=|m|\dist(o,x).$$ Since $mn>0$, the equality $|n|\dist(o,x)=|m|\dist(o,y)$ implies $n\dist(o,x)=m\dist(o,y)$. By analogy we can show that $\ell_{oy}(p\dist(o,y))=\ell_{oz}(q\dist(o,z))$ and $pq<0$ imply $p\dist(o,y)=-q\dist(o,z)$. It follows from $nm>0$ and $pq<0$ that $pnqm<0$.  The equalities $n\dist(o,x)=m\dist(o,y)$ and $p\dist(o,y)=-q\dist(o,z)$ imply $pn\dist(o,x)=pm\dist(o,y)=-qm\dist(oz)$. Lemma~\ref{l:upup} ensures that 
$$\ell_{ox}(pn\dist(0,x))=\ell_{oy}(pm\dist(o,y))=\ell_{oz}(qm\dist(oz))$$ and hence $ox\updownarrows oz$ because $pnqm<0$.
\end{proof}


\begin{lemma}\label{l:inv} For any points $a,b$ of a Banakh space $X$ we have
$$\ell_{ab}(n\dist(a,b))=\ell_{ba}(-(n-1)\dist(a,b))$$for every $n\in\IZ$.
\end{lemma}

\begin{proof} Let $r\defeq\dist(a,b)$ and observe that the map $f:r\IZ\to r\IZ$, $f:x\mapsto r-x$, is an isometry and the composition $\ell_{ba}\circ f:r\IZ\to X$ is an isometric embedding with $\ell_{ba}\circ f(0)=\ell_{ba}(r)=a=\ell_{ab}(0)$ and 
$\ell_{ba}\circ f(r)=\ell_{ba}(0)=b=\ell_{ab}(r)$. Now the uniqueness of an isometric embedding $\ell_{ab}:r\IZ\to X$ implies that $\ell_{ab}=\ell_{ba}\circ f$. Then for every $n\in\IZ$ we have
$$\ell_{ab}(n\dist(a,b))=\ell_{ab}(nr)=\ell_{ba}\circ f(nr)=\ell_{ba}(r-nr)=\ell_{ba}(-(n-1)r)=\ell_{ba}(-(n-1)\dist(a,b)).$$
\end{proof}

The following theorem is the principal result of this section. 
 
\begin{theorem}\label{t:or}  For any points $o,x,y$ of a Banakh space $X$ the following conditions hold.
\begin{enumerate}
\item If $ox\upuparrows oy$ or $ox\updownarrows oy$, then $\dist(o,x)=\frac mn\dist(o,y)$ for some $m,n\in\IN$;
\item If $\dist(o,x)=\frac mn\dist(o,y)$ for some $m,n\in\IN$, then $ox\upuparrows  oy$ or $ox\updownarrows oy$.
\item If $ox\upuparrows oy$, $yo\upuparrows yx$, and $\dist(x,y)\in \IQ{\cdot}\dist(o,x)=\IQ{\cdot}\dist(o,y)$, then $xy\updownarrows xo$.
\item If $ox\updownarrows oy$, then $\dist(x,y)=\dist(x,o)+\dist(o,y)$.
\item If $ox\upuparrows oy$, then either $x=o=y$ or $\dist(x,y)<\dist(x,o)+\dist(o,y)$.
\item If $ox\upuparrows oy$ and $\dist(o,x)=\dist(o,y)$, then $x=y$.
\end{enumerate}
\end{theorem}

\begin{proof} 1. If $ox\upuparrows oy$ or $ox\updownarrows oy$, then $\ell_{ox}(n\dist(o,x))=\ell_{oy}(m\dist(o,y))$ for some nonzero numbers $n,m\in\IZ$. Then $|n|\dist(o,x)=\dist\big(\ell_{ox}(0),\ell_{ox}(n\dist(0,x))\big)=\dist\big(\ell_{oy}(0),\ell_{oy}(m\dist(o,y))\big)=|m|\dist(o,y)$ and hence $\dist(o,x)=\frac{|m|}{|n|}\dist(o,y)$.
\smallskip

2. Assume that $\dist(o,x)=\frac mn\dist(o,y)$ for some numbers $m,n\in\IN$. Then 
\begin{multline*}
\dist(o,\ell_{ox}(n\dist(o,x))=\dist\big(\ell_{ox}(0),\ell_{ox}(n\dist(o,x))\big)=n\dist(o,x)\\
=m\dist(o,y)=\dist\big(o,\ell_{oy}(m\dist(0,y))\big)=
\dist\big(o,\ell_{oy}(-m\dist(0,y))\big).
\end{multline*} Since $\dist\big(\ell_{oy}(m\dist(o,y)),\ell_{oy}(-m\dist(o,y))\big)=2m\dist(o,y)>0$, the Banakh property ensures that $$\ell_{ox}(n\dist(0,x))\in \Sf(o,m\dist(o,y))=\{\ell_{oy}(m\dist(o,y)),\ell_{oy}(-m\dist(o,y))\},$$ witnessing that $ox\upuparrows oy$ or $ox\updownarrows oy$.
\smallskip

3. Assume that $ox\upuparrows oy$, $yo\upuparrows yx$ and $\dist(x,y)\in \IQ{\cdot}\dist(o,x)=\IQ{\cdot}\dist(o,y)$.
If $x=y$, then the relations $yo\upuparrows yx$ and $ox\upuparrows oy$ imply $y=o=x$. Then $\ell_{xo}(\dist(x,0))=o=\ell_{xy}(-\dist(x,y))$ and hence $xo\updownarrows xy$.  So, we assume that $x\ne y$. In this case, our assumption  $\dist(x,y)\in \IQ{\cdot}\dist(o,x)=\IQ{\cdot}\dist(o,y)$ implies $\dist(o,x)\ne 0\ne\dist(o,y)$. 

By (already proved) Theorem~\ref{t:or}(2), 
$xy\upuparrows xo$ or $xy\updownarrows xo$. It remains to prove that the case $xy\upuparrows xo$ is impossible. To derive a contradiction, assume that $xy\upuparrows xo$. 

Since $0<\dist(x,y)\in \IQ{\cdot}\dist(o,x)=\IQ{\cdot}\dist(o,y)$, there exist numbers $k,m,n\in\IN$ such that $\dist(o,x)=\frac{k}{m}\dist(x,y)$, $\dist(o,y)=\frac{n}{m}\dist(x,y)$ and hence $\dist(x,o)=\frac kn \dist(o,y)$. Replacing the numbers $k,m,n$ by $pk,pm,pn$ for a sufficiently large number $p\in\IN$ we can additionally assume that $(2mn-1)\dist(o,x)>\dist(x,y)+\dist(o,y)$. 

By Lemma~\ref{l:upup}(2), the relations  $xy\upuparrows xo$, $yo\upuparrows yx$, $ox\upuparrows oy$, and the equalities $nm\dist(x,o)=nk\dist(x,y)$, $-km\dist(y,o)=-kn\dist(y,x)$, $mn\dist(o,x)=mk\dist(o,y)$, imply
$$
\begin{gathered}
\ell_{xo}(nm\dist(x,o))=\ell_{xy}(nk\dist(x,y)),\\
\ell_{yo}(-km\dist(y,o))=\ell_{yx}(-kn\dist(x,y)),\\
\ell_{ox}(mn\dist(o,x))=\ell_{oy}(mk\dist(o,y)).\\ 
\end{gathered}
$$
Applying Lemma~\ref{l:inv}, we obtain
$$
\begin{aligned}
(2mn-1)\dist(o,x)&=\dist\big(\ell_{ox}(-(mn-1)\dist(o,x)),\ell_{ox}(mn\dist(o,x))\big)\\
&=\dist\big(\ell_{xo}(mn\dist(o,x)),\ell_{oy}(mk\dist(o,y))\big)\\
&=\dist\big(\ell_{xy}(nk\dist(x,y)),\ell_{oy}(mk\dist(o,y))\big)\\
&=\dist\big(\ell_{yx}(-(kn-1)\dist(x,y)),\ell_{yo}(-(mk-1)\dist(o,y))\big)\\
&\le \dist\big(\ell_{yx}(-kn\dist(x,y)),\ell_{yo}(-mk\dist(o,y))\big)+\dist(x,y)+\dist(o,y)\\
&=0+\dist(x,y)+\dist(o,y),
\end{aligned}
$$
which contradicts the choice of the numbers $m,n$.
\smallskip

4. Assume that $ox\updownarrows oy$ and hence $\ell_{ox}(n\dist(o,x))=\ell_{oy}(-m\dist(o,y))$ for some $n,m\in\IN$, by Lemma~\ref{l:upup}(2). 
Then $n\dist(o,x)=\dist\big(\ell_{ox}(0),\ell_{ox}(n\dist(o,x))\big)=\dist\big(\ell_{oy}(0),\ell_{oy}(-m\dist(o,y))\big)=m\dist(o,y)$ and hence
$$
\begin{aligned}
&n\dist(o,x)+m\dist(o,y)=2m\dist(o,y)=\dist\big(\ell_{oy}(-m\dist(o,y)),\ell_{oy}(m\dist(o,y))\big)\\
&=\dist\big(\ell_{ox}(n\dist(o,x)),\ell_{oy}(m\dist(o,y))\big)\\
&\le \dist\big(\ell_{ox}(n\dist(o,x)),\ell_{ox}(\dist(o,x))\big)+
\dist\big(\ell_{ox}(\dist(o,x)),\ell_{oy}(\dist(o,y))\big)+\dist\big(\ell_{oy}(\dist(o,y)),\ell_{oy}(m\dist(o,y))\big)\\
&=(n-1)\dist(o,x)+\dist(x,y)+(m-1)\dist(o,y)
\end{aligned}
$$
 and hence $\dist(o,x)+\dist(o,y)\le \dist(x,y)\le \dist(x,o)+\dist(o,y)$ and finally, $\dist(x,y)=\dist(o,x)+\dist(o,y)$.
\smallskip

5. Assume that $ox\upuparrows oy$ and hence $\ell_{ox}(n\dist(o,x))=\ell_{oy}(m\dist(o,y))$ for some integer numbers $n,m$ with $nm>0$. If $x=o$, then $\ell_{oy}(m\dist(o,y))=\dist_{ox}(n\dist(o,x))=\dist_{ox}(0)=o$ and hence $m\dist(o,y)=0$ and $o=y$. By analogy we can prove that $o=y$ implies $o=x$. So, assume that $x\ne o\ne y$. To derive a contradiction, assume that $\dist(x,y)=\dist(x,o)+\dist(o,y)$. Let $y'\defeq\ell_{oy}(-\dist(o,y))$ and observe that $oy\updownarrows oy'$. By Proposition~\ref{p:updown}(6), $ox\upuparrows oy\updownarrows oy'$ implies $ox\updownarrows oy'$. By Theorem~\ref{t:or}(4), $\dist(x,y')=\dist(x,o)+\dist(o,y')=\dist(x,o)+\dist(o,y)=\dist(x,y)$. The Banakh property of $X$ ensures that $2\dist(o,y)=\dist(y,y')=2\dist(x,y)=2\dist(x,o)+2\dist(o,y)$ and hence $\dist(x,o)=0$, which contradicts our assumption. This contradiction shows that $\dist(x,y)\ne \dist(x,o)+\dist(o,y)$ and hence $\dist(x,y)<\dist(x,o)+\dist(o,y)$, by the triangle inequality. 
\smallskip

6. Assume that $ox\upuparrows oy$ and $\dist(o,x)=\dist(o,y)$. By the preceding statement, $\dist(x,y)<\dist(x,o)+\dist(o,y)=2\dist(o,x)$ and hence $x=y$, by the  Banakh property of $X$.
\end{proof}

\begin{proposition}\label{p:TE} Let $x,y,z$ be points of a Banakh space $X$. If $\{\dist(x,y),\dist(x,z),\dist(y,z)\}\subseteq r\IQ$ for some real number $r$, then
$$\dist(x,y)=\dist(x,z)+\dist(z,y)\;\vee\;\dist(x,z)=\dist(x,y)+\dist(y,z)\;\vee\;\dist(y,z)=\dist(y,x)+\dist(x,z)$$
and hence the subspace $\{x,y,z\}$ of $X$ is isometric to a subspace of the real line.
\end{proposition} 

\begin{proof} The conclusion of the proposition is true if $0\in\{\dist(x,y),\dist(x,z),\dist(y,z)\}$. So, we assume that the points $x,y,z$ are pairwise distinct. In this case Theorem~\ref{t:or}(2,3) implies that $(xy\updownarrows xz)\vee(yx\updownarrows yz)\vee(zx\updownarrows zy)$. Applying Theorem~\ref{t:or}(4), we obtain that $$\dist(x,y)=\dist(x,z)+\dist(z,y)\;\vee\;\dist(x,z)=\dist(x,y)+\dist(y,z)\;\vee\;\dist(y,z)=\dist(y,x)+\dist(x,z),$$
which implies that the subspace $\{x,y,z\}$ of $X$ is isometric to a subspace of the real line.
\end{proof}

\begin{proposition}\label{p:sum=>arrows} Let $x,y,z$ be points of a Banakh space $X$ such that $\{\dist(x,y),\dist(y,z)\}\subseteq\IQ r\setminus\{0\}$ for some real number $r$. If $\dist(x,z)=\dist(x,y)+\dist(y,z)$, then $$ yx\updownarrows yz,\quad
xy\upuparrows xz, \quad\mbox{and}\quad zy\upuparrows zx.$$
\end{proposition}

\begin{proof} It follows that $\dist(x,z)=\dist(x,y)+\dist(y,z)\in \IQ r\setminus\{0\}$. Theorem~\ref{t:or}(2,5) ensures that $xy\updownarrows xz$. By Theorem~\ref{t:or}(2), either $xy\upuparrows xz$ of $xy\updownarrows xz$. In the latter case, Theorem~\ref{t:or}(4) implies that a contradiction $$\dist(y,z)=\dist(y,x)+\dist(x,z)=\dist(y,x)+\dist(x,y)+\dist(y,z)>\dist(y,z),$$ showing that $xy\upuparrows xz$. By analogy we can prove that $zx\upuparrows zy$.
\end{proof}

\section{Constructing segments in Banakh spaces}\label{s:SC}

In this section we shall prove several results on construction of segments of given length in Banakh spaces.

The following theorem shows that Banakh spaces satisfy a ``rational'' version of the Axiom of Segment Construction, well-known in Foundations of Geometry \cite{Bolyai,BS,Greenberg,Hartshorne,Hilbert,SST}

\begin{theorem}\label{t:SC} For any points $x,y$ of a Banakh space $X$ and any real number $r\in \dist[X^2]\cap \IQ{\cdot}\dist(x,y)$, there exists a unique point $z\in X$ such that $\dist(y,z)=r$ and $\dist(x,z)=\dist(x,y)+\dist(y,z)$.
\end{theorem} 

\begin{proof} If $r=0$, then $z=y$ is a unique point of $X$ such that $\dist(y,z)=r=0$ and $\dist(x,z)=\dist(x,y)+0=\dist(x,y)+\dist(y,z)$.

So, we assume that $r>0$. In this case $r\in \IQ{\cdot}\dist(x,y)$ implies $\dist(x,y)>0$. Since $X$ is a Banakh space, there exist points $u,v\in X$ such that $\{u,v\}=\Sf(y;r)$ and $\dist(u,v)=2r$. It follows from $\ell_{yu}(r)=u=\ell_{yv}(-r)$ that $yu\updownarrows yv$. By Theorem~\ref{t:or}(2), $yx\upuparrows yu$ or $yx\updownarrows yu$. 

If $yx\updownarrows yu$, then $\dist(x,u)=\dist(x,y)+\dist(y,u)$ by Theorem~\ref{t:or}(4), and we can put $z\defeq u$.

If $yx\upuparrows yu$, then $yx\updownarrows yv$, by Propositions~\ref{p:updown}(6). Theorem~\ref{t:or}(4) ensures $\dist(x,v)=\dist(x,y)+\dist(y,v)$, and we can put $z\defeq v$.

The uniqueness of the point $z$ follows from Theorem~\ref{t:GPS}.
\end{proof}

Theorem~\ref{t:SC} implies a ``rational'' version of another property of Segment Construction, well-known in Geometry. 

\begin{corollary}\label{c:SC} For any points $x,z$ of a Banakh space $X$ and any real numbers $a,b\in \dist[X^2]\cap\IQ{\cdot}\dist(x,z)$ with $a+b=\dist(x,z)$, there exists a unique point $y\in X$ such that $\dist(x,y)=a$ and $\dist(y,z)=b$.
\end{corollary}

\begin{proof} We lose no generality assuming that $a\le b$. If $a=0$, then $b=\dist(x,z)-a=\dist(x,z)$ and $y\defeq x$ is a unique point of $X$ such that $\dist(x,y)=a$ and $\dist(y,z)=b$. So, we assume that $0<a\le b$. By Theorem~\ref{t:SC}, there exists a point $y'\in X$ such that $\dist(y',x)=a$ and $\dist(y',z)=\dist(y',x)+\dist(x,z)=2a+b$. Since $X$ is a Banakh space, there exists a unique point $y\in X$ such that $\dist(x,y)=a$ and $\dist(y',y)=2a$. By the triangle inequality, $\dist(y,z)\le \dist(y,x)+\dist(x,z)=2b+b$. 

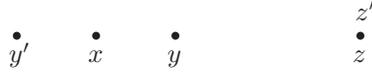
\begin{picture}(200,45)(-100,0)
\put(50,20){\circle*{3}}
\put(47,10){$x$}
\put(150,20){\circle*{3}}
\put(147,10){$z$}
\put(148,26){$z'$}
\put(20,20){\circle*{3}}
\put(17,10){$y'$}
\put(80,20){\circle*{3}}
\put(77,10){$y$}

\end{picture}

We claim that $\dist(y,z)<2a+b=\dist(y',z)$. In the opposite case, the Banakh property ensures that $2a=\dist(y,y')=2(2a+b)$, which is a contradiction showing that $\dist(y,z)<2a+b$. By Theorem~\ref{t:SC}, there exists a unique point $z'\in X$ such that $\dist(y,z')=b$ and $\dist(x,z')=\dist(x,y)+\dist(y,z')=a+b=\dist(x,z)$. Assuming that $z'\ne z$ and applying the Banakh property of $X$, we conclude that $\dist(z,z')=2(a+b)$. On the other hand, the triangle inequality ensures that $2(a+b)=\dist(z,z')\le\dist(z,y)+\dist(y,z')<(2a+b)+b=2(a+b)$, which is a contradiction showing that $z'=z$ and hence $y$ is a point with $\dist(x,y)=a$ and $\dist(y,z)=\dist(y,z')=b$. The uniqueness $y$ follows from Theorem~\ref{t:GPS}.
\end{proof}

We shall also need the following proposition on constructing a segement of a given length in a given direction.

\begin{proposition}\label{p:SC} For every points $x,y$ of a Banakh space $X$ and every positive real number $r\in\dist[X^2]\cap \IQ{\cdot}\dist(x,y)$, there exists a unique point $z\in \Sf(x;r)$ such that $xz\upuparrows xy$.
\end{proposition}

\begin{proof} It follows from $0<r\in\IQ{\cdot}\dist(x,y)$ that $x\ne y$. First we find a point 
$z\in\Sf(x;r)$ with $xz\upuparrows xy$ and then prove its uniqueness. To derive a contradiction, assume that no point $z\in \Sf(x;r)$ has $xz\upuparrows xy$. Applying Theorem~\ref{t:or}(2), we conclude that $xz\updownarrows xy$ for every $z\in \Sf(x;r)$. Since $X$ is a Banakh space, the sphere $\Sf(x;r)$ contains two points $u,v$ with $\dist(u,v)=2r$. Taking into account that $xu\updownarrows xy\updownarrows xv$ and applying Proposition~\ref{p:updown}(5), we conclude that $xu\upuparrows xv$. Applying Theorem~\ref{t:or}(5), we obtain a contradiction $2r=\dist(u,v)<\dist(u,o)+\dist(o,v)=2r$, showing that $xz\upuparrows xy$ for some $z\in \Sf(x;r)$.

To see that the point $z$ is unique, assume that $u\in\Sf(x;r)$ is another point such that $xu\upuparrows xy$. Then $xu\upuparrows xy\upuparrows xz$ implies $xu\upuparrows xz$, by Proposition~\ref{p:updown}(5). Theorem~\ref{t:or}(5) ensures that $\dist(u,z)<\dist(o,u)+\dist(o,z)=2r$ and hence $u=z$ by the Banakh property of $X$.
\end{proof}

\section{Monoids and half-groups in the real line}

In this section we study subsets of the closed half-line possessing certain algebraic properties. The obtained results will be applied for studying algebraic properties of the set of distances $\dist[X^2]$ of a Banakh space $X$.

For a subset $M$ of the real line and a positive real number $p$, consider the sets 
$$
\begin{gathered}
M+M\defeq\{x+y:x,y\in M\},\quad M-M=\{x-y:x,y\in M\},\\
\pm M\defeq M\cup(-M),\quad\tfrac1pM\defeq\{x\in \IR:px\in M\}
\end{gathered}
$$of the real line.

\begin{definition} A subset $M$ of the real line is called
\begin{itemize}
\item a {\em monoid} if $0\in M=M+M$;
\item a {\em group} if $0\in M=M-M$;
\item a {\em half-group} if $\pm M$ is a group;
\item {\em $p$-divisible in a subset} $S\subseteq \IR$ for a prime number $p$ if $S\cap\frac1p M\subseteq M$.
\end{itemize}
\end{definition}

The following lemma was suggested to the author by Pavlo Dzikovskyi.

\begin{lemma}\label{l:Dzik1} A monoid $M\subseteq\IZ_+$ is a half-group if $M$ is $p$-divisible in $\IZ_+$  for some prime number $p$.
\end{lemma}

\begin{proof} Assume for some prime number $p$, the monoid $M$ is $p$-divisible in $\IZ_+$. To show that the monoid $M$ is a half-group, it suffices to show that for every numbers $a<b$ in $M$ their difference $b-a$ belongs to $M$. 

For a nonzero integer number $x$, let 
$$\ld_p(x)\defeq\max\{n\in\IZ\setminus p\IZ:x\in n\IZ\}$$ be the greatest integer number that divides $x$ but is not divisible by $p$. 

For two numbers $x,y\in\IN$ let 
$$\lcd(x,y)\defeq\max\{n\in\IZ: \{x,y\}\subseteq n\IZ\}$$ be the greatest common divisor of $x,y$, and $$\lcd_p(x,y)\defeq\ld_p(\lcd(x,y))$$ be the greatest integer number that divides $x$ and $y$, but is not divisible by $p$. 

The following property of the function $\delta_p(x,y)$ is immediate.

\begin{claim}\label{cl:delta-p} For any positive integer numbers $x,y$ we have $\delta_p(x,y)=\delta_p(x,x+y)$.
\end{claim}

Given two integer numbers $x,y$, we write $x\equiv y\mod p$ iff $x-y\in p\IZ$. 
Since the quotient ring $\IZ/p\IZ=\{x+p\IZ:x\in\IZ\}$ is a field, for any numbers $x,y\in \IZ\setminus p\IZ$, there exists a unique positive integer number $k<p$  such that $ky\equiv x\mod p$. This unique number $k$ will be denoted by $x\div_p y$. So, $x\equiv (x\div_p y)y\mod p$.  

Consider the sequences of positive integer numbers $(a_k)_{k\in\w}$, $(b_k)_{k\in\w}$, $(x_k)_{k\in\w}$, $(y_k)_{k\in\w}$, defined by the recursive formulas:
$$
\begin{aligned}
&a_0=\delta_p(a),\quad b_0=\delta_p(b);\\
&x_k=\min\{a_k,b_k\},\quad y_k=\max\{a_k,b_k\};\\
&a_{k+1}=x_k,\quad b_{k+1}=\delta_p(y_k+(-y_k\div_p x_k) x_k).
\end{aligned}
$$

\begin{claim}\label{cl:Dzik1} For every $k\in\w$ the following conditions hold:
\begin{enumerate}
\item $\{a_k,b_k,x_k,y_k\}\subseteq M$;
\item $\delta_p(a_k,b_k)=\delta_p(a,b)$;
\item If $a_k\ne b_k$, then $a_{k+1}+b_{k+1}<a_k+b_k$.
\end{enumerate}
\end{claim}

\begin{proof} 1. The $p$-divisibility of $M$ in $\IZ_+$ implies that $\{\delta_p(x):x\in M\}\subseteq M$. Then $\{a_0,b_0,x_0,y_0\}=\{a_0,b_0\}=\{\delta_p(a),\delta_p(b)\}\subseteq M$. Assume that for some $k\in\w$ we know that $\{a_k,b_k,x_k,y_k\}\subseteq M$. Then $a_{k+1}=x_k\in M$ and $b_{k+1}=\delta_p(y_k+(-y_k\div_px_k)x_k)\in M$ because $M$ is a $p$-divisible monoid in $\IZ_+$. Finally,
$\{a_{k+1},b_{k+1},x_{k+1},y_{k+1}\}=\{a_{k+1},b_{k+1}\}\subseteq M$. 
\smallskip

2. Observe that $\delta_p(a_0,b_0)=\delta_p(\delta_p(a),\delta_p(b))=\delta_p(a,b)$. Assume that for some $k\in\w$ we know that $\delta_p(a_k,b_k)=\delta_p(a,b)$. Then 
$$\delta_p(x_k,y_k)=\delta_p(\min\{a_k,b_k\},\max\{a_k,b_k\})=\delta_p(a_k,b_k)=\delta(a,b)$$ and
$$\delta_p(a_{k+1},b_{k+1})=\delta_p(x_k,\delta_p(y_k+(-y_k\div_p x_k)x_k))=
\delta_p(x_k,y_k+(-y_k\div_p x_k)x_k)=
\delta_p(x_k,y_k)=\delta_p(a,b),$$
by Claim~\ref{cl:delta-p}.
\smallskip

3. If $a_k\ne b_k$ for some $k\in\w$, then $x_k=\min\{a_k,b_k\}<\max\{a_k,b_k\}=y_k$ and 
$$y_k+(-y_k\div_p x_k)x_k\le y_k+(p-1)x_k<y_k+(p-1)y_k=py_k.$$
Since the number $y_k+(-y_k\div_p x_k)x_k$ is divisible by $p$, $$b_{k+1}=\delta_p\big(y_k+(-y_k\div_p x_k)x_k\big)\le \frac{y_k+(-y_k\div x_k)x_k}p<\frac{py_k}p=y_k.$$ 
Then $a_{k+1}+b_{k+1}=x_k+b_{k+1}<x_k+y_k=a_k+b_k$.
\end{proof}

Since $(a_k+b_k)_{k\in\w}$ is a sequence of positive integer numbers, Claim~\ref{cl:Dzik1}(3) implies that there exists a number $k\in\w$ such that $a_k=b_k$ and hence $\delta_p(a,b)=\delta_p(a_k,b_k)=\delta_p(a_k)=a_k\in M$. 

Taking into account that $M$ is a monoid and the difference $b-a$ is divisible by $\delta_p(a,b)\in M$, we conclude that $b-a\in M$.
\end{proof}

\begin{theorem}\label{t:Dzik} For a submonoid $M$ of the monoid $\IQ_+$ the following conditions are equivalent:
\begin{enumerate}
\item $M$ is a half-group;
\item $M$ is $p$-divisible in $M-M$ for every prime number $p$;
\item $M$ is $p$-divisible in $M-M$ for some prime number $p$.
\end{enumerate}
\end{theorem}

\begin{proof} $(1)\Ra(2)$ If $M$ is a half-group, then $M\cup(-M)$ is a group, equal to $M-M$. Then for every prime number $p$ we have $(M-M)\cap\frac1p M\subseteq (M-M)\cap \frac1p\IQ_+=(M\cup(-M))\cap\IQ_+=M\cap \IQ_+\subseteq M$, which witnesses that $M$ is $p$-divisible in $M-M$.
\smallskip

The implication $(2)\Ra(3)$ is trivial.
\smallskip

$(3)\Ra(1)$. Assume that the monoid $M$ is $p$-divisible in $M-M$ for some prime number $p$. To show that $M$ is a half-group, it suffices to show that $b-a\in M$ for any numbers $a<b$ in $M\subseteq\IQ_+$. Since the numbers $a,b$ are rational, the group $a\IZ+b\IZ\subseteq \IQ$ is cyclic and hence $a\IZ+b\IZ=r\IZ$ for some rational number $r$. The $p$-divisibility of $M$ in $M-M$ implies the $p$-divisibility of $M\cap r\IZ$ in $r\IZ=a\IZ+b\IZ=a\IZ-b\IZ\subseteq (M\cap r\IZ)-(M\cap r\IZ)$. By Lemma~\ref{l:Dzik1}, the monoid $M\cap r\IZ$ is a half-group and hence $b-a\in M\cap r\IZ\subseteq M$.
\end{proof}

\section{Central symmetry in Banakh spaces}

In this section we characterize $\IQ$-spheres in Banakh spaces, which are isometric to subgroups of the group $\IQ$. The characterization involves the notion of a centrally symmetric space.

\begin{definition} A metric space $X$ is {\em centrally symmetric} if for any distinct points $c,x\in X$ there exists an isometry $f:X\to X$ such that $f(c)=c$ and $f(x)\ne x$. 
\end{definition}

\begin{example} Every subgroup $G$ of the real line is centrally symmetric. Indeed, for any distinct points $a,c\in X$, the isometry $f:X\to X$, $f:x\mapsto 2c-x$, has the required properies: $f(c)=c$ and $f(a)\ne a$.
\end{example}

The following important theorem was suggested to the author by Pavlo Dzikovskyi.
 
\begin{theorem}[Dzikovskyi]\label{t:Dzik2} If for some subgroup $G$ of $\IQ$, and some point $x$ of a Banakh space $X$, the metric space $\Sf(x;G)$ is centrally symmetric, then the monoid $\dist[X^2]\cap G$ is a half-group in $\IR$.
\end{theorem}

\begin{proof} Assume that the $G$-sphere $\Sf(x;G)$ is centrally symmetric.  Theorem~\ref{t:SC} implies that $M\defeq \dist[X^2]\cap G$ is a submonoid of $\IQ_+$ and hence $M-M$ is a group. We claim that the monoid $M$ is $2$-divisible in the group $M-M$. Given any points $a,b\in M$ with $2(b-a)\in M$, we should prove that $b-a\in M$. This is trivially true if $a=0$ or $a=b$. So, we assume that $0<a<b$.

Since $X$ is a Banakh space with $a\in\dist[X^2]$, there exists a point $u\in X$ such that $\dist(x,u)=a$. It follows from $\{a,2(b-a)\}\subseteq M=\dist[X^2]\cap G\subseteq \IQ$ that  $2(b-a)\in \IQ a$. By Theorem~\ref{t:SC}, there exists a point $v\in X$ such that $\dist(u,v)=2(b-a)\in \IQ$ and $\dist(x,v)=\dist(x,u)+\dist(u,v)=a+2(b-a)\in \IQ$. By Theorem~\ref{t:SC}, there exists a point $z\in X$ such that $\dist(v,z)=a$ and $\dist(x,z)=\dist(x,v)+\dist(v,z)=a+2(b-a)+a=2b$.
Since $X$ is a Banakh space and $b\in\dist[X^2]$, the sphere $\Sf(x;b)$ consists of two points $y,y'$ such that $xy\updownarrows xy'$. Since $2\dist(x,y)=2\dist(x,y')=2b=\dist(x,z)$, either $xy\upuparrows xz$ or $xy'\upuparrows xz$. We lose no generality assuming that $xy\upuparrows xz$. Then $\ell_{xy}(2b)=\ell_{xz}(2b)=z$ and hence $\dist(y,z)=b=\dist(x,y)$.

\begin{picture}(200,40)(-80,0)
\put(20,20){\circle*{3}}
\put(18,10){$y'$}
\linethickness{1pt}
\put(70,20){\color{blue}\line(1,0){20}}
\put(90,20){\color{red}\line(1,0){60}}
\put(70,20){\circle*{3}}
\put(68,10){$x$}
\put(90,20){\circle*{3}}
\put(88,10){$u$}
\put(120,20){\circle*{3}}
\put(118,10){$y$}
\put(150,20){\color{blue}\line(1,0){20}}
\put(150,20){\circle*{3}}
\put(148,10){$v$}
\put(170,20){\circle*{3}}
\put(168,10){$z$}
\end{picture}

\begin{claim}\label{cl:xu||xz} $xu\upuparrows xz$.
\end{claim}

\begin{proof} Since $\dist(x,u)=a\in\IQ$ and $\dist(x,z)=2b\in\IQ$, by Theorem~\ref{t:or}(2), $xu\upuparrows xz$ or $xu\updownarrows xz$. If $xu\updownarrows xz$, then $\dist(u,z)=\dist(u,x)+\dist(x,z)=a+2b$, by Theorem~\ref{t:or}(4). On the other hand, $a+2b=\dist(u,z)\le \dist(u,v)+\dist(v,z)=2(b-a)+a=2b-a$, which implies $a=0$ and contradicts our assumption. This contradiction shows that $xu\updownarrows xz$ does not hold and hence $xu\upuparrows xz$.
\end{proof}

\begin{claim}\label{cl:xu||xy} $xu\upuparrows xy$ and $\dist(y,u)<\dist(y,x)+\dist(x,u)$.
\end{claim}

\begin{proof} The choice of $y$ ensures that $xy\upuparrows xz$. By Claim~\ref{cl:xu||xz}, $xz\upuparrows xu$ and hence $xy\upuparrows xu$, by Proposition~\ref{p:updown}(4). Since $u\ne x\ne y$, Theorem~\ref{t:or}(5) ensures that $\dist(y,u)<\dist(y,x)+\dist(x,u)$.
\end{proof}

By analogy with Claim~\ref{cl:xu||xy}, we can prove

\begin{claim} $zv\upuparrows zy$.
\end{claim}

Since the metric space $\Sf(x;G)$ is centrally symmetric, there exists an isometry $f:\Sf(x;G)\to\Sf(x;G)$ such that $f(y)=y$ and $f(x)\ne x$. The Banakh property ensures that $\{x,f(x)\}=\Sf(y;\dist(y,x))=\Sf(y;b)=\{x,z\}$ and hence $f(x)=z$.

\begin{claim}\label{cl:u'v} $f(u)=v$.
\end{claim}

\begin{proof} 
Let $\tilde u\defeq f(u)$. Taking into account that $$2b-a=\dist(x,z)-\dist(x,u)\le \dist(u,z)\le\dist(u,v)+\dist(v,z)=2(b-a)+a=2b-a,$$ we conclude that $\dist(u,z)=2b-a$. Since $f$ is an isometry,   
$\dist(z,\tilde u)=\dist(f(x),f(u))=\dist(x,u)=a<a+2(b-a)=\dist(z,u)$, which implies  $\tilde u\ne u$.

Assuming that $\tilde u\ne v$ and taking into account that $\dist(z,\tilde u)=a=\dist(z,v)$, we conclude that $zv\updownarrows z\tilde u$. Taking into account that $zv\upuparrows zy$, we can apply Proposition~\ref{p:updown}(6) and conclude that $zy\updownarrows z\tilde u$. By Theorem~\ref{t:or}(4), $\dist(y,u)=\dist(y,\tilde u)=\dist(y,z)+\dist(z,\tilde u)=\dist(y,x)+\dist(u,x)$, which contradicts Claim~\ref{cl:xu||xy}. This contradiction shows that $f(u)=\tilde u=v$.
\end{proof}

By Claim~\ref{cl:u'v}, $f(u)=v$ and hence $\dist(y,u)=\dist(f(y),f(u))=\dist(y,v)$. Since $u\ne v$, the Banakh property of $X$ ensures that $2(b-a)=\dist(u,v)=2\dist(u,y)$ and hence $b-a=\dist(u,y)\in\dist[X^2]\cap G=M$. This shows that the monoid $M$ is $2$-divisible in $M-M$. By Theorem~\ref{t:Dzik}, the monoid $M=\dist[X^2]\cap G$ is a half-group.
\end{proof}

In the proof of Theorem~\ref{t:Q} we shall use the following classical result of Menger  \cite{Menger} (see also \cite{Bowers}, \cite{BBKR}).

\begin{theorem}[Menger]\label{t:Menger} A metric space $X$ of cardinality $|X|\ne 4$ is isometric to a subspaces of the real line if and only if every subspace $T\subseteq X$ of cardinality $|T|=3$ is isometric to a subspace of the real line.
\end{theorem}

The main result of this section is the following characterization.

\begin{theorem}\label{t:Q} For every print $x$ of a Banakh space $X$ and every $r\in \dist[X^2]$, the following conditions are equivalent:
\begin{enumerate}
\item The $\IQ r$-sphere $\Sf(x;\IQ r)$ is isometric to a subspace of the real line;
\item The $\IQ r$-sphere $\Sf(x;\IQ r)$ is isometric to a subgroup of the group $\IQ r$;
\item $\dist[\Sf(x;\IQ r)^2]\subseteq \IQ r$;
\item the metric space $\Sf(x;\IQ r)$ is centrally symmetric;
\item $\dist[X^2]\cap \IQ r$ is a half-group in $\IR$.
\end{enumerate}
\end{theorem}

\begin{proof} 
We shall prove that $(1)\Ra(2)\Ra(4)\Ra(5)\Ra(3)\Ra(1)$.
\smallskip

$(1)\Ra(2)$ Assume that $\Sf(x;G)$ is isometric to a subspace of the real line.
Then there exists an isometry $f:\Sf(x;G)\to \IR$ such that $f(x)=0$. For every $y\in\Sf(x;G)$, we have $|f(y)|=|f(y)-f(x)|=\dist(y,x)\in G$ and hence $f(y)\in G$. Therefore, $f[\Sf(x;G)]\subseteq G$. On the other hand, for every $y,z\in \Sf(x;G)$, 
we have $\{f(y),f(z)\}\subseteq f[\Sf(x;G)]\subseteq G$ and hence $\dist(y,z)=|f(y)-f(z)|\subseteq G_+$. Since $X$ is a Banakh space, there exist points $u,v\in X$ such that $\{u,v\}=\Sf(x;\dist(y,z))\subseteq \Sf(x;G)$ and $\dist(u,v)=2\dist(y,z)$. 
Then $\{f(u),f(v)\}=f[\Sf(x;\dist(y,z))]=\{f(x)-\dist(y,z),f(x)+\dist(y,z)\}=\{-|f(y)-f(z)|,|f(y)-f(z)|\}=\{f(y)-f(z),f(z)-f(y)\}$ and hence $\{f(x)-f(y),f(x)+f(y)\}\subseteq f[\Sf(x;G)]$, which means that $f[\Sf(x;G)]$ is a subgroup of the group $G$.
\smallskip

$(2)\Ra(4)$ If $\Sf(x;G)$ is isometric to a subgroup $H$ of the group $G$, then $\Sf(x;G)$ is centrally symmetric because the group $H\subseteq R$ is centrally symmetric (for any distinct elements $a,c\in H$, the function $f:H\to H$, $f:h\mapsto 2c-h$, has two properties $f(c)=c$ and $f(a)\ne a$ witnessing that the metric space $H$ is centrally symmetric).
\smallskip

$(4)\Ra(5)$ If the metric space $\Sf(x;G)$ is centrally symmetric, then $\dist[X^2]\cap G$ is a half-group by Theorem~\ref{t:Dzik2}.
\smallskip

$(5)\Ra(3)$ Assume that $\dist[X^2]\cap G$ is a half-group. Given two points $y,z\in\Sf(x;G)$, we should prove that $\dist(y,z)\in\IQ$. This is clear if $y=z$ or $x\in\{y,z\}$. So, we assume that $y\ne z$ and $x\notin\{y,z\}$. We lose no generality assuming that $\dist(x,y)\le\dist(x,z)$. By Theorem~\ref{t:or}(2), $xy\upuparrows xz$ or $xy\updownarrows xz$. If $xy\updownarrows xz$, then $\dist(y,z)=\dist(y,x)+\dist(x,z)\in G+G\subseteq\IQ$, by Theorem~\ref{t:or}(4). 

So, we assume that  $xy\upuparrows xz$. Since $\dist[X^2]\cap G$ is a half-group, $r\defeq \dist(x,z)-\dist(x,y)$ belongs to $\dist[X^2]\cap G$. By Theorem~\ref{t:SC}, there exists a unique point $z'\in X$ such that $\dist(y,z')=r$ and $\dist(x,z')=\dist(x,y)+\dist(y,z')=\dist(x,y)+r=\dist(x,z)$. We claim that $xz'\upuparrows xy$. In the opposite case, $r=\dist(z',y)=\dist(z',x)+\dist(x,y)=r+2\dist(x,y)$ and hence $x=y$, which contradicts our assumption. This contradiction shows that $xz'\upuparrows xy\upuparrows xz$ and hence $xz'\upuparrows xz$, by Proposition~\ref{p:updown}(4). Since $\dist(x,z')=\dist(x,z)$, the relation $xz'\upuparrows xz$ implies $z'=z$, by Theorem~\ref{t:or}(6). Then $\dist(z,y)=\dist(z',y)=r\in \dist[X^2]\cap G\subseteq\IQ$.
\smallskip

$(3)\Ra(1)$. Assume that $\dist[\Sf(x;G)^2]\subseteq\IQ$. If $\Sf(x;G)=\{x\}$, then the $G$-sphere $\Sf(x;G)$ is isometric to the subspace $\{0\}$ of the real line. So, we assume that $\Sf(x;G)\ne \{x\}$ and hence $\dist[X^2]\cap G\ne\{0\}$. Theorem~\ref{t:SC} implies that $\dist[X^2]\cap G$ is a submonoid of $\IR$. Then $\dist[X^2]\cap G$ is an inifinite set and so is the $G$-sphere $\Sf(x;G)$. Now we prove that every $3$-element subset $\{a,b,c\}\subseteq\Sf(x;G)$ is isometric to a subspace of the real line. Since 
$$\{\dist(a,b),\dist(a,c),\dist(b,c)\}\subseteq\dist[\Sf(x;G)^2]\subseteq\IQ,$$
the set $\{a,b,c\}$ is isometric to a subspace of the real line, according to Proposition~\ref{p:TE}. Therefore, $\Sf(x;G)$ is an infinite metric space whose every 3-element subspace is isometric to a subspace of the real line. By Menger's Theorem~\ref{t:Menger}, the metric space $\Sf(x;G)$ is isometric to a subspace of the real line.
\end{proof}

Theorem~\ref{t:Q} implies the following corollary that nicely complements Theorem~\ref{t:main3}.

\begin{corollary}\label{c:Q} Every nonempty Banakh space $X$ with $\dist[X^2]\subseteq\IQ$ is isometric to a subgroup of the group $\IQ$.
\end{corollary}

\section{$\IQ r$-spheres and $\IQ r$-hyperspheres in Banakh spaces}\label{s:Spme}

In general, $\IQ r$-spheres in a Banakh spaces are not isometric to subgroups of $\IQ r$. To describe the metric structure of $\IQ r$-spheres in Banakh spaces, for every point $x$ of a Banakh space $X$, let us consider the set
$$\Sf^\pm(x;\IQ r)\defeq\bigcup_{y\in\Sf(x;\IQ r)}\Sf(y;\IQ r),$$
called the {\em $\IQ r$-hypersphere} around the point $x$ in the Banakh space $X$. 

Theorem~\ref{t:Sf-rZ} implies that for every Banakh space $X$ and every $r\in \dist[X^2]$, the family $\{\Sf(x;\IZ r):x\in X\}$ is a partition of $X$ into pairwise disjoint $\IZ r$-spheres. A similar fact holds for $\IQ r$-hyperspheres in a Banakh space: they form a partition of the Banakh space.  
 We shall deduce this fact  from the following lemma.

\begin{lemma}\label{l:Spme} Let $r$ be a positive real number and $a$ be a point of a Banakh space $X$. For every point $c\in \Sf^\pm(a;\IQ r)$ we have $\Sf(c;\IQ r)\subseteq \Sf^\pm(a;\IQ r)$.
\end{lemma}

\begin{proof} If $c\in \Sf(a;\IQ r)$, then $\Sf(c;\IQ r)\subseteq\Sf^\pm(a;\IQ r)$, by the definition of $\Sf^\pm(a;\IQ r)$. So, we assume that $c\notin\Sf(a;\IQ r)$. By the definition of $\Sf^\pm(a;\IQ)\ni c$, there exists $b\in\Sf(a;\IQ r)$ such that $c\in\Sf(b;\IQ r)$. It follows from $c\notin\Sf(a;\IQ)$ that $a\ne b\ne c$. By Theorem~\ref{t:or}(2), $ba\upuparrows bc$ or $ba\updownarrows bc$. In the latter case, $\dist(a,c)=\dist(a,b)+\dist(b,c)\in\IQ r$ by Theorem~\ref{t:or}(4).  Then $c\in\Sf(a;\IQ r)$, which contradicts our assumption. This contradiction shows that $ba\upuparrows bc$. 

Given any $x\in\Sf(c;\IQ r)\setminus\{c\}$, we should prove that $x\in\Sf^\pm(a;\IQ r)$. By Theorem~\ref{t:or}(2), $cx\upuparrows cb$ or $cx\updownarrows cb$. In the latter case, Theorem~\ref{t:or}(4) guarantees that $\dist(x,b)=\dist(x,c)+\dist(c,b)\in\IQ r$  and hence $x\in\Sf(b;\IQ r)\subseteq\Sf^\pm(a;\IQ r)$. So, we assume that $cx\upuparrows cb$.

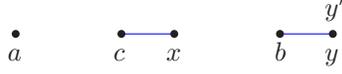
\begin{picture}(200,40)(-120,0)

\put(120,20){\color{blue}\line(1,0){20}}
\put(60,20){\color{blue}\line(1,0){20}}
\put(20,20){\circle*{3}}
\put(17,10){$a$}
\put(120,20){\circle*{3}}
\put(118,10){$b$}

\put(60,20){\circle*{3}}
\put(57,10){$c$}
\put(80,20){\circle*{3}}
\put(77,10){$x$}
\put(140,20){\circle*{3}}
\put(137,10){$y$}
\put(137,27){$y'$}

\end{picture}

By Theorem~\ref{t:SC}, there exists a point $y\in X$ such that $\dist(b,y)=\dist(c,x)>0$ and $$\dist(c,y)=\dist(c,b)+\dist(b,y)=\dist(c,b)+\dist(c,x)\in\IQ r\setminus\{0\}.$$ By Proposition~\ref{p:sum=>arrows}, $cy\upuparrows cb$ and $bc\updownarrows by$. 

 By Theorem~\ref{t:SC}, there exists a point $y'\in X$ such that $\dist(x,y')=\dist(c,b)>0$ and $$\dist(c,y')=\dist(c,x)+\dist(x,y)=\dist(c,x)+\dist(c,b)=\dist(c,y)\in\IQ r\setminus\{0\}.$$ By Proposition~\ref{p:sum=>arrows}, $cy'\upuparrows cx$. Then $cy'\upuparrows cx\upuparrows cb\upuparrows cy$ and hence $y=y'$, by Proposition~\ref{p:updown}(4) and Theorem~\ref{t:or}(6).
 
By Proposition~\ref{p:updown}(6), $ba\upuparrows bc\updownarrows by$ implies $ba\updownarrows by$ and then Theorem~\ref{t:or}(4) implies $\dist(a,y)=\dist(a,b)+\dist(b,y)\in \IQ r$.  Then $y'=y\in\Sf(a;\IQ r)$ and $x\in \Sf(y';\dist(b,c))\subseteq \Sf(y';\IQ r)=\Sf^\pm(a;\IQ r)$.
\end{proof}

\begin{theorem}\label{t:Spme} For any real number $r$ and points $a,b$ of a Banakh space $X$, either $$\Sf^\pm(a;\IQ r)\cap\Sf^\pm(b;\IQ r)=\emptyset  \quad\mbox{or}\quad \Sf^\pm(a;\IQ r)=\Sf^\pm(b;\IQ r).$$
\end{theorem}

\begin{proof} Assuming that $\Sf^\pm(a;\IQ r)\cap\Sf^\pm(b;\IQ r)\ne\emptyset$, we shall prove that $\Sf^\pm(a;\IQ r)=\Sf^\pm(b;\IQ r)$. Fix any point $c\in \Sf^\pm(a;\IQ r)\cap\Sf^\pm(b;\IQ)$. To show that $\Sf^\pm(b;\IQ r)\subseteq\Sf^\pm(a;\IQ r)$, take any point $x\in \Sf^\pm(a;\IQ r)$. By the definition of the set $\Sf^\pm (b;\IQ r)$, there exists $y\in\Sf(b;\IQ r)$ such that $x\in\Sf(y;\IQ r)$. Since $c\in \Sf^\pm(b;\IQ r)$, there exists $z\in\Sf(b;\IQ r)$ such that $c\in\Sf(z;\IQ r)$. Applying Lemma~\ref{l:Spme}, we conclude that $c\in\Sf^\pm(a;\IQ r)$, implies $z\in \Sf(c;\IQ r)\subseteq\Sf^\pm(a;\IQ)$, which implies $b\in \Sf(z;\IQ r)\subseteq \Sf^\pm(a;\IQ)$, which implies $y\in\Sf(b;\IQ r)\subseteq\Sf^\pm(a;\IQ)$, which implies $x\in\Sf(y;\IQ r)\subseteq \Sf^\pm(a;\IQ r)$. Therefore, $\Sf^\pm(b;\IQ r)\subseteq\Sf^\pm(a;\IQ r)$. By analogy we can prove that $\Sf^\pm(a;\IQ r)\subseteq \Sf^\pm(b;\IQ r)$.
\end{proof}

The following description of the structure of $\IQ r$-spheres and $\IQ r$-hyperspheres in the main result of this section.

\begin{theorem}\label{t:Spm} Let $a,b$ be two points of a Banakh space $X$, $r\defeq\dist(a,b)$ and $M_r\defeq\dist[X^2]\cap\IQ r$. For the subgroup $M_r-M_r$ of\/ $\IQ r$, there exists a unique bijective function $\ell^\pm_{ab}:M_r-M_r\to \Sf^\pm(a;\IQ r)$ such that $\ell^\pm_{ab}(0)=a$, $\ell^\pm_{ab}(r)=b$, $\ell^\pm_{ab}[\pm M_r]=\Sf(a;\IQ r)$ and for every $s,t\in M_r-M_r$ the following statements hold:
\begin{enumerate}
\item $|s-t|\le \dist(\ell^\pm_{ab}(s),\ell^\pm_{ab}(t))\le \inf\{u+v:u,v\in M_r,\;|t-s|=u-v\}$;
\item $\dist(\ell^\pm_{ab}(s),\ell^\pm_{ab}(t))=|s-t|$ if and only if $|s-t|\in M_r$ if and only if $\dist(\ell^\pm_{ab}(s),\ell^\pm_{ab}(yt))\in \IQ r$.
\end{enumerate}
\end{theorem}
 
\begin{proof} Given any real number $t\in \IQ r$, consider the set $$Y_t\defeq\big\{y\in X:\exists x\in X\;\big((ax\upuparrows ab)\wedge (xy\upuparrows xa)\wedge(\dist(a,x)-\dist(x,y)=t)\big)\big\}.$$

\begin{lemma} For every $t\in \IQ r$ we have $Y_t\subseteq\Sf^\pm(a;\IQ r)$.
\end{lemma}

\begin{proof} By the definition of the set $Y_t$, for every point $y\in T_t$,  there exists a point $x\in  X$ such that $$(ax\upuparrows ab)\wedge (xy\upuparrows xa)\wedge(\dist(a,x)-\dist(x,y)=t).$$
By Theorem~\ref{t:or}(1), $ax\upuparrows ab$ implies $0\ne \dist(a,x)\in\IQ\dist(a,b)=\IQ r$ and $xy\upuparrows xa$ implies $0\ne\dist(x,y)\in\IQ\dist(x,a)=\IQ r$. Then $x\in\Sf(a;\IQ r)$ and $y\in \Sf(x;\IQ r)\subseteq \Sf^\pm(a;\IQ r)$.
\end{proof}

\begin{lemma}\label{l:Yt2} For every $t\in \IQ r$, the set $Y_t$ contains at most one point.
\end{lemma}

\begin{proof} To derive a contradiction, assume that the set $Y_t$ contains two distinct points $\hat y,\check y\in Y_t$. By the definition of the set $Y_t$, for the points $\hat y,\check y$ there exist points $\hat x,\check x\in X$ such that $a\hat x\upuparrows ab\upuparrows a\check x$, $\hat x\hat y\upuparrows\hat xa$, $\check x\check y\upuparrows \check xa$, and  $\dist(a,\hat x)-\dist(\hat x,\hat y)=t=\dist(a,\check x)-\dist(\check x,\check y)$. By Theorem~\ref{t:or}(1), $$\{\dist(a,\hat x),\dist(\hat x,\hat y),\dist(a,\check x),\dist(\check x,\check y)\}\subseteq\IQ r\setminus\{0\}.$$ 

By Theorem~\ref{t:SC}, there exists a point $x\in X$ such that $\dist(\hat x,x)=\dist(a,\check x)$ and $$\dist(a,x)=\dist(a,\hat x)+\dist(\hat x,x)=\dist(a,\hat x)+\dist(a,\check x)\in\IQ r\setminus\{0\}.$$ By Proposition~\ref{p:sum=>arrows}, $\hat xa\updownarrows \hat x x$, $ax\upuparrows a\hat x\upuparrows ab$ and $xa\upuparrows x\hat x$. By Proposition~\ref{p:updown}(6), $\hat x\hat y\upuparrows\hat x a\updownarrows \hat x x$ implies $\hat x\hat y\updownarrows \hat x x$ and hence $\dist(x,\hat y)=\dist(x,\hat x)+\dist(\hat x,\hat y)$, by Theorem~\ref{t:or}(4). By Proposition~\ref{p:sum=>arrows}, $x\hat y\upuparrows x\hat x$.

By Theorem~\ref{t:SC}, there exists a point $x'\in X$ such that $\dist(\check x,x')=\dist(a,\hat x)$ and $$\dist(a,x')=\dist(a,\check x)+\dist(\check x,x')=\dist(a,\check x)+\dist(a,\hat x)\in\IQ r\setminus\{0\}.$$ Applying Propositions~\ref{p:sum=>arrows}, \ref{p:updown} and Theorem~\ref{t:or}, we can show that $ax'\upuparrows a\check x\upuparrows ab$, $\check x\check y\updownarrows \check xx'$, $\dist(x',\check y)=\dist(x',\check x)+\dist(\check x,\check y)=\dist(a,\hat x)+\dist(\check x,\check y)$, and $x'\check y\upuparrows x'\check x\upuparrows x'a$.

Then $ax\upuparrows a\hat x\upuparrows ab\upuparrows a\check x\upuparrows ax'$ and $$\dist(a,x)=\dist(a,\hat x)+\dist(\hat x,x)=\dist(a,\hat x)+\dist(a,\check x)=\dist(a,\check x)+\dist(\check x,x')=\dist(a,x'),$$ and hence $x=x'$, by Theorem~\ref{t:or}(6).

By the choice of $\hat x$ and $\check x$, we have $\dist(a,\hat x)-\dist(\hat x,\hat y)=t=\dist(a,\check x)-\dist(\check x,\check y)$ and hence $$\dist(a,\check x)+\dist(\hat x,\hat y)=\dist(a,\hat x)+\dist(\check x,\check y).$$
Since $x\hat y\upuparrows x\hat x\upuparrows xa=x'a\upuparrows x'\check y=x\check y$ and $$\dist(x,\hat y)=\dist(x,\hat x)+\dist(\hat x,\hat y)=\dist(a,\check x)+\dist(\hat x,\hat y)=\dist(a,\hat x)+\dist(\check x,\check y)=\dist(x',\check x)+\dist(\check x,\check y)=\dist(x',\check y)=\dist(x,\check y),$$ we can apply Theorem~\ref{t:or}(6) and conclude that $\hat y=\check y$, which contradicts the choice of the points $\hat y$ and $\check y$.
\end{proof}

\begin{lemma}\label{l:Yt1} For every $t\in M_r-M_r$, the set $Y_r$ is a singleton.
\end{lemma}

\begin{proof} Theorem~\ref{t:SC} implies that $M_r$ is a monoid and hence $u+r\in M_r$ for all $u\in M_r$. Since $t\in M_r-M_r$ there exist real numbers $u,v\in M_r$ such that $t=u-v$. We lose no generality assuming that the numbers $u,v$ are positive (in the opposite case, replace the numbers $u,v$ by the positive numbers $u+r,v+r$). By Proposition~\ref{p:SC}, there exist a point $x\in \Sf(a;u)$ such that $ax\upuparrows ab$, and a point $y\in\Sf(x,v)$ such that $xy\upuparrows xa$. Then  $y\in Y_t$, by the definition of the set $Y_t$. By Lemma~\ref{l:Yt2}, $Y_t=\{y\}$ is a singleton.
\end{proof}

Let $\ell_{ab}^\pm:M_r-M_r\to \Sf^\pm(s;\IQ r)$ be the function assigning to every $t\in M_r-M_r$ the unique point of the singleton $Y_t$. It the following claims we show that the function $\ell_{ab}^\pm$ has the properties required in Theorem~\ref{t:Spm}.

\begin{claim}\label{cl:ell0a} $\ell_{ab}^\pm(0)=a$.
\end{claim}

\begin{proof} The points $x\defeq b$ and $y\defeq a$ with $ax\upuparrows ab$, $xy\upuparrows xa$ and $\dist(a,x)-\dist(x,y)=r-r=0$ witness that $a=y\in Y_0$ and hence $\ell_{ab}^\pm(0)=a$.
\end{proof}

\begin{claim} $\ell_{ab}^\pm(r)=b$.
\end{claim}

\begin{proof} By Theorem~\ref{t:SC}, there exists a point $x\in X$ such that $\dist(b,x)=r$ and $\dist(a,x)=\dist(a,b)+\dist(b,x)=2r$. By Proposition~\ref{p:sum=>arrows}, $ax\upuparrows ab$ and $xb\upuparrows xa$. Since $\dist(a,x)-\dist(x,b)=2r-r=r$, the point $b$ belongs to the set $Y_r$ and hence $\ell_{ab}^\pm(r)=b$.
\end{proof}

\begin{claim}\label{cl:10.5} For any $s,t\in M_r-M_r$ and points $y_s\defeq\ell_{ab}^\pm(s)$ and $y_t\defeq\ell_{ab}^\pm(t)$, there exists a point $x\in X$ such that
\begin{enumerate}
\item $ax\upuparrows ab$;
\item $xy_t\upuparrows xa\upuparrows xy_s$;
\item $t=\dist(a,x)-\dist(a,y_t)$ and $s=\dist(a,x)-\dist(x,y_s)$.
\end{enumerate}
\end{claim}

\begin{proof} By the definition of the sets $Y_t\ni y_t$ and $Y_s\ni y_s$, there exist points $x_t,x_s\in X$ such that 
$ax_t\upuparrows ab\upuparrows ax_s$, $x_ty_t\upuparrows x_ta$, $x_sy_s\upuparrows x_sa$, $s=\dist(a,x_s)-\dist(x_s,y_s)$ and $t=\dist(a,x_t)-\dist(x_t,y_t)$.
By Theorem~\ref{t:or}(1), $$\{\dist(a,x_t),\dist(a,x_s),\dist(x_t,y_t),\dist(x_s,y_s)\}\subseteq \IQ{\cdot}\dist(a,b)\setminus\{0\}=\IQ r\setminus\{0\}.$$
By Theorem~\ref{t:SC}, there exists a point $x\in X$ such that $\dist(x_t,x)=\dist(a,x_s)$ and $$\dist(a,x)=\dist(a,x_t)+\dist(x_t,x)=\dist(a,x_t)+\dist(a,x_s)\in\IQ r\setminus\{0\}.$$ 
By Proposition~\ref{p:sum=>arrows}, $ax\upuparrows ax_t\upuparrows ab$ and $xx_t\upuparrows xa$.

By analogy, we can find a point $x'\in X$ such that $\dist(x_s,x')=\dist(a,x_t)$, $$\dist(a,x')=\dist(a,x_s)+\dist(x_s,x')=\dist(a,x_s)+\dist(a,x_t)=\dist(a,x)$$ and $ax'\upuparrows ab\upuparrows ax$. By Proposition~\ref{p:updown}(4) and Theorem~\ref{t:or}(6), $x=x'$. 

We claim that the point $x=x'$ satisfies the conditions (1)--(3) of Claim~\ref{cl:10.5}. The condition (1) holds by the choice of $x$. 

By Proposition~\ref{p:sum=>arrows}, the equality $\dist(a,x)=\dist(a,x_t)+\dist(x_t,x)$ implies $x_ta\updownarrows x_tx$ and $xx_t\upuparrows xa$. By Proposition~\ref{p:updown}(5), $x_ty_t\upuparrows x_ta\updownarrows x_tx$ implies $x_ty_t\updownarrows x_tx$. Applying Theorem~\ref{t:or}(4), we conclude that $\dist(x,y_t)=\dist(x,x_t)+\dist(x_t,y_t)$. By Proposition~\ref{p:sum=>arrows}, $xy_t\upuparrows xx_t\upuparrows xa$.

By analogy we can prove that $\dist(x,y_s)=\dist(x,x_s)+\dist(x_s,y_s)$ and $xy_s\upuparrows xy_t\upuparrows xa$, witnessing that the condition (2) holds.

Finally, observe that
$$t=\dist(a,x_t)-\dist(x_t,y_t)=(\dist(a,x_t)+\dist(x_t,x))-(\dist(x,x_t)+\dist(x_t,y_t))=\dist(a,x)-\dist(x,y_t).$$
By analogy we can show that $s=\dist(a,x)-\dist(x,y_s)$.
\end{proof}



\begin{claim} For every $s,t\in M_r-M_r$, we have
$$\dist(\ell_{ab}^\pm(s),\ell_{ab}^\pm(t))\ge|s-t|.$$
\end{claim} 

\begin{proof}  By Claim~\ref{cl:10.5}, for the points $y_t\defeq\ell_{ab}^\pm(t)$ and $y_s\defeq\ell_{ab}^\pm(s)$, there exists a point $x\in X$ such that $ax\upuparrows ab$, $xy_t\upuparrows xa\upuparrows xy_s$, $t=\dist(a,x)-\dist(x,y_t)$ and $s=\dist(a,x)-\dist(x,y_s)$. By the triangle inequality,
$$\dist(\ell_{ab}^\pm(s),\ell_{ab}^\pm(t))=\dist(y_s,y_t)\ge| \dist(x,y_s)-\dist(x,y_t)|=|(\dist(a,x)-\dist(x,y_t))-(\dist(a,x)-\dist(x,y_s))|=|t-s|.$$
\end{proof}

\begin{claim}\label{cl:ell-eq} For any real numbers $s,t\in M_r-M_r$ the following conditions are equivalent:
\begin{enumerate}
\item $|s-t|\in M_r$;
\item $\dist(\ell_{ab}^\pm(s),\ell_{ab}^\pm(t))=|s-t|$;
\item  $\dist(\ell_{ab}^\pm(s),\ell_{ab}^\pm(t))\in\IQ r$.
\end{enumerate}
\end{claim}

\begin{proof} We lose no generality assuming that $s<t$. By Claim~\ref{cl:10.5}, for the points $y_t\defeq\ell_{ab}^\pm(t)$ and $y_s\defeq\ell_{ab}^\pm(s)$, there exists a point $x\in X$ such that $ax\upuparrows ab$, $xy_t\upuparrows xa\upuparrows xy_s$, $t=\dist(a,x)-\dist(x,y_t)$ and $s=\dist(a,x)-\dist(x,y_s)$. Then $\dist(x,y_t)=\dist(a,x)-t<\dist(a,x)-s=\dist(x,y_s)$. By Theorem~\ref{t:or}(1), the relations $xy_t\upuparrows xa\upuparrows xy_s$ imply
$$\{\dist(a,x),\dist(x,y_t),\dist(x,y_s)\}\subseteq\IQ{\cdot}\dist(a,b)\setminus\{0\}=\IQ r\setminus\{0\}.$$
\smallskip

$(1)\Ra(2)$ If $|s-t|\in M_r$, then $t-s=|s-t|\in M_r$. By Theorem~\ref{t:SC}, there exists a point $y'_s\in X$ such that $\dist(y_t,y_s')=t-s$ and $\dist(x,y'_s)=\dist(x,y_t)+\dist(y_t,y'_s)=\dist(a,x)-t+t-s=\dist(a,x)-s=\dist(x,y_s)$. By Proposition~\ref{p:sum=>arrows}, $xy'_s\upuparrows xy_t\upuparrows xa\upuparrows xy_s$ and by Theorem~\ref{t:or}(6), $y_s=y_s'$, and finally
$\dist(\ell_{ab}^\pm(s),\ell_{ab}^\pm(t))=\dist(y_s,y_t)=\dist(y_s',y_t)=t-s=|s-t|$.
\smallskip

$(2)\Ra(3)$ If $\dist(\ell_{ab}^\pm(s),\ell_{ab}^\pm(t))=|s-t|$, then 
$$ \dist(\ell_{ab}^\pm(s),\ell_{ab}^\pm(t))=t-s\subseteq (M_r-M_r)-(M_r-M_r)\subseteq (\IQ r-\IQ r)-(\IQ r-\IQ r)=\IQ r.$$

$(3)\Ra(1)$ If  $\dist(\ell_{ab}^\pm(s),\ell_{ab}^\pm(t))\in\IQ r$, then $\{\dist(y_t,y_s),\dist(x,y_t),\dist(x,y_s)\}\subseteq \IQ r$. Taking into account that $xy_s\upuparrows xy_t$ and $\dist(x,y_t)<\dist(x,y_s)$ and applying Proposition~\ref{p:TE}, we conclude that $\dist(x,y_s)=\dist(x,y_t)+\dist(y_s,y_t)$ and hence
$$\dist(y_s,y_t)=\dist(x,y_s)-\dist(x,y_t)=(\dist(a,x)-s)-(\dist(a,x)-t)=t-s=|t-s|.$$
\end{proof}

\begin{claim} For every $s,t\in M_r-M_r$, we have
$$\dist(\ell_{ab}^\pm(s),\ell_ab^\pm(t))\le\inf\{u+v:(u,v\in M_r)\;\wedge (|s-t|=u-v)\}.$$
\end{claim}

\begin{proof} Given any $u,v\in M_r$ with $|s-t|=u-v$, it suffices to prove that $\dist(\ell_{ab}^\pm(t)-\ell_{ab}^\pm(s))\le u+v$. We lose no generality assuming that $s<t$. Consider the real number $T=s+u$ and observe that $T-t=s+u-t=u-(t-s)=u-|t-s|=v\in M_r$. By  the triangle inequality and Claim~\ref{cl:ell-eq}, we have
$$\dist(\ell_{ab}^\pm(s),\ell_{ab}^\pm(t))\le\dist(\ell_{ab}^\pm(s),\ell_{ab}^\pm(T))+\dist(\ell_{ab}^\pm(T),\ell_{ab}^\pm(t))=|T-s|+|T-t|=u+v.$$
\end{proof}

\begin{claim} The function $\ell_{ab}^\pm:M_r-M_r\to \Sf^\pm(a;\IQ r)$ is bijective.
\end{claim}

\begin{proof} The injectivity of the function $\ell_{ab}^\pm$ follows from Claim~\ref{cl:10.5}. To see that the function $\ell_{ab}^\pm:M_r-M_r\to\Sf^\pm(a;\IQ r)$ is surjective, take any point $y\in \Sf^\pm(a;\IQ r)$. We have to find a real number $t\in M_r-M_r$ such that $\ell_{ab}^\pm(t)=y$. If $y=a$, then for $t=0$ we have $\ell_{ab}^\pm(t)=\ell_{ab}^\pm(0)=a=y$, by Claim~\ref{cl:ell0a}.

So, we assume that $y\ne a$. Two cases are possible.
\smallskip

1. First assume that $y\in\Sf(a;\IQ r)$. By Theorem~\ref{t:or}(2), either $ay\upuparrows ab$ or $ay\updownarrows ab$.

1.1. In case $ay\upuparrows ab$, we shall show that $y=\ell_{ab}^\pm(t)$ for $t\defeq\dist(a,y)$. By Theorem~\ref{t:SC}, there exists a point $x\in X$ such that $\dist(y,x)=r$ and $\dist(a,x)=\dist(a,y)+\dist(y,z)=t+r$ and hence $t=\dist(a,x)-r=\dist(a,x)-\dist(x,y)$. By Proposition~\ref{p:sum=>arrows}, $ax\upuparrows ay\upuparrows ab$ and $xy\upuparrows xa$. The point $x$  witness that $y\in Y_t$ and hence $y=\ell_{ab}^\pm(t)$.

1.2. Next, assume that $ay\updownarrows ab$.  In this case we shall show that $y=\ell_{ab}^\pm(t)$ for the real number $t\defeq -\dist(a,y)$. Indeed, by Theorem~\ref{t:or}(4), the relation $ay\updownarrows ab$ implies $\dist(y,b)=\dist(y,a)+\dist(a,b)=r-t$. By Proposition~\ref{p:sum=>arrows}, $by\upuparrows ba$. Then for the point $x\defeq b$, we have $\dist(a,x)-\dist(x,y)=\dist(a,b)-\dist(b,y)=r-(r-t)=t$,  $ax=ab$ and $xy=by\upuparrows ba=xa$, witnessing that $y\in Y_t$ and $y=\ell_{ab}^\pm(t)$.
\smallskip

2. Now consider the second case $y\in \Sf^\pm(a;\IQ r)\setminus\Sf(a;\IQ r)$. In this case there exists $z\in \Sf(a;\IQ r)$ such that $y\in \Sf(z;\IQ r)$. It follows from $y\notin \Sf(a;\IQ)$ that $z\ne a$. By Theorem~\ref{t:or}(2), either $az\upuparrows ab$ or $az\updownarrows ab$.

2.1. If $az\upuparrows ab$, then put $x\defeq z$. By Theorem~\ref{t:or}(2), either $xy\upuparrows xa$ or $xy\updownarrows xa$. In the latter case, by Theorem~\ref{t:or}(4), $\dist(a,y)=\dist(a,x)+\dist(x,y)\in \IQ r$, which contradicts $y\notin \Sf(a;\IQ)$. This contradiction shows that $xy\upuparrows xa$ and hence $y\in Y_t$ and $y=\ell_{ab}^\pm(t)$ for $t=\dist(a,x)-\dist(x,y)$.

2.2. Finally assume that $az\updownarrows ab$.  By Theorem~\ref{t:or}(2), either $zy\upuparrows za$ or $zy\updownarrows za$. In the latter case, by Theorem~\ref{t:or}(4), $\dist(a,y)=\dist(a,z)+\dist(z,y)\in \IQ r$, which contradicts $y\notin \Sf(a;\IQ)$. This contradiction shows that $zy\upuparrows za$. 

By Theorem~\ref{t:SC}, there exists a point $x\in X$ such that $\dist(a,x)=\dist(z,y)$ and $\dist(z,x)=\dist(z,a)+\dist(a,x)=\dist(z,a)+\dist(z,y)\in\IQ r\setminus\{0\}$. Proposition~\ref{p:sum=>arrows} implies $az\updownarrows ax$, $zx\upuparrows za$ and $xz\upuparrows xa$. 
By Proposition~\ref{p:updown}(5), $ax\updownarrows az\updownarrows ab$ implies $ax\upuparrows ab$. By Theorem~\ref{t:SC}, there exists a point $x'\in X$ such that $\dist(y,x')=\dist(z,a)$ and $\dist(z,x')=\dist(z,y)+\dist(y,x')=\dist(z,y)+\dist(z,a)=\dist(z,x)$. Proposition~\ref{p:sum=>arrows} implies $zx'\upuparrows zy\upuparrows za\upuparrows zx$ and $x'y\upuparrows x'z$. By Theorem~\ref{t:or}(6), $x=x'$ and hence $xy=x'y\upuparrows x'z=xz\upuparrows xa$. Now we see that $y\in Y_t$ and $y=\ell_{ab}^\pm(t)$ for $t\defeq\dist(a,x)-\dist(x,y)$.
\end{proof}

Finally, we prove the uniqueness of the function $\ell_{ab}^\pm$.
 
\begin{claim} The function $\ell_{ab}^\pm$ is equal to any function $f:M_r-M_r\to X$ such that $f(0)=a$, $f(r)=b$ and 
$\dist(f(u-v),f(u))=v$, $\dist(f(u-v),f(-v))=u$ for every $u,v\in M_r$.
\end{claim}

\begin{proof} By the choice of $f$ and Claim~\ref{cl:ell-eq}, 
$$
\begin{gathered}
\dist(a,f(-r))=\dist(f(r-r),f(-r))=r=\dist(\ell_{ab}^\pm(0),\ell_{ab}^\pm(-r))=\dist(a,\ell_{ab}^\pm(-r))\quad\mbox{and}\\ 
\dist(f(-r),b)=\dist(f(r-2r),f(r))=2r=\dist(\ell_{ab}^\pm(r),\ell_{ab}^\pm(-r))=\dist(b,\ell_{ab}^\pm(-r)),
\end{gathered}
$$and hence $f(-r)=\ell_{ab}^\pm(-r)$, according to Theorem~\ref{t:GPS}.

For every $t\in M_r$ we have  $t+r\in M_r$, by Theorem~\ref{t:Sigma-floppy}. Applying Claim~\ref{cl:ell-eq}, we obtain
$$\dist(a,f(t))=\dist(f(t-t),f(t))=t=\dist(\ell_{ab}^\pm(0),\dist_{ab}^\pm(t))=\dist(a,\ell_{ab}^\pm(t))$$and 
$\dist(\ell_{ab}^\pm(-r),\ell_{ab}^\pm(t))=t+r=\dist(f(t-(t+r)),f(t))=\dist(f(-r),f(t))=\dist(\ell_{ab}^\pm(-r),f(t))$ and hence $f(t)=\ell_{ab}^\pm(t)$, by Theorem~\ref{t:GPS}.

On the other hand, Claim~\ref{cl:ell-eq} implies $$\dist(a,f(-t))=\dist(f(t-t),f(-t))=t=\dist(\ell_{ab}^\pm(0),\ell_{ab}^\pm(-t))=\dist(a,\ell_{ab}^\pm(-t))$$and 
$\dist(\ell_{ab}^\pm(r),\ell_{ab}^\pm(-t))=t+r=\dist(f((t+r)-t),f(-t)=\dist(f(r),f(-t))=\dist(\ell_{ab}^\pm(r),f(-t))$ and hence $f(-t)=\ell_{ab}^\pm(-t)$, by Theorem~\ref{t:GPS}.

Therefore, $f(t)=\ell_{ab}^\pm(t)$ for all $t\in \pm M_r$. Now take any nonzero number $t\in M_r-M_r$ and find real numbers $u,v\in M_r$ such that $t=u-v$. It follows from $t\ne 0$ and $-v\le 0\le u$ that $-v\ne u$ and hence $\ell_{ab}^\pm(-v)\ne\ell_{ab}^\pm(u)$. Since $u,-v\in\pm M_r$, we have
$f(u)=\ell_{ab}^\pm(u)\ne \ell_{ab}^\pm(-v)=f(-v)$. 
The property of $f$ and Claim~\ref{cl:ell-eq} imply
$$
\begin{gathered}
\dist(f(t),f(u))=\dist(f(u-v),f(u))=v=|t-u|=\dist(\ell_{ab}^\pm(t),\ell_{ab}^\pm(u))=\dist(\ell_{ab}^\pm(t),f(u))\quad\mbox{and}\\
\dist(f(t),f(-v))=\dist(f(u-v),f(-v))=u=|t-(-v)|=\dist(\ell_{ab}^\pm(t),\ell_{ab}^\pm(-v))=\dist(\ell_{ab}^\pm(t),f(-v)).
\end{gathered}
$$
Aplying Theorem~\ref{t:GPS}, we conclude that $f(t)=\ell_{ab}^\pm(t)$.
\end{proof}
\end{proof}

\section{Banakh groups}

Many examples of Banakh spaces have the structure of a metric group, which motivates studying Banakh metric groups in more details.

First we recall some algebraic notions related to groups.

A {\em semigroup} is a set $X$ endowed with an associative binary operation $+:X\times X\to X$, $+:(x,y)\mapsto x+y$. The {\em associativity} of the binary operation means that $$\forall x,y,z\in X\;\;(x+y)+z=x+(y+z).$$ 

A semigroup $X$ is called 
\begin{itemize}
\item {\em commutative} if its binary operation is {\em commutative}, i.e., $\forall x,y\in X\;\;(x+y=y+x)$;
\item a {\em monoid} if $X$ possessing a {\em neutral element} $\mathbf 0\in X$ such that $\forall x\in X\;(x+\mathbf 0=\mathbf 0+x)$;
\item a {\em group} if $X$ is a monoid such that for every $x\in X$ there exists an element $y\in X$ such that $x+y=\mathbf 0=y+x$.
\end{itemize}

A metric $\dist:X\times X\to\IR_+$ on a group $X$ is called {\em invariant} if $$\dist(a+x+b,a+y+b)=\dist(x,y)$$for all $a,b,x,y\in X$. 

A {\em norm} on a group $X$ is a function $\|\cdot\|:X\to\IR$, $\|\cdot\|:x\mapsto \|x\|$, satifying three axioms:
\begin{itemize}
\item $\|x+y\|\le\|x\|+\|y\|$,
\item $\|y+x-y\|=\|x\|=\|{-}x\|$,
\item $\|x\|=0\;\Leftrightarrow\;x=\mathbf 0$,
\end{itemize}
for every $x,y\in X$.

A norm $\|\cdot\|:X\to\IR$ on a group $X$ is called {\em $\IZ$-homogeneous} if $$\|nx\|=|n|\cdot\|x\|$$for any $n\in\IZ$ and $x\in X$. The elements $nx$ are defined by the recursive formula: 
$$0x=\mathbf 0\quad\mbox{and}\quad (n+1)x=nx+x,\quad -(n+1)x=-nx-x$$for every integer $n\ge 0$.

For any norm $\|\cdot\|:X\to\IR$ on a group $X$, the function $\dist:X\times X\to \IR$, $\dist:(x,y)\mapsto\|x-y\|$, is an invariant metric on the group $X$. Conversely, for every invariant metric $\dist:X\times X\to\IR$ on a group $X$, the function $\|\cdot\|:X\to\IR$, $\|\cdot\|:x\mapsto \|x\|\defeq\dist(x,\mathbf 0)$, is a norm on $X$. 

By a {\em normed group} we understand a group endowed with a norm. By a {\em Banakh group} we understand a normed group which is a Banakh metric space with respect to the invariant metric generated by the norm. Our main result on Banakh groups is Theorem~\ref{t:BG} saying that every Banakh group is a normed $\IZ$-module. A {\em normed $\IZ$-module} is a commutative group $X$ endowed with a $\IZ$-homogeneous norm $\|\cdot\|:X\to\IR$.

\begin{theorem}\label{t:BG} Every Banakh group $X$ is a normed $\IZ$-module.   
\end{theorem}

\begin{proof}  Let $\|\cdot\|:X\to\IR$ be the norm of the Banakh group $X$ and $$\dist:X\times X\to\IR,\quad \dist(x,y)\defeq\|x-y\|,$$ be the invariant metric induced by the norm $\|\cdot\|$. Let $\mathbf 0$ be the neutral element of the group $X$. 

\begin{claim}\label{cl:BG1} The Banakh group $X$ is commutative.
\end{claim}

\begin{proof}
 Assuming that the group $X$ is not commutative, we can find two elements $x,y\in X$ such that $x+y\ne y+x$. Then also $x+y-x\ne y\ne -x+y+x$. Since $\|x+y-x\|=\|y\|=\|-x+y+x\|$, the Banakh property of $X$ implies that $x+y-x=-x+y+x$ and $\|x+y-x-y\|=\dist(x+y-x,y)=2\|y\|$. The equality $x+y-x=-x+y+x$ implies $2x+y=y+2x$. 
 
Adding $-2x$ to the inequality $x+y\ne y+x$ and using the commutativity $2x+y=y+2x$, we obtain that $-x+y\ne -2x+y+x=y-x$ and hence $-x\ne y-x-y$. Since $\|y-x-y\|=\|{-}x\|$, the Banakh property implies that $\|x+y-x-y\|=\dist(-x,y-x-y)=2\|{-}x\|=2\|x\|$. Then $2\|x\|=\|x+y-x-y\|=2\|y\|$ and hence $\|{-}x\|=\|x\|=\|y\|=\|{-}y\|$. It follows from $x+y\ne y+x$ that $y\notin\{x,-x\}$ and $x\notin\{y,-y\}$. Taking into account that $|\{x,x^{-1},y,y^{-1}\}|\le |\Sf(\mathbf 0,\|x\|)|=2$, we conclude that $-x=x\ne y=-y$. It follows from $y+x-y\ne x\ne y$ and $\|y+x-y\|=\|x\|=\|y\|$ that $y+x-y=y$ and hence $x=y$, which contradicts the choice of $x,y$.
\end{proof}

\begin{claim} The group $X$ contains no elements of order $2$.
\end{claim}

\begin{proof} To derive a contradiction, assume that $x+x=\mathbf 0$ for some element $x\ne \mathbf 0$ in $X$. By the Banakh property, the group $X$ contains an element $y\ne x$ such that $\|y\|=\|x\|$ and $\dist(x,y)=2\|x\|$. It follows from $y\ne x=-x$ and $\|{-}y\|=\|y\|=\|x\|$ that $-y=y$. It follows that $\|x+y\|=\dist(x,-y)=\dist(x,y)=2\|x\|$. By the Banakh property, the group $X$ contain an element $z$ such that $\|z\|=\|x+y\|=2\|x\|$ and $\dist(z,x+y)=4\|x\|$. The triangle inequality implies that 
$\dist(z,x)\le \|z\|+\|x\|=3\|x\|$. On the other hand, $4\|x\|=\dist(z,x+y)\le \dist(z,x)+\dist(x,x+y)\le\dist(z,x)+\|y\|=\dist(z,x)+\|x\|$ implies $3\|x\|\le\dist(z,x)$ and hence $\dist(z,x)=3\|x\|$. By analogy we can show that $\dist(z,y)=3\|x\|$. Since $x\ne y$, the Banakh property guarantees that $2\|x\|=\dist(x,y)=6\|x\|$, which implies $\|x\|=0$ and contradicts the choice of $x\ne\mathbf 0$. This contradiction shows that the group $X$ contains no elements of order two.
\end{proof}

\begin{claim}\label{cl:BG3} For every $x\in X$, there exists an isometric embedding  $\ell:\IZ\|x\|\to X$ such that $\ell(n\|x\|)=nx$ and $\|nx\|=|n|\cdot\|x\|$ for all $n\in\IZ$.
\end{claim}

\begin{proof} By Theorem~\ref{t:main1}, there exists an isometric embedding $\ell:\IZ\|x\|\to X$ such that $\ell(0)=\mathbf 0$ and $\ell(\|x\|)=x$. Then $\Sf(nx;\|x\|)=\{(n-1)x,(n+1)x\}$. We claim that $\ell(n\|x\|)=nx$ for all $n\in\IZ$. For $n\in\{0,1\}$ this follows from the choice of the isometry $\ell$. The inequality $x\ne -x$ and the Banakh property of $X$  ensure that $\Sf(\mathbf 0,\|x\|)=\{-x,x\}$. Taking into account that $\ell:\|x\|\IZ\to X$ is an isometric embedding with $\ell(\|x\|)=x$, we conclude that $\{\ell(-\|x\|),x\}=\{\ell(-\|x\|),\ell(\|x\|)\}\subseteq \Sf(\mathbf 0,\|x\|)=\{-x,x\}$ and hence $\ell(-\|x\|)=-x$.

Assume that for some $n\in\IN$ we have proved that $\ell(k\|x\|)=kx$ for all $k\in\IZ$ with $|k|\le n$. It follows from $x+x\ne\mathbf 0$ that $(n-1)x\ne (n+1)x$ and hence $\Sf(nx,\|x\|)=\{(n-1)x,(n+1)x\}$. Taking into account that $\ell:\|x\|\IZ\to X$ is an isometric embedding with $\ell((n-1)\|x\|)=(n-1)x$ and $\ell(n\|x\|)=nx$, we conclude that 
\begin{multline*}
\{(n-1)x,\ell((n+1)\|x\|)\}=\{\ell((n-1)\|x\|),\ell((n+1)\|x\|)\}\\
\subseteq\Sf(\ell(n\|x\|);\|x\|)=\Sf(nx;\|x\|)=\{(n-1)x,(n+1)x\}
\end{multline*} 
and hence $\ell(n+1)\|x\|)=(n+1)x$. By analogy we can prove that $\ell(-(n+1)\|x\|)=-(n+1)x$. This completes the proof of the inductive step.

The Principle of Mathematical Induction ensures that $\ell(n\|x\|)=nx$ for all $n\in \IZ$. Then for every $n\in\IZ$ we have the equality
$$\|nx\|=\dist(nx,\mathbf 0)=\dist(\ell(n\|x\|),\mathbf 0)=\big|n\|x\|-0\big|=|n|\cdot\|x\|.$$
\end{proof}

Claims~\ref{cl:BG1} and \ref{cl:BG3} imply that $X$ is a normed $\IZ$-module.
\end{proof}

\begin{proposition} A normed $\IZ$-module $X$ is Banakh if and only if $|\Sf(\mathbf 0;r)|\le2$ for every $r\in\IR_+$.
\end{proposition}

\begin{proof} The ``only if'' part is trivial. To prove the ``if'' part, assume that  $|\Sf(\mathbf 0;r)|\le2$ for every $r\in\IR_+$. Then for every $a,b,c\in X$, the invariance of the metric $\dist:X\times X\to\IR$, $\dist(x,y)\defeq\|x-y\|$, generated by the norm implies that $\Sf(c,\dist(a,b))=c+\Sf(\mathbf 0;\dist(a,b))=c+\{a-b,b-a\}$ and $\dist(a-b,b-a)=\|2(a-b)\|=2\|a-b\|=2\dist(a,b)$, witnessing that the metric space $(X,\dist)$ is Banakh.
\end{proof}

\begin{corollary}\label{c:nZ-B} A normed $\IZ$-module $X$ is Banakh if and only if for any $x,y\in X$ the equality $\|x\|=\|y\|$ implies $y\in\{x,-x\}$.
\end{corollary}

A subset $X$ of a group $G$ is defined to be
\begin{itemize}
\item {\em $p$-divisible} by a prime number $p$ if for every $x\in X$ there exists $y\in X$ such that $py=x$;
\item {\em divisible} if it is $p$-divisible for every prime number $p$.
\end{itemize}

\begin{proposition} A Banakh group $X$ is $p$-divisible by a prime number $p$ if and only if the subset $\dist[X^2]$ is $p$-divisible in $\IR$.
\end{proposition}

\begin{proof} To prove the ``only if'' part, assume that a Banakh group $X$ is $p$-divisible. To see that the subset $\dist[X^2]$ of the group $\IR$ is $p$-divisible, take any real number $r\in\dist[X^2]$ and find an element $x\in X$ such that $\|x\|=r$. By the $p$-divisibility of $X$, there exists an element $y\in X$ such that $x=py$. By Theorem~\ref{t:BG}, the norm of the Banakh group $X$ is $\IZ$-homogeneous and hence $r=\|x\|=\|py\|=p\|y\|$, witnessing that the set $\dist[X^2]$ is $p$-divisible.
\smallskip

To prove the ``if'' part, assume that the set $\dist[X^2]$ is $p$-divisible. To see that the group $X$ is $p$-divisible, take any $x\in X$ and by the $p$-divisibility of the set $\dist[X^2]$, find an element $y\in X$ such that $\|x\|=p\|y\|$. We claim that $py\in\{x,-x\}$. Indeed, by Theorem~\ref{t:main1}, there exists an isometric embedding $\ell:\IZ\|y\|\to X$ such that $\ell(n\|y\|)=ny$ for all $n\in\IZ$. Then $\|py\|=p\|y\|=\|x\|$ and $py\in\{x,-x\}$, by the Banakh property of $G$. If $py=x$, then the equality $x=py$ witnesses that the group $X$ is $p$-divisible. If $py=-x$, then the equality $x=p(-y)$ witnesses that $X$ is $p$-divisible.
\end{proof}

\section{$\ell_2$-independent functions}

The principal result of this section is Lemma~\ref{l:ind} on the existence of an $\ell_2$-indepenent function on the cardinal $\mathfrak c$. This lemma will be used in constructions of a discrete Banakh space in Theorem~\ref{t:H1}.

\begin{definition} A real-valued function $r:\kappa\to \IR$ on a nonzero ordinal $\kappa$ is defined to be {\em $\ell_2$-independent} if $r(0)=1$ and for every functions $f:D[f]\to\IQ\setminus\{0\}$ and $g:D[g]\to\IQ\setminus\{0\}$ defined on finite subsets $D[f]$ and $D[g]$ of $\kappa$,  the following equivalences hold:
$$
\begin{aligned}
f=g&\Leftrightarrow\sum_{\alpha\in D[f]}f(\alpha)r(\alpha)=\sum_{\alpha\in D[g]}g(\alpha)r(\alpha),\\
f=\pm g&\Leftrightarrow\big|\sum_{\alpha\in D[f]}f(\alpha)r(\alpha)\big|^2+\sum_{0\ne \alpha\in D[f]}|f(\alpha)|^2= \big|\sum_{\alpha\in D[g]}g(\alpha)r(\alpha)\big|^2+\sum_{0\ne\alpha\in D[g]}|g(\alpha)|^2.
\end{aligned}
$$
\end{definition}

In the proof of the  existence of an $\ell_2$-independent function on the cardinal $\mathfrak c$, we shall apply the following elementary lemma.

\begin{lemma}\label{l:real} For real numbers $a_1,a_2,a_2,b_1,b_2,b_3$ with $(a_1,b_1)\ne(0,0)$ and $(a_1,a_2)\ne\pm(b_1,b_2)$, the set
$$T\defeq\{t\in\IR:|a_1t+a_2|^2+a_3^2=|b_1t+b_2|^2+b_3^2\}$$ has cardinality $|T|\le 2$.
\end{lemma}

\begin{proof} Observe that for a real number $t$, the equation 
$$|a_1t+a_2|^2+a_3^2=|b_1t+b_2|^2+b_3^2$$
is equivalent to the equation $$((a_1-b_1)t+(a_2-b_2))((a_1+b_1)t+(a_2+b_2))=b_3^2-a_3^2,$$which has more than two solutions if and only if one of the following conditions holds:
\begin{itemize}
\item $a_1-b_1=0=a_1+b_1$ and $a_2^2-b_2^2=b_3^2-a_3^2$,
\item $a_1-b_1=0=a_2-b_2$ and $b_3^2=a_3^2$,
\item $a_1+b_1=0=a_2+b_2$ and $b_3^2=a_3^2$.
\end{itemize}
Those conditions are equivalent to the conditions
\begin{itemize}
\item $(a_1,b_1)=(0,0)$ and $a_2^2+a_3^2=b_2^2+b_3^2$,
\item $(a_1,a_2)=(b_1,b_2)$ and $a_2^2+a_3^2=b_2^2+b_3^2$,
\item $(a_1,a_2)=-(b_1,b_2)$ and $a_2^2+a_3^2=b_2^2+b_3^2$,
\end{itemize}
which are excluded by the assumption of the lemma.
\end{proof}

\begin{lemma}\label{l:ind} There exists an $\ell_2$-indepenent function $r:\mathfrak c\to \IR$.
\end{lemma}

\begin{proof} 
Let $\mathcal R$ be the family of all $\ell_2$-indepenent functions $r:D[r]\to\IR$ defined on nonzero ordinals $D[r]\le\mathfrak c$. The family $\mathcal R$ is endowed with the partial order $\subseteq$ (each function is identified with its graph). Observe that the function $\{\langle 0,1\rangle\}$ is the smallest element of the poset $\mathcal R$. It is easy to see that for every linearly ordered subset $\mathcal L\subseteq\mathcal R$, the union $\bigcup\mathcal L$ is an $\ell_2$-indepenent function such that $r\subseteq \bigcup\mathcal L$ for all $r\in\mathcal L$. By the Kuratowski--Zorn Lemma, the partially ordered set $\mathcal R$ has a maximal element $r$, which is an $\ell_2$-indepenent function $r:\delta\to \IR$, defined on some nonzero ordinal $\delta\le\mathfrak c$. We claim that the ordinal $\delta$ equals the cardinal $\mathfrak c$. To derive a contradiction, assume that $\delta\ne \mathfrak c$ and hence $|\delta|<\mathfrak c$.

For every function $f:D[f]\to\IQ\setminus\{0\}$ defined on a finite subset $D[f]$ of $\mathfrak c$, consider the function $\bar f:\mathfrak c\to\IQ$  defined by 
$$\bar f(\alpha)=\begin{cases}
f(\alpha)&\mbox{if $\alpha\in D[f]$};\\
0&\mbox{otherwise}.
\end{cases}
$$

Let $P$ be the family of all ordered pairs $(f,g)$ consisting of functions $f:D[f]\to\IQ\setminus\{0\}$ and $g:D[g]\to\IQ\setminus\{0\}$ defined on finite subsets $D[f],D[g]$ of the ordinal $\delta+1=D[r]\cup\{\delta\}$ such that  $f\ne g$. 

\begin{claim} For every pair $(f,g)\in P$, the set 
$$T'_{fg}\defeq\{t\in\IR:\bar f(\delta)t+\sum_{\alpha\in D[f]\cap \delta}f(\alpha)r(\alpha)=\bar g(\delta)r+\sum_{\alpha\in D[g]\cap \delta}g(\alpha)r(\alpha)\}$$has cardinality $|T'_{fg}|\le 1$.
\end{claim}

\begin{proof} If $\bar f(\delta)=\bar g(\delta)$, then $f\ne g$ implies $f{\restriction}_{D[f]\cap\delta}\ne g{\restriction}_{D[g]\cap\delta}$. In this case, the $\ell_2$-indepenence of the function $r$ ensures that $\sum_{\alpha\in D[f]\cap\delta}f(\alpha)r(\alpha)\ne\sum_{\alpha\in D[g]\cap\delta}g(\alpha)r(\alpha)$, and hence the set 
$$
\begin{aligned}
T'_{fg}&=\{t\in \IR:\bar f(\delta)t+\sum_{\alpha\in D[f]\cap\delta}f(\alpha)r(\alpha)=\bar g(\delta)t+\sum_{\alpha\in D[g]\cap \delta}g(\alpha)r(\alpha)\}\\
&=\{t\in \IR:\sum_{\alpha\in D[f]\cap\delta}f(\alpha)r(\alpha)=\sum_{\alpha\in D[g]\cap\delta}g(\alpha)r(\alpha)\}
\end{aligned}
$$is empty.

If $\bar f(\delta)\ne\bar g(\delta)$, then the set 
$$
\begin{aligned}
T'_{fg}&=\{t\in \IR:\bar f(\delta)r+\sum_{\alpha\in D[f]\cap \delta}f(\alpha)r_\alpha=\bar g(\delta)r+\sum_{\alpha\in D[g]\cap\delta}g(\alpha)r_\alpha\}\\
&=\{r\in \IR:(\bar f(\delta)-\bar g(\delta))r=\sum_{\alpha\in D[g]\cap\delta}g(\alpha)r_\alpha-\sum_{\alpha\in D[f]\cap\delta}f(\alpha)r_\alpha\}
\end{aligned}
$$
is a singleton.
\end{proof}

Let $P_\pm\defeq\{(f,g)\in P:f\ne -g\}$  and for every pair $(f,g)\in P_\pm$, consider the set $T''_{fg}$ of real numbers $t$ such that 
$$\big|\bar f(\delta)t+\!\!\sum_{\alpha\in D[f]\cap \delta}\!f(\alpha)r(\alpha)\big|^2+\sum_{0\ne\alpha\in D[f]}|f(\alpha)|^2=\big|\bar g(\delta)t+\!\!\sum_{\alpha\in D[f]\cap \delta}\!g(\alpha)r(\alpha)\big|^2+\sum_{0\ne\alpha\in D[f]}|g(\alpha)|^2.$$

\begin{claim} For every $(f,g)\in P_\pm$, the set $T''_{fg}$ has cardinality $|T''_{fg}|\le 2$.
\end{claim}

\begin{proof} If $D[f]\cup D[g]\subseteq \delta$, then the $\ell_2$-indepenence of $r$ ensures that 
$$\Big|\sum_{\alpha\in D[f]}f(\alpha)r(\alpha)\Big|^2+\sum_{0\ne\alpha\in D[f]}f(\alpha)^2\ne 
\Big|\sum_{\alpha\in D[g]}g(\alpha)r(\alpha)\Big|^2+\sum_{0\ne\alpha\in D[g]}g(\alpha)^2,$$
which implies that the set $T''_{(f,g)}$ is empty.

So, assume that $D[f]\cup D[g]\not\subseteq \delta$ and hence $\delta\in D[f]\cup D[g]$. If $\bar f(\delta)=\bar g(\delta)$, then $\delta\in D[f]\cap D[g]$ implies that $\bar f(\delta)=\bar g(\delta)\ne 0$. The inequality $f\ne g$ implies the inequality $f{\restriction}_{D[f]\cap\delta}\ne g{\restriction}_{D[g]\cap\delta}$ and the $\ell_2$-indepenence of $r$ ensures that $\sum_{\alpha\in D[f]\cap\delta}f(\alpha)r(\alpha)\ne \sum_{\alpha\in D[g]\cap\delta}g(\alpha)r(\alpha)$. By Lemma~\ref{l:real},  $|T''_{fg}|\le 2$.

If $\bar f(\delta)=-\bar g(\delta)$, then $\delta\in D[f]\cap D[g]$ implies that $\bar f(\delta)=-\bar g(\delta)\ne 0$. The inequality $f\ne -g$ implies the inequality $f{\restriction}_{D[f]\cap\delta}\ne -g{\restriction}_{D[g]\cap\delta}$ and the $\ell_2$-indepenence of $r$ ensures that $\sum_{\alpha\in D[f]\cap\delta}f(\alpha)r(\alpha)\ne -\sum_{\alpha\in D[g]\cap\delta}g(\alpha)r(\alpha)$. By Lemma~\ref{l:real},  $|T''_{fg}|\le 2$.

It remains to consider the case $\bar f(\delta)\ne \pm \bar g(\delta)$. In this case Lemma~\ref{l:real} implies that $|T''_{fg}|\le 2$.
\end{proof}

It is easy to see that $|P|\le\max\{|\delta|,\w\}<\mathfrak c$. Then the set $$T \defeq\bigcup_{(f,g)\in P}T'_{fg}\cup\bigcup_{(f,g)\in P_\pm}T''_{fg}$$ has cardinality $<\mathfrak c$ and hence there exists a real number $t\notin T$. The choice of $t\notin T$ ensures that the function $\hat r\defeq r\cup\{(\delta,t)\}$ is $\ell_2$-independent and hence $\hat r$ is an element of the poset $\mathcal R$, which is strictly larger than $r$.
But this contradicts the maximality of $r$. This contradiction shows that $\delta=\mathfrak c$, which means that $r$ is an $\ell_2$-idependent function on the cardinal $\mathfrak c$.
\end{proof}

\section{Dense Banakh groups in Hilbert spaces}\label{s:ex}

In this section we construct examples of Banakh groups which are not isometric to subsets of the real line. Those examples are subgroups of the Hilbert spaces $\ell_2(\kappa)$.

Given a cardinal $\kappa$, we denote by $\ell_2(\kappa)$ the Hilbert space consisting of all functions $x:\kappa\to\IR$ such that $\sum_{\alpha\in\kappa}|x(\alpha)|^2<\infty$. The space $\ell_2(\kappa)$ is endowed with the inner product $$\langle x|y\rangle\defeq\sum_{\alpha\in\kappa}x(\alpha)\cdot y(\alpha),$$which induces the norm $\|x\|\defeq\sqrt{\langle x|x\rangle}$. In its turn, the norm indices the complete invariant metric $\dist(x,y)\defeq\|x-y\|$ on $\ell_2(\kappa)$.

A subset $X\subseteq\ell(\kappa)$ is called 
\begin{itemize}
\item a  {\em subgroup} if $X-X\subseteq X\ne\emptyset$;
\item a {\em $\IQ$-linear subspace} if $px+qy\in X$ for any $x,y\in X$ and $p,q\in\IQ$. 
\end{itemize}
For a subset $X\subseteq\ell_2(\kappa)$ its {\em group hull} (resp.  {\em $\IQ$-linear hull\/}) in $\ell_2(\kappa)$ is the smallest subgroup (resp. $\IQ$-linear subspace) of $\ell_2(\kappa)$ that contains $X$ as a subset. 

\begin{theorem}\label{t:H1} For every nonzero cardinal $\kappa\le\mathfrak c$, the Hilbert space $\ell_2(\kappa)$ contains a closed discrete Banakh subgroup $H$ of cardinality $|H|=\max\{\kappa,\w\}$ whose $\IQ$-linear hull is a dense  Banakh subgroup of $\ell_2(\kappa)$.
\end{theorem}

\begin{proof} By Lemma~\ref{l:ind}, there exists an $\ell_2$-independent function $r:\mathfrak c\to\IR$. For every ordinal $\alpha\in\mathfrak c$, consider the function $\mathbf e_\alpha:\kappa\to\IR$ defined by
$$\mathbf e_\alpha(x)=\begin{cases}
r(\alpha)&\mbox{if $x=0$};\\
1&\mbox{if $x=\alpha$};\\
0&\mbox{otherwise}.
\end{cases}
$$
Since $r(0)=1$, the vector $\mathbf e_0\in\ell_2(\kappa)$ is well-defined.

Let $L$ be the $\IQ$-linear hull of the set $\{\mathbf e_\alpha\}_{\alpha\in\kappa}$ in $\ell_2(\kappa)$, and $H$ be the group hull of the set $\{\mathbf e_\alpha\}_{\alpha\in\kappa}$ in $\ell_2(\kappa)$. Since the system $\{\mathbf e_\alpha\}_{\alpha\in\kappa}$ is linearly independent (over the field $\IR$), for every $x\in L$ there exists a unique function $f_x:D[f_x]\to\IQ\setminus\{0\}$ defined on a finite subset $D[f_x]$ of $\kappa$ such that $x=\sum_{\alpha\in D[f_x]}f_x(\alpha)\mathbf e_\alpha$. The element $x\in L$ belongs to the group $H$ if and only if $f_x[D_x]\subseteq\IZ$. 

\begin{claim} The groups $L$ and $H$ are Banakh.
\end{claim}

\begin{proof} Since $L$ and $H$ are $\IZ$-modules, by Corollary~\ref{c:nZ-B}, it suffices to show that for any $x,y\in L$, the equality $\|x\|=\|y\|$ implies $x=\pm y$. Observe that 
$$\|x\|^2=\big|\sum_{\alpha\in D[f_x]}f_x(\alpha)r(\alpha)\big|^2+\sum_{0\ne\alpha\in D[f_x]}|f_x(\alpha)|^2$$ for every $x\in L$. Then for two elements $x,y\in L$, the equality $\|x\|=\|y\|$ is equivalent to 
$$\big|\sum_{\alpha\in D[f_x]}f_x(\alpha)r(\alpha)\big|^2+\sum_{0\ne\alpha\in D[f_x]}|f_x(\alpha)|^2=\big|\sum_{\alpha\in D[f_y]}f_y(\alpha)r(\alpha)\big|^2+\sum_{0\ne\alpha\in D[f_y]}|f_y(\alpha)|^2.$$
By the $\ell_2$-independence of the function $r$, the latter equality is equivalent to $f_x=\pm f_y$, which is equivalent to $x=\pm y$.
\end{proof}

\begin{claim} The group $H$ is closed and discrete in $\ell_2(\kappa)$.
\end{claim}

\begin{proof} To see that $H$ is closed and discrete, it suffices to show that every non-zero element $x\in H$ has norm $\|x\|\ge 1$. If $D[f_x]\subseteq \{0\}$, then $\|x\|=|f_x(0)|\cdot\|\mathbf e_\alpha\|=|f(0)|\ge 1$. If $D[f]\not\subseteq \{0\}$, then $$\|x\|^2=|\sum_{\alpha\in D[f_x]}f_x(\alpha)r(\alpha)|^2+\sum_{0\ne \alpha\in D[f_x]}|f_x(\alpha)|^2\ge \sum_{0\ne\alpha\in D[f_x]}|f_x(\alpha)|^2\ge 1.$$
\end{proof}

\begin{claim} The group $L$ is dense in $\ell_2(\kappa)$.
\end{claim}

\begin{proof} For every $\alpha\in\kappa$, consider the function $\mathbf b_\alpha:\kappa\to\{0,1\}$ such that $\mathbf b_\alpha^{-1}(1)=\{\alpha\}$. Observe that $\mathbf b_0=\mathbf e_0$ and $\mathbf b_\alpha=\mathbf e_\alpha-r(\alpha)\mathbf e_0$ for every nonzero ordinal $\alpha\in\kappa$. Therefore, the standard orthonormal basis $\{\mathbf b_\alpha\}_{\alpha\in\kappa}$ of the Hilbert space $\ell_2(\kappa)$ is contained in the $\IR$-linear hull of $L$ of the set $\{\mathbf e_\alpha\}_{\alpha\in\kappa}$. Now the density of the group $L$ in $\ell_2(\kappa)$ follows from the density of the $\IQ$-linear hull of $L$ in its $\IR$-linear hull and the density of the $\IR$-linear hull of the standard orthonormal basis $\{\mathbf b_\alpha\}_{\alpha\in\kappa}$ in $\ell_2(\kappa)$.
\end{proof}
\end{proof}

The following theorem shows that the completeness is essential in Theorem~\ref{t:WB}

\begin{theorem}\label{t:H2} For every nonzero cardinal $\kappa\le\mathfrak c$, the Hilbert space $\ell_2(\kappa)$ contains a dense divisible Banakh subgroup $L$ such that $\dist[L^2]=\IR_+$.
\end{theorem}

\begin{proof} If $\kappa=1$, then $\ell_2(\kappa)$ is isomorphic to the real line and $H=\ell_2(\kappa)$ is the required dense divisible Banakh subgroup of $\ell_2(\kappa)$ with $\dist[H^2]=\IR_+$. So, we assume that $\kappa>1$. 

For every ordinal $\alpha\in\kappa$, consider the function $\mathbf b_\alpha:\kappa\to\{0,1\}$ such that $\mathbf b_\alpha^{-1}(1)=\{\alpha\}$. Therefore, $(\mathbf b_\alpha)_{\alpha\in\kappa}$ is the standard orthonormal basis of the Hilbert space $\ell_2(\kappa)$. For every $\alpha\in\mathfrak c$ consider the vector $$\mathbf e_\alpha=\begin{cases}\mathbf b_\alpha&\mbox{if $\alpha\in \kappa$};\\
\mathbf b_1&\mbox{otherwise}
\end{cases}
$$

Write the set $\IR_+$ as $(r_\alpha)_{\alpha\in\mathfrak c}$ for some pairwise distinct real numbers such that $r_0=0$ and $r_1=1$. The enumeration $(r_\alpha)_{\alpha\in\mathfrak c}$ induces the well-order $\preceq$ on $\IR$ defined by $x\preceq y$ if and only if $x=r_\alpha$ and $y=r_\beta$ for some ordinals $\alpha\le\beta<\mathfrak c$. The well-order $\preceq$ has the following property: for every $x\in \IR$, the set ${\downarrow}x\defeq\{y\in\IR:y\preceq x\}$ has cardinality $|{\downarrow}x|<\mathfrak c$.

For a nonempty set $A\subseteq \IR$ denote by $\min_{\preceq} A$ the smallest element of the set $A$ with respect to the well-order $\preceq$.

By transfinite induction we shall construct a transfinite sequence $(X_\alpha)_{\alpha\in\mathfrak c}$ of Banakh  divisible subgroups of $\ell_2(\kappa)$ such that $X_0=\IQ\mathbf e_0$ and for every  ordinal $\beta\in\mathfrak c$ the following conditions are satisfied:
\begin{enumerate}
\item[$(1_\beta)$] if $\beta$ is an infinite limit ordinal, then $X_\beta=\bigcup_{\alpha\in\beta}X_\alpha$;
\item[$(2_\beta)$] If $\beta=\alpha+1$ is a successor ordinal, then $X_{\beta}=X_\alpha+\IQ \mathbf v_{\beta}$ for some vector\newline $\mathbf v_{\beta}\in (\IR\mathbf e_0+\IR\mathbf e_\alpha)\setminus\IR\mathbf e_0$ of norm $\|\mathbf v_\beta\|=\min_{\preceq}(\IR_+\setminus\{\|x\|:x\in X_\alpha\})$;
\item[$(3_\beta)$] $|X_\beta|\le\max\{\w,|\beta|\}<\mathfrak c$.
\end{enumerate}
To start the inductive construction, put $X_0=\IQ\mathbf e_0$. Assume that for some ordinal $\beta<\mathfrak c$ we have constructed an increasing transfinite sequence of Banakh divisible subgroups $(X_\alpha)_{\alpha\in\beta}$ satisfying the inductive conditions. If $\beta$ is a limit ordinal, then put $X_\beta=\bigcup_{\alpha\in\beta}X_\alpha$ and observe that $|X_\beta|\le\sum_{\alpha\in\beta}|X_\alpha|\le\sum_{\alpha\in\beta}\max\{\w,|\alpha|\}\le\max\{\w,|\beta|\}<\mathfrak c$. 

Next, assume that $\beta=\alpha+1$ is a successor ordinal. Since $|X_\alpha|\le\max\{\w,|\alpha|\}<\mathfrak c$, the positive real number $v_\beta\defeq\min_{\preceq}(\IR_+\setminus\{\|x\|:x\in X_\alpha\})$ is well-defined. Denote by $\langle \;|\;\rangle$ the inner product of the Hilbert space $\ell_2(\kappa)$.
 
\begin{claim}\label{cl:ip!=0} For every $x\in X_\alpha\setminus\{0\}$, $\langle \mathbf e_0|x\rangle\ne0$.
\end{claim}

\begin{proof} To derive a contradiction, assume that $\langle \mathbf e_0|x\rangle=0$ for some $x\in X_\alpha\setminus\{0\}$. By the Pithagoras Theorem,  the sphere $\Sf(\mathbf 0,\sqrt{\|\mathbf e_0\|^2+\|x\|^2})$ contains at least four pairwise distinct points: $$\mathbf e_0+x,\mathbf e_0-x,-\mathbf e_0+x,-\mathbf e_0+x,$$ which contradicts the Banakh property of the group $X_\alpha$.
\end{proof}

\begin{claim}\label{cl:trig} For real numbers $a,b,c$, the set 
$$T\defeq\{t\in\IR:\cos(t)a+\sin(t)b=c\}$$ is uncountable if and only if $a=b=c=0$.
\end{claim}

\begin{proof} The ``if'' part is trivial. To prove the ``only if'' part, assume that at least  one of the numbers $a,b,c$ is non-zero.

If $a=b=0$ and $c\ne 0$, then the set $T$ is empty. Next, assume that $a$ or $b$ is not zero. In this case $a^2+b^2\ne\emptyset$ and there exists a real number $\varphi$ such that $\big(\!\sin(\varphi),\cos(\varphi)\big)=\big(\frac a{\sqrt{a^2+b^2}},\frac b{\sqrt{a^2+b^2}}\big)$. Then the set 
$$\begin{aligned}
T&=\{t\in \IR:a\cos(t)+b\sin(t)=c\}\\
&=\{t\in\IR:\sqrt{a^2+b^2}\big(\sin(\varphi)\cos(t)+\cos(\varphi)\sin(t)\big)=c\}\\
&=\{t\in\IR:\sqrt{a^2+b^2}\sin(\varphi+t)=c\}
\end{aligned}
$$is at most countable.
\end{proof}

\begin{claim} For any pairs $(p,x),(q,y)\in \IQ\times X_\alpha$ with $(q,y)\notin\{(p,x),(-p,-x)\}$, the set $$T_{(p,x),(q,y)}\defeq\big\{t\in\IR:
\|pv_\beta\big(\!\cos(t)\mathbf e_0+\sin(t)\mathbf e_\alpha\big)+x\|=\|qv_\beta\big(\!\cos(t)\mathbf e_0+\sin(t)\mathbf e_\alpha\big)+y\|\big\}$$
is at most countable.
\end{claim}

\begin{proof} Observe that a real number $t$ belongs to the set $T_{(p,x),(q,y)}$   if and only if
$$p^2v_\beta^2+\|x\|^2+2pv_\beta\langle \cos(t) \mathbf e_0+\sin(t)\mathbf e_\alpha|x\rangle=q^2v_\beta^2+\|y\|^2+2qv_\beta\langle \cos(t)\mathbf e_0+\sin(t)\mathbf e_\alpha|y\rangle$$ if and only if 
\begin{equation}\label{eq:2}2v_\beta\big(\!\cos(t)\langle \mathbf e_0|px-qy\rangle+\sin(t)\langle \mathbf e_1|px-qy\rangle\big)=v_\beta(q^2-p^2)+\|y\|^2-\|x\|^2.\end{equation}
Since $X_\alpha$ is a $\IQ$-linear subspace of $\ell_2(\kappa)$, $px-qy\in X_\alpha$. If $px\ne qy$, then $\langle \mathbf e_0|px-qy\rangle\ne0$, according to Claim~\ref{cl:ip!=0}.  By Claim~\ref{cl:trig}, the equation (\ref{eq:2}) has at most countably many solution, which implies that the set $T_{(p,x),(q,y)}$ is at most countable.

If $px=qy$, then $(q,y)\notin\{(p,x),(-p,-x)\}$ implies $q\notin\{p,-p\}$ and hence $q^2\ne p^2$. 

If $p\ne 0$, then  
$$v_\beta(q^2-p^2)+\|y\|^2-\|x\|^2=v_\beta(q^2-p^2)+\|y\|^2-\frac{q^2}{p^2}\|y\|^2=
(q^2-p^2)(v_\beta-\frac1{p^2}\|y\|^2)\ne 0,$$
by the choice of $v_\beta\notin \{\|z\|:z\in X_\alpha\}$. In this case, the equation (\ref{eq:2}) has no solutions and  the set $T_{(p,x),(q,y)}$ is empty.

If $q\ne 0$, then  
$$v_\beta(q^2-p^2)+\|y\|^2-\|x\|^2=v_\beta(q^2-p^2)+\frac{p^2}{q^2}\|x\|^2-\|x\|^2=
(q^2-p^2)(v_\beta-\frac1{q^2}\|x\|^2)\ne 0$$
by the choice of $v_\beta\notin \{\|z\|:z\in X_\alpha\}$.  In this case, the equation (\ref{eq:2}) has no solutions and  the set $T_{(p,x),(q,y)}$ is empty.
\end{proof}

Let $P\defeq\{((p,x),(q,y))\in (\IQ\times X_\alpha)^2:(p,x)\notin\{(q,y),(-q,-y)\}$. Since $|P|=|\w\times X_\alpha|<\mathfrak c$, the set
$$T=\bigcup_{((p,x),(q,y))\in P}T_{(p,x),(q,y)}$$has cardinality $|T|\le|P\times \w|<\mathfrak c$, so we can choose a real number $t_\beta\notin T\cup\pi\IZ$. Then for the vector $\mathbf v_\beta=v_\beta(\cos(t_\beta)\mathbf e_0+\sin(t_\beta)\mathbf e_\alpha)$ and any pair $((p,x),(q,y))\in P$ we have $$\|p\mathbf v_\beta+x\|\ne \|q\mathbf v_\beta+y\|,$$ which implies that the $\IQ$-linear subspace $X_\beta\defeq X_\alpha+\IQ\mathbf v_\beta$ of $\ell_2(\kappa)$ is Banakh, according to Corollary~\ref{c:nZ-B}. The choice of $t_\beta\notin\pi\IZ$ guanartees that $\mathbf v_\beta\in(\IR\mathbf e_0+\IR\mathbf e_\beta)\setminus\IR\mathbf e_0$.
It is clear that $|X_\beta|=|\IQ\times X_\alpha|\le\max\{\w,|\alpha|\}\le\max\{|\w|,|\beta|\}$, witnessing that the inductive conditions $(2_\beta),(3_\beta)$ are satisfied. This completes the inductive step.
\smallskip

After completing the inductive construction, consider the $\IQ$-linear subspace  $L=\bigcup_{\alpha\in\mathfrak c}X_\alpha$ of $\ell_2(\kappa)$ and observe that $L$ is a divisible subgroup of $\ell_2(\kappa)$. Applying Corollary~\ref{c:nZ-B}, we can prove that the group $L$ is Banakh. Assuming that $\dist[L^2]\ne\IR_+$, we can find a positive real number $r\notin \{\|x\|:x\in L\}$. Then for every $\alpha\in\mathfrak c$ we have $\|\mathbf v_{\alpha+1}\|=\min_{\preceq}(\IR_+\setminus\{\|x\|:x\in X_\alpha\})\preceq r$ and hence the set ${\downarrow} r=\{x\in \IR_+:x\preceq r\}$ contains the set $\{\|\mathbf v_{\alpha+1}\|:\alpha\in\mathfrak c\}$ of cardinality $\mathfrak c $, which contradicts the choice of the well-order $\preceq$ on $\IR_+$.
This contradiction shows that $\dist[X^2]=\IR_+$.
\smallskip

To see that the subgroup $L$ is dense in $\ell_2(\kappa)$, observe that $\mathbf b_0=\mathbf e_0\in X_0\subseteq L$ and for every nonzero ordinal $\alpha\in\kappa$ we have $\mathbf b_\alpha=\mathbf e_\alpha\in \IR\mathbf v_{\alpha+1}+\IR\mathbf e_0$. Therefore, the orthonormal basis $(\mathbf b_\alpha)_{\alpha\in\kappa}$ of the Hilbert space $\ell_2(\kappa)$ belongs to the $\IR$-linear hull of the set $L$ in $\ell_2(\kappa)$. Since the $\IQ$-linear hull of $L$ coincides with $L$ and is dense in the $\IR$-linear hull of $L$, the $\IQ$-linear subspace $L$ of $\ell_2(\kappa)$ is dense $\ell_2(\kappa)$.
\end{proof} 

\section{The distance sets of Banakh spaces}

For a metric space $X$, the set $\dist[X^2]$ is called the {\em distance set} of $X$.
It is clear that every distance set is a subset of the closed half-line $\bar\IR_+$.

\begin{definition} A set $B\subseteq \bar\IR_+$ is called a {\em Banakh distance set} if $B=\dist[X^2]$ for some Banakh space $X$.
\end{definition}

In this section we will give some partial answers to the following intriguing (and still open) problem.

\begin{problem}\label{prob:Bdst} Find necessary and sufficient conditions for a set to be a Banakh distance set.
\end{problem}

The following definition introduces one of necessary properties of Banakh distance sets.

\begin{definition} A subset $B$ of the real line is called  {\em floppy} if for every real number $r$ with $r=\inf\{a+b:(a,b\in B)\;\wedge\; (r=b-a)\}$, we have $r\in B$.
\end{definition}

It is easy to see that every half-group $H\subseteq\bar\IR_+$ is floppy.

\begin{proposition}\label{p:flop0} A subset $B\subseteq\bar\IR_+$ is floppy if $\inf B\setminus\{0\}>0$.
\end{proposition}

\begin{proof} Let $\e\defeq\inf B\setminus\{0\}>0$. Here we assume that $\inf\emptyset=+\infty$. To show that the set $B$ is floppy, take any real number  $r$ with $r=\inf\{b+a:(b,a\in B)\;\wedge\; (r=b-a)\}$ and choose numbers $a,b\in B$ such that $r=b-a$ and $b+a<r+\e$. It follows that $b=r+a\ge r$ and hence $a<r+\e-b\le \e$. The definition of $\e$ ensures that $a=0$ and hence $r=b-0\in X$, witnessing that the set $B$ is floppy.
\end{proof}

\begin{example}\em By Proposition~\ref{p:flop0}, the sets $\w$ and $\w\setminus\{1\}$ are floppy monoids. The floppy monoid $\w\setminus\{1\}$ is not a half-group.
\end{example} 

For a subset $B$ of $\bar\IR_+$, let $$\ddot B\defeq\{r\in B\setminus\{0\}:\forall x,y\in (B\cap\IQ r)\setminus\{0,r\} \;\;(x+y\ne r)\}.$$

\begin{example} For the monoid $M=\w\setminus\{1\}$, we have $\ddot M=\{2,3\}$.
\end{example}

In the following theorem we collect all known (to the author) properties of Banakh distance sets.

\begin{theorem}\label{t:Bdst} Let $B\subseteq\bar\IR_+$ be a set.
\begin{enumerate} 
\item If $B$ is a half-group, then $B$ is a Banakh distance set.
\item If $B$ is a Banakh distance set and $0\in B\subseteq\IQ r$ for some $r\in B$, then $B$ is a half-group.
\item If $B$ is a Banakh distance set, then for every $r\in B$ the set $M_r\defeq B\cap \IQ r$ is a floppy monoid. 
\item $B$ is a Banakh distance set whenever $|B|\le\w_1$, for every $r\in B$, the set $M_r\defeq B\cap r\IQ$ is a floppy monoid, and for every subset $A\subseteq B$ of cardinality $|A|<|B|$, the set $\ddot B\setminus(\IQ A+\IQ A)$ is dense in $\IR_+$.
\end{enumerate}
\end{theorem}

\begin{proof} 1. If $B$ is a half-group, then $\pm B$ is a subgroup of the real line. By Theorem~\ref{t:real-Banakh}, the group $\pm B$ is a Banakh space. Since  $\dist[(\pm B)^2]=(\pm B)=B$, the set $B$ is a Banakh distance set.
\smallskip

2. Assume that $B$ is a Banakh distance set and $0\in B\subseteq\IQ r$ for some $r\in B$. Let $X$ be a Banakh space such that $\dist[X^2]=B$. Since $0\in B=\dist[X^2]$, the space $X$ is not empty and hence $X$ contains some point $x\in X$. Taking into account that $\dist[X^2]=B\subseteq\IQ r$, we conclude that $\Sf(x;\IQ r)=X$.
 Since $\dist[\Sf(x;\IQ r)^2]=\dist[X^2]=B\subseteq\IQ r$, we can apply Theorem~\ref{t:Q} and conclude that the metric space $X=\Sf(x;\IQ r)$ is isometric to a subgroup $G$ of the group $\IQ r$. Then $B=\dist[X^2]=\dist[G^2]=G_+$ is a half-group.
 \smallskip
 
 3,4. The third and fourth statements follow from Theorems~\ref{t:3} and \ref{t:exotic}, respectively.
 \end{proof}

\begin{theorem}\label{t:3} For every Banakh space $X$ and every real number $r\in\dist[X^2]$, the set $M_r \defeq \dist[X^2]\cap \IQ r$ is a floppy monoid. The monoid $M_r$ is a half-group if and only if $M_r$ is $p$-divisible in $M_r-M_r$ for some prime number $p$.
\end{theorem}

\begin{proof} Theorem~\ref{t:SC} implies that $M_r$ is a monoid. To see that the monoid $M_r$ is floppy, take any real number $s$ such that $s=I\defeq \inf\{u+v:(u,v\in M_r)\;\wedge\;(r=u-v)\}$. Here we assume that $\inf\emptyset=+\infty$. The equality $s=I$ implies $s\ge 0$ and $s\in M_r-M_r\subseteq\IQ r-\IQ r=\IQ r$. Since $r\in\dist[X^2]$, there exist points $a,b\in X$ with $\dist(a,b)=r$. By Theorem~\ref{t:Spm}, there exists a map $\ell_{a,b}^\pm:M_r-M_r\to \Sf^\pm(a;\IQ r)$ such that  
$$s=|s-0|\le \dist(\ell_{ab}^\pm(s),\ell_{ab}^\pm(0))\le I=s.$$
Then $s=\dist(\ell_{ab}^\pm(s),\ell_{ab}^\pm(0))\in \dist[X^2]\cap\IQ r=M_r$, witnessing that the monoid $M_r$ is floppy.

By Theorem~\ref{t:Dzik}, the monoid $M_r$ is a half-group if and only if $M_r$ is $p$-divisible in $M_r-M_r$ for some prime number $p$.
\end{proof}

The main result of this section is the following (difficult) theorem.

\begin{theorem}\label{t:exotic} Assume that for a subset $M\subseteq\bar\IR_+$ the following conditions are satisfied:
\begin{enumerate}
\item $|M|\le\w_1$;
\item for every $r\in M$ the set $M_r\defeq M\cap \IQ r$ is a floppy monoid;
\item for every subset $B\subseteq M$ of cardinality $|B|<|M|$ the set $\ddot M\setminus(\IQ B+\IQ B)$ is dense in $\IR_+$.
\end{enumerate}
Then there exists a Banakh space $X$ with $\dist[X^2]=M$.
\end{theorem} 

Theorem~\ref{t:exotic} has long and difficult proof, which is postponed till the end of this section. Now let us derive one implication of this theorem, and pose two  open problems.

\begin{corollary} For every floppy monoid $S\subseteq \IQ_+$, there exists a countable Banakh space $X$ such that $\dist[X^2]\cap \IQ=S$.
\end{corollary} 

\begin{proof} Choose any countable dense subset $B\subseteq\IR_+$ such that $1\in B$ and for any pairwise distinct elements $b_1,\dots,b_n\in B$ and nonzero rational numbers, we have $\sum_{i=1}^n q_i b_i\ne 0$. Applying Theorem~\ref{t:exotic} to the countable set $M\defeq S\cup \IZ B$, we obtain a countable Banakh space $X$ such that  $\dist[X^2]=M$ and hence $\dist[X^2]\cap \IQ=M\cap \IQ=S$.
\end{proof}

\begin{remark} Theorem~\ref{t:3} shows that the condition (2) of Theorem~\ref{t:exotic} is necessary and the condition (3) of Theorem~\ref{t:exotic} is essential.  On the other hand, the condition $|M|\le\w_1$ can be removed from Theorem~\ref{t:exotic} under the Continuum Hypothesis. This motivates the following problem.
\end{remark}

\begin{problem} Let $M\subseteq\bar\IR_+$ be a set satisfying the conditions \textup{(2)} and \textup{(3)} of Theorem~{\em \ref{t:exotic}}. Is $M$ a Banakh distance set?
\end{problem}

Another open problem is suggested by the condition (3) of Theorem~\ref{t:exotic} and Corollary~\ref{c:Q}.

\begin{problem} Is a Banakh space $X$ isometric to a subgroup of the real line if $\dist[X^2]\subseteq \IQ F$ for some finite set $F\subseteq\IR_+$?
\end{problem} 



The proof of Theorem~\ref{t:exotic} involves the technique of (floppy) graph metrics, developed in \cite{BM}. First we recall the necessary definitions (including the very basic definitions of Set Theory, see \cite{Ban} or \cite{Jech}). 

Following the classical definition of Kuratowski, the {\em ordered pair} $\langle x,y\rangle$ of two sets $x,y$ is the set $\{\{x\},\{x,y\}\}$. A {\em relation} is a set whose elements are ordered pairs of sets. For any relation $R$ and an ordered pair $\langle x,y\rangle\in R$ we have $\{x,y\}\in \langle x,y\rangle$ and hence $\{x,y\}\in\bigcup R$ and $x,y\in\bigcup\bigcup R$.  For a relation $R$, the sets
$$\textstyle\dom[R]\defeq\{x:\exists y\;(\langle x,y\rangle\in R)\}\mbox{ \ and \ }\textstyle\rng[R]\defeq \{y:\exists x\;(\langle x,y\rangle \in R)\}
$$are called the {\em domain} and the {\em range} of $R$, respectively. A {\em function} is a relation $F$ such that for every $x\in\dom[F]$ there exists a unique element $y\in\rng[R]$ such that $\langle x,y\rangle\in F$. This unique element  $y$ is called the value of the function $F$ at $x$ and is denoted by $F(x)$. For a function $F$ and a set $X$, let $F[X]\defeq\{F(x):x\in X\cap\dom[F]\}$ and $F{\restriction}_X\defeq \{\langle x,y\rangle\in F:x\in X\}$ be the {\em restriction} of the function $F$ to the set $X$. Given a function $F$ and sets $X,Y$, we write $F:X\to Y$ and say that $F$ is a function from $X$ to $Y$ if $\dom[F]=X$ and $\rng[F]\subseteq F$.


A {\em doubleton} is a set $D$ of cardinality $|D|=2$. Given two distinct sets $x,y$ we shall denote the doubleton $\{x,y\}$ by $xy$. 

A {\em graph} is a set $E$ whose elements are doubletons, called {\em edges} of the graph $E$. Elements of the set $\bigcup E$ are called {\em vertices} of a graph $E$. A sequence $x_0,\dots,x_n\in\bigcup E$ is called an {\em $E$-chain} if $\{x_{i-1}x_i:0<i\le n\}\subseteq E$. A graph $E$ is {\em connected} if for every vertices $x,y\in\bigcup E$ there exists an $E$-chain $x_0,x_1,\dots,x_n$ such that $x_0=x$ and $x_n=y$.
\smallskip

 A {\em graph pseudometric} is any function $d:E_d\to\bar\IR_+\defeq\{x\in\IR:x\ge 0\}$ defined on a connected graph $E_d=\dom[d]$ and satisfying the {\em polygonal inequality} $$d(x_0x_n)\le\sum_{i=1}^nd(x_{i-1}x_i)$$ for every $E_d$-chain $x_0,\dots,x_n$ with $x_0x_n\in E_d$. For a graph pseudometric $d$, we denote by $V_d\defeq\bigcup E_d$ the set of vertices of the graph $E_d$. A graph pseudometric $d:E_d\to\bar\IR_+$ is called a {\em graph metric} if $d(e)>0$ for every edge $e\in E_d$.
 
A graph (pseudo)metric $d$ is a {\em full} ({\em pseudo}){\em metric}  if $E_d=[V_d]^2\defeq\{A\subseteq X:|A|=2\}$, i.e., the domain of $d$ is the complete graph on the set $V_d=\bigcup E_d$.

Two graph pseudometrics $p,d$ are {\em isomorphic} if there exists a bijective function $f:V_p\to V_d$ such that $d=\{\langle f[e],r\rangle:\langle e,r\rangle\in p\}$. It is easy to see that isomorphic graph pseudometrics $p,d$ have equal ranges, i.e., $\rng[p]=\rng[d]$. 

Given a graph pseudometric $p$, a point $x\in V_p$, a real number $r$, and a set $D\subseteq \IR$, let $$\Sf_p(x;r)\defeq\{y\in V_p:\langle xy,r\rangle\in p\}\quad \mbox{and}\quad\Sf_p(x;D)\defeq\bigcup_{r\in D}\Sf_p(x;r).$$

Every function $d:E_d\to\bar\IR_+$ on a connected graph $E_d$ determines the {\em shortest-path pseudometric}
$\hat d:V_d\times V_d\to\bar\IR_+$,
$$\hat d(x,y)\defeq\inf\Big\{\sum_{i=1}^nd(x_{i-1}x):\mbox{$x_0,\dots,x_n$ is an $E$-chain with $x_0=x$ and $x_n=y$}\Big\}$$on the set of vertices $V_d\defeq\bigcup E_d$ of the graph $E_d$. The function $d$ is a graph pseudometric if and only if $d(xy)=\hat d(x,y)$ for every $x,y\in V_d$ with $xy\in E_d$.

Given a function $d:E_d\to\bar \IR_+$ on a connected graph $E_d$, consider the function
$$\check d:V_d\times V_d\to\bar\IR_+,\quad \check d(x,y)\defeq\inf\{\max\{0,d(ab)-\hat d(a,u)-\hat d(b,v)\}:ab\in E_d\}.$$For every graph pseudometric $d$, the polygonal inequality for $d$ implies that $\check d\le \hat d$. By \cite{P}, a graph pseudometric $d$ extends to a unique full pseudometric $f$ with $V_f=V_d$ if and only if $\hat d=\check d$. 

A graph pseudmetric $d:E_d\to\bar\IR_+$ is called {\em floppy} if $\check d(x,y)<\hat d(x,y)$ for every $x,y\in V_d$ with $xy\notin E_d$.

The following lemma is proved in \cite{BM} and is a crucial tool in the proof of Theorem~\ref{t:exotic}.

\begin{lemma}\label{l:extend} Let $d:E_d\to\IR_+$ is a floppy graph metric and $(F_{e})_{e\in[V_d]^2\setminus E_d}$ be an indexed family of dense sets in $\IR_+$. If the set $[V_d]^2\setminus E_d$ is at most countable, then there exists an injective function $r\in \prod_{e\in[V_d]^2\setminus E_d}F_e$ such that $d\cup r$ is a  full metric. 
\end{lemma}

The following lemma is proved in \cite[Example 1.2]{BM}.

\begin{lemma}\label{l:floppy-union} Let $p$ be a graph pseudometric and $\F$ be a family of graph pseudometrics satisfying the following conditions:
\begin{enumerate}
\item for every $f\in\F$, the set $V_f\cap V_p$ is not empty;
\item for every $f\in \F$ and $x,y\in V_f\cap V_p$ we have $\hat f(x,y)=\hat p(x,y)$;
\item for every distinct $f,g\in\F$, we have $V_f\cap V_g\subseteq V_p$.
\end{enumerate}
Then the union $d\defeq p\cup\bigcup\F$ is a graph pseudometric. The graph pseudometric $d$ is floppy if  $p$ is a full pseudometric and every $f\in\F$ is a floppy graph pseudometric such that for every $x\in V_f\setminus V_p$ and $y\in V_p\setminus V_f$, the real numbers
$$
\begin{aligned}
&\inf\{\hat f(a,x)+\hat f(x,b)-\hat f(a,b):a,b\in V_f\cap V_p\}\mbox{ \ and \ }\\
&\inf\{p(ay)+p(yb)-p(ab):a,b\in V_p\cap V_f\}
\end{aligned}
$$are positive.
\end{lemma}

We recall that a subset $M\subseteq\bar\IR_+$ is floppy if for every real number $r$ with $r=\inf\{a+b:a,b\in M\wedge r=b-a\}$ we have $r\in M$.

\begin{lemma}\label{l:floppy} For every monoid $M\subseteq\bar\IR_+$, the function  
$$\mu\defeq\{\langle xy,|x-y|\rangle:x,y\in M-M\;\wedge\;0<|x-y|\in M\}$$
is a graph metric with $\mu[E_\mu]=M\setminus\{0\}$. Moreover, for every $x,y\in M-M$,$$\check \mu(x,y)=|x-y|\le \hat \mu(x,y)=\inf\{u+v:u,v\in M\;\wedge\;|x-y|=u-v\}.$$
Consequently, the graph metric $\mu$ is floppy if and only if the monoid $M$ is floppy.
\end{lemma}

\begin{proof} Fix any $x,y\in M-M$. We lose no generality assuming that $y<x$ and hence $|x-y|=x-y$.

\begin{claim}\label{cl:floppy1} $|x-y|\le \hat \mu(x,y)=I\defeq \inf\{u+v:u,v\in M\;\wedge\;|x-y|=u-v\}$.
\end{claim}

\begin{proof} Fix any real numbers $u,v\in M$ with $x-y=|x-y|=u-v$. Then for the $E_\mu$-chain $x=y+u-v,y+u,y$, we obtain the inequality $$\hat\mu(x,y)\le \mu(x,y+u)+\mu(y+u,y)=v+u,$$witnessing that $\hat\mu(x,y)\le I$.

Assuming that $\hat\mu(x,y)<I$, we can  find an $E_\mu$-chain $x_0,x_1,\dots,x_n\in M-M$ such that $x_0=x$, $x_n=y$, $\{|x_{i-1}-x_i|:0<i\le n\}\subseteq M\setminus\{0\}$ and $\sum_{i=1}^n\mu(x_{i-1},x_i)=\sum_{i=1}^n|x_{i-1}-x_i|<I$. Let $\Omega_+=\{i\in \{1,\dots,n\}:x_{i-1}-x_i>0\}$ and $\Omega_-=\{i\in \{1,\dots,n\}:x_{i-1}-x_i<0\}$. Then $\{x_{i-1}-x_i:i\in \Omega_+\}=\{|x_{i-1}-x_i|:i\in \Omega_+\}\subseteq M$ and $\{x_i-x_{i-1}:i\in \Omega_-\}=\{|x_{i-1}-x_i|:i\in\Omega_-\}\subseteq M$. Since $M$ is a monoid, the real numbers $u\defeq\sum_{i\in\Omega_+}(x_{i-1}-x_i)$ and $v\defeq\sum_{i\in\Omega_-}(x_i-x_{i-1})$ belong to $M$. Observe that 
$$u-v=\sum_{i\in \Omega_+}(x_{i-1}-x_i)-\sum_{i\in\Omega_-}(x_i-x_{i-1})=\sum_{i=1}^n(x_{i-1}-x_i)=x_0-x_n=x-y=|x-y|$$and hence 
$$I\le u+v=\sum_{i\in \Omega_+}(x_{i-1}-x_i)+\sum_{i\in\Omega_-}(x_i-x_{i-1})=\sum_{i=1}^n|x_{i-1}-x_i|<I,$$which is a contradiction showing that $\hat \mu(x,y)=I$.

The definition of $I=\inf\{u+v:u,v\in M\;\wedge\;|x-y|=u-v\}$ implies that $|x-y|\le I=\hat \mu(x,y)$. 
\end{proof}

By Claim~\ref{cl:floppy1}, for any points $x,y\in M-M$ with $xy\in E_\mu$, we have $\mu(xy)=|x-y|\le I=\hat\mu(x,y)$, which means that the function $\mu$ is a graph metric. The definition of $\mu$ ensures that $\mu[E_\mu]=M\setminus\{0\}$. 

\begin{claim}\label{cl:floppy2} $\check \mu(x,y)=|x-y|$.
\end{claim}

\begin{proof}  It follows from $x,y\in M-M$ that $|x-y|=x-y\in (M-M)-(M-M)=(M+M)-(M+M)=M-M$ and hence $|x-y|=x-y=u-v$ for some $u,v\in M$. Then for the points $y+u\in M-M$ and $y\in M-M$, we have $\{y+u,y\}\in E_\mu$ and $$\check \mu(x,y)\ge \mu(y+u,y)-\hat \mu(y+u,x)-\hat\mu(y,y)=u-|y+u-x|-0=u-v=x-y=|x-y|.$$ Assuming that $\check\mu(x,y)>|x-y|$, we can find points $a,b\in M-M$ such that $ab\in E_\mu$ and $\mu(ab)-\hat \mu(a,x)-\hat \mu(b,y)>|x-y|$. Then $|a-x|+|b-y|\le \hat \mu(a,x)+\hat \mu(b,y)<\mu(ab)-|x-y|=|a-b|-|x-y|$, which implies a contradiction
$$|x-y|+|a-x|+|b-y|<|a-b|\le |a-x|+|x-y|+|y-b|$$
completing the proof of the equality $\check \mu(x,y)=|x-y|$.
\end{proof}

 It remains to prove that the graph metric $\mu$ is floppy if and only if the monoid $M$ is floppy.
 
 Assume that the graph metric $\mu$ is floppy. Then for every doubleton $xy\in [M-M]^2\setminus E_\mu$ we have $\check \mu(x,y)<\hat \mu(x,y)$. To see that the set $M$ is floppy, we should prove that every real number $r\in M-M$ with $r=\inf\{u+v:u,v\in M\wedge r=u-v\}$ belongs to $M$. Observe that $r\ge 0$. If $r=0$, then $r=0\in M$ since $M$ is a monoid. So, assume that $r>0$. Then for the doubleton $\{0,r\}\in[M-M]^2$, Claims~\ref{cl:floppy2} and \ref{cl:floppy1} and the choice of $r$ ensure that $$\check \mu(0,r)=|0-r|=r=\inf\{u+v:u,v\in M\;\wedge\;|0-r|=u-v\}=\hat\mu(0,r).$$ Taking into account that the graph metric $\mu$ is floppy, we conclude that $\{0,r\}\in E_\mu$ and hence $r=|r-0|\in M$, witnessing that the monoid $M$ is floppy.
\smallskip

Next, assume that the monoid $M$ is floppy. To see that the graph metric $\mu$ is floppy, take any doubleton $xy\in[M-M]^2\setminus E_\mu$. Then $|x-y|\in (M-M)\setminus M$. Since the monoid $M$ is floppy, $|x-y|<\inf\{u+v:u,v\in M\;\wedge\;|x-y|=u-v\}$. Claims~\ref{cl:floppy1} and \ref{cl:floppy2} imply the strict inequality $$\check \mu(x,y)=|x-y|<\inf\{u+v:u,v\in M\;\wedge\;|x-y|=u-v\}=\hat\mu(x,y),$$ witnessing that the graph metric $\mu$ is floppy.
\end{proof}

We recall that for a subset $M\subseteq\bar\IR_+$,
$$\ddot M\defeq\{r\in M\setminus\{0\}:\forall x,y\in (M\cap \IQ r)\setminus\{0\}\;(x+y\ne r)\}.$$

\begin{lemma}\label{l:Lambda} Let $M$ be a subset of $\bar\IR_+$ and $r\in \ddot M$ be a number such that $M_r\defeq M\cap\IQ r$ is a floppy monoid. Let $\mu_r\defeq\{\langle xy,z\rangle:(x,y\in M_r-M_r)\;\wedge\;(0<|x-y|=z\in M_r)\}.$ Then for every $a,b\in M_r-M_r$ with $|a-b|=r$ and every $c\in (M_r-M_r)\setminus\{a,b\}$ we have
$$\min\{2\hat \mu_r(c,a),2\hat\mu_r(c,b),\hat \mu_r(c,a)+\hat\mu_r(c,b)-\hat\mu_r(a,b)\}>0.$$
\end{lemma}

\begin{proof} By Lemma~\ref{l:floppy}, $\mu_r$ is a floppy graph metric and $\hat \mu_r$ is a metric on the set $M_r-M_r\subseteq\IR$, which implies that $\min\{\hat\mu_r (c,a),\hat\mu_r(c,b)\}>0$. The triangle inequality for the  metric $\hat \mu_r$ ensures that $\hat\mu_r(c,a)+\hat\mu_r(c,b)-\hat \mu_r(a,b)\ge 0$. Assuming that $\hat\mu_r(c,a)+\hat\mu_r(c,b)-\hat \mu_r(a,b)=0$ and applying Lemma~\ref{l:floppy}, we conclude that 
$$r=|a-b|=\mu_r(a,b)=\hat\mu_r(a,b)=\hat\mu_r(c,a)+\hat\mu_r(c,b)\ge |c-a|+|c-b|\ge|a-b|=r$$and hence $r=|c-a|+|c-b|$, $\hat \mu_r(c,a)=|c-a|=\check\mu_r(c,a)$ and $\hat \mu_r(c,b)=|c-b|=\check\mu_r(c,a)$, which implies $\{ac,cb\}\subseteq E_{\mu_r}$ and $\{|c-a|,|c-b|\}\subseteq M_r$ as $\mu_r$ is a floppy graph metric. But the equality $r=|c-a|+|c-b|$ with $\{|c-a|,|c-b|\}\subseteq M_r\setminus\{0\}$ contradicts the choice of $r\in \ddot M$.
\end{proof}

Now we are able to present the {\em\bf proof of Theorem~\ref{t:exotic}}. Given a subset  $M\subseteq\bar\IR_+$ with the properties
\begin{itemize}
\item[(a)] $|M|\le\w_1$;
\item[(b)] for every $r\in M$ the set $M_r\defeq M\cap \IQ r$ is a floppy monoid, and
\item[(c)] for every subset $B\subseteq M$ of cardinality $|B|<|M|$ the set $\ddot M\setminus(\IQ B+\IQ B)$ is dense in $\IR_+$, 
\end{itemize}
we will construct a Banakh space $X$ with $\dist[X^2]=M$. 

The conditions (a) and (c) imply that  $\kappa\defeq |M|\in\{\w,\w_1\}$. Write the set $M\setminus\{0\}$ as the union $M\setminus\{0\}=\bigcup_{\alpha\in\kappa}B_\alpha$ of an increasing sequence of subsets of cardinality $|B_\alpha|<\kappa$ such that $B_0=\{r_0\}$ for some positive real numer $r_0$ and $B_\beta=\bigcup_{\alpha<\beta}B_\alpha$ for every limit nonzero ordinal $\beta<\kappa$.

\begin{lemma}\label{l:Falpha} There exists a family $(F_\alpha)_{\alpha\in \kappa}$ of pairwise disjoint subsets of $\ddot M$ such that for every $\alpha\in\kappa$ the following conditions hold:
\begin{enumerate}
\item the set $F_\alpha$ is dense in $\IR_+$;
\item for every $x\in F_\alpha$ and $y,z\in (F_\alpha\setminus \{x\})\cup \IQ B_\alpha\cup\bigcup_{\gamma<\alpha}\IQ F_\gamma$ we have $x\notin \IQ y+\IQ z$;
\item for every real number $r$, the set $F_\alpha\cap \IQ r\setminus\{0\}$ contains at most one element.
\end{enumerate}
\end{lemma}

\begin{proof} Let $\mathcal U$ be a countable base of the topology of $\IR_+$ that consists of nonempty open subsets of $\IR_+$. Let $\xi:\kappa\to \kappa\times\U$ be any bijective function. There exist functions $\xi_1:\kappa\to\kappa$ and $\xi_2:\kappa\to\mathcal U$ such that $\xi(\alpha)=\langle\xi_1(\alpha),\xi_2(\alpha)\rangle$ for every $\alpha\in\kappa$. For every $\alpha\in\kappa$, let $U_\alpha\defeq\xi_2(\alpha)\in\mathcal U$. Then $\{U_\alpha\}_{\alpha\in\kappa}$ is an enumeration of the set $\mathcal U$ such that for every $U\in\U$ the set $\{\alpha\in\kappa:U_\alpha=U\}$ has cardinality $\kappa$.

Construct inductively  a sequence of real numbers $(r_\alpha)_{\alpha\in\kappa}$ such that for every $\beta\in\kappa$,
$$\textstyle r_\beta\in U_\beta\cap \ddot M\setminus\bigcup\{\IQ r+\IQ s:r,s\in \{r_\alpha:\alpha<\beta\}\cup\bigcup_{\alpha\le \beta}B_{\xi_1(\beta)}\}.$$
The property (c) of the set $M$ ensures that for every ordinal $\beta\in\kappa$, the choice of the number $r_\beta$ is always possible.

For every $\alpha\in\kappa$, consider the set $$F_\alpha\defeq\{r_\beta:\beta\in\xi_1^{-1}(\alpha)\}\subseteq \ddot M.$$ The following claim establishes the condition (1) of Lemma~\ref{l:Falpha}.

\begin{claim} For every ordinal $\alpha\in\kappa$, the set $F_\alpha$ in dense in $\IR_+$.
\end{claim}

\begin{proof} Given any basic open set $U\in\U$, find a unique ordinal $\beta\in\kappa$ such that $\xi(\beta)=\langle \alpha,U\rangle$. Then $\beta\in\xi_1^{-1}(\alpha)$ and $r_\beta\in F_\alpha\cap U_\beta=F_\alpha\cap U$, witnessing that the set $F_\alpha$ is dense in $\IR_+$.
\end{proof}

\begin{claim}\label{cl:xy} For every ordinal $\alpha\in\kappa$ and real numbers  $x\in F_\alpha$ and $y\in (F_\alpha\setminus\{x\})\cup \IQ B_\alpha\cup\bigcup_{\gamma<\alpha}\IQ F_\gamma$, we have $x\notin \IQ y$. 
\end{claim}

\begin{proof} If $y=0$, then $x\notin \IQ y=\{0\}$ as $x\in F_\alpha\subseteq\ddot M\subseteq M\setminus\{0\}$. So, we assume that $y\ne 0$. 
Let  $\beta\in\xi_1^{-1}(\alpha)$ be a unique ordinal such that $x=r_\beta$.  Depending on the location of the point $y$ in the set  $(F_\alpha\setminus\{x\})\cup \IQ B_\alpha\cup\bigcup_{\gamma<\alpha}\IQ F_\gamma$, we consider two cases.
\smallskip

1. First we assume that $y\in (F_\alpha\setminus \{x\})\cup\bigcup_{\gamma<\alpha}\IQ F_\gamma$. In this case we can find an ordinal $\gamma\in\kappa\setminus\{\beta\}$ such that $y\in\IQ r_\gamma$. Then $\IQ y=\IQ r_\gamma$.

If $\beta>\gamma$, then the choice of $r_\beta\notin \IQ r_\gamma$ implies $x=r_\beta\notin\IQ r_\gamma=\IQ y$. 

If $\gamma>\beta$, then $r_\gamma\notin \IQ r_\beta$ and hence $x=r_\beta\notin \IQ r_\gamma=\IQ y$. 
\smallskip

2. Next, assume that $y\in \IQ  B_\alpha=\IQ B_{\xi_1(\beta)}$. In this case $\IQ y\subseteq \IQ B_{\xi_1(\beta)}$. The choice of $r_\beta\notin \IQ B_{\xi_1(\beta)}$  ensures that $x=r_\beta\notin \IQ y$.
\end{proof}

The following claim proves the condition (2) of Lemma~\ref{l:Falpha}. 

\begin{claim} For every ordinal $\alpha\in\kappa$ and real numbers  $x\in F_\alpha$ and $y,z\in (F_\alpha\setminus\{x\})\cup \IQ B_\alpha\cup\bigcup_{\gamma<\alpha}\IQ F_\gamma$, we have $x\notin \IQ y+\IQ z$. 
\end{claim}

\begin{proof}  By Claim~\ref{cl:xy}, $x\notin \IQ y\cup\IQ z$. Then $x\not\in \IQ y+\IQ z$  if $y=0$ or $z=0$. So, we can (and will) assume  that $y\ne 0\ne z$. 
Let  $\beta\in\xi_1^{-1}(\alpha)$ be a unique ordinal such that $x=r_\beta$.  Depending on the location of the points $y,z$ in the set  $(F_\alpha\setminus\{x\})\cup \IQ B_\alpha\cup\bigcup_{\gamma<\alpha}\IQ F_\gamma$, we consider four cases.
\smallskip

1. First we assume that $y,z\in (F_\alpha\setminus \{x\})\cup\bigcup_{\gamma<\alpha}\IQ F_\gamma$. In this case we can find ordinals $\gamma,\delta\in\kappa\setminus\{\beta\}$ such that $y\in\IQ r_\gamma$ and $z\in\IQ r_\delta$. Then $\IQ y\subseteq\IQ r_\gamma$ and $\IQ z \subseteq \IQ r_\delta$.

If $\beta>\max\{\gamma,\delta\}$, then the choice of $r_\beta\notin \IQ r_\gamma+\IQ r_\delta$ implies $x=r_\beta\notin\IQ y+\IQ z$. 

If $\gamma>\max\{\beta,\delta\}$, then $r_\gamma\notin \IQ r_\beta+\IQ r_\delta$. Assuming that $x\in \IQ y+\IQ z$ and taking into account that $x\notin \IQ z$, we conclude that $y\in \IQ x+\IQ z$ and $r_\gamma\in \IQ y\subseteq\IQ r_\beta+\IQ 
r_\delta$, which contradicts the choice of $r_\gamma$.

If $\delta>\max\{\beta,\gamma\}$, then we can prove that $x\notin \IQ y+\IQ z$ by analogy with the case $\gamma>\max\{\beta,\delta\}$.

If $\gamma=\delta>\beta$, then $r_\gamma\notin \IQ r_\beta$ implies $r_\beta\notin \IQ r_\gamma=\IQ r_\gamma+\IQ r_\delta=\IQ y+\IQ z$. 
\smallskip

2. Next, assume that $y\in (F_\alpha\setminus \{x\})\cup\bigcup_{\gamma<\alpha}\IQ F_\gamma$ and $z\in \IQ  B_\alpha=\IQ B_{\xi_1(\beta)}$. Find an ordinal $\gamma\in\kappa\setminus\{\beta\}$ such that $y\in \IQ r_\gamma$.

If $\beta>\gamma$, then the choice of $r_\beta\notin \IQ r_\gamma+\IQ B_{\xi_1(\beta)}$ ensures that $x=r_\beta\notin \IQ y+\IQ z$.

If $\gamma>\beta$, then $r_\gamma\notin \IQ r_\beta+\IQ B_{\xi_1(\beta)}$, by the choice of $r_\gamma$. Assuming that $x\in \IQ y+\IQ z$ and taking into account that $x\notin\IQ z$, we conclude that $y\in \IQ x+\IQ z$ and hence $r_\gamma\in \IQ y\subseteq\IQ x+\IQ z\subseteq\IQ r_\beta+\IQ B_{\xi_1(\beta)}$, which contradicts the choice of $\gamma$. 
\smallskip

3. The case $y\in \IQ  B_\alpha$ and $y\in (F_\alpha\setminus \{x\})\cup\bigcup_{\gamma<\alpha}\IQ F_\gamma$ can be considered by analogy with the case 2.

4. If $y,z\in \IQ  B_\alpha$, then the choice of $r_\beta\notin \IQ B_{\xi_1(\beta)}+\IQ B_{\xi_1(\beta)}=\IQ B_\alpha+\IQ B_\alpha$ implies that $x=r_\beta\notin \IQ y+\IQ z$. 
\end{proof}

To prove the last statement of Lemma~\ref{l:Falpha}, take any real number $r$. Assuming that the set $F_\alpha\cap\IQ r\setminus\{0\}$, contains two distinct points $x,y$, we conclude that $x\in\IQ y$, which contradicts Claim~\ref{cl:xy}.
\end{proof}

For every $r\in\IR_+$ let $M_r^\circ\defeq M_r\setminus\{0\}=M\cap\IQ r\setminus\{0\}$.

\begin{lemma}\label{l:step} There exist sequences $(f_\alpha)_{\alpha\in\kappa}$ and $(g_\alpha)_{\alpha\in\kappa}$ of graph metrics such that for every ordinal $\beta\in\kappa$, the following conditions are satisfied:
\begin{enumerate}
 \item The sets $f_\beta, g_\beta$ are countable.
 \item $g_\beta$ is a floppy graph metric such that $\bigcup_{\alpha<\beta} f_\alpha\subseteq g_\beta$.
\item $f_\beta$ is a full metric such that $g_\beta\subseteq f_\beta$ and $V_{f_\beta}=V_{g_\beta}$.
\item $f_\beta[E_{f_\beta}\setminus E_{g_\beta}]\subseteq F_\beta$ and the function $f_\beta{\restriction}_{E_{f_\beta}\setminus E_{g_\beta}}$ is injective.
\item $g_\beta[E_{g_\beta}]\subseteq  \bigcup _{r\in B_\beta}M^\circ_r\cup \bigcup_{\gamma\in\beta}\bigcup_{r\in F_\gamma}M^\circ_r$.
\item For every $x\in \bigcup_{\alpha<\beta}V_{f_\alpha}$ and $r\in \bigcup_{\alpha<\beta}\bigcup_{s\in B_\alpha}M^\circ_s$ we have $|\Sf_{g_\beta}(x;r)|\ge 2$.
\item  $\forall r\in \IR_+\;\forall x,y\in V_{f_\beta}\;\;\big(\Sf_{f_\beta}(x;M_r)=\{x,y\}\Leftrightarrow \Sf_{f_\beta}(y;M_r)=\{y,x\}\big)$.
\item For every $x\in V_{f_\beta}$ and $r\in\IR_+$ we have $|\Sf_{f_\beta}(x;r)|\le 2$ and $f_\beta[[\Sf_{f_\beta}(x;r)]^2]\subseteq\{2r\}$.

\item If $\beta$ is a limit ordinal, then $g_\beta=f_\beta=\bigcup_{\alpha\in\beta}f_\alpha=\bigcup_{\alpha\in\beta}g_\alpha$.
\item If $\beta=\alpha+1$ is a successor ordinal, then $g_{\beta}=f_\alpha\cup\bigcup_{\langle x,r\rangle \in P_\beta}p_{x,r}$ where\\ $P_{\beta}\defeq\{\langle x,r\rangle:x\in V_{f_\alpha}\times\big(f_\alpha[E_{f_{\alpha}}]\cup\bigcup_{s\in B_\beta}M_s^\circ\big):|\Sf_{f_\alpha}(x;r)|=|\Sf_{f_\alpha}(x;M_r^\circ)|\le 1\}$ and $(p_{x,r})_{\langle x,r\rangle\in P_\beta}$ is a sequence of floppy graph metrics such that for every $\langle x,r\rangle,\langle y,s\rangle\in P_\beta$ the following conditions are satisfied:
\begin{itemize}
\item[(i)] the graph metric $p_{x,r}$ is isomorphic to the graph metric\\ $\mu_r\defeq\{\langle xy,|x-y|\rangle:(x,y\in M_r-M_r)\;\wedge\; (0<|x-y|\in M_r)\}$;
\item[(ii)] $V_{p_{x,r}}\cap V_{f_\alpha}=\Sf_{f_\alpha}(x;M_r)$;
\item[(iii)] if $M_r=M_s$ and $\Sf_{f_\alpha}(x;M_r)=\Sf_{f_\alpha}(y;M_s)$, then $p_{x,r}=p_{y,s}$;
\item[(iv)] if $p_{x,r}\ne p_{y,s}$, then $V_{p_{x,r}}\cap V_{p_{y,s}}=\Sf_{f_\alpha}(x;M_r)\cap\Sf_{f_\alpha}(y;M_s)$.
\end{itemize}
\end{enumerate}
\end{lemma}

\begin{proof} For every real number $r$, consider the set $M_r\defeq M\cap\IQ r$ and the graph metric $\mu_r\defeq\{\langle xy,|x-y|\rangle:(x,y\in M_r-M_r)\;\wedge\; (0<|x-y|\in M_r)\}$. By the property (b) of the set $M$, the monoid $M_r$ is floppy and hence the graph metric $\mu_r$ is floppy, according to Lemma~\ref{l:floppy}.

We start the inductive construction letting $g_0=\mu_{r_0}$ where $r_0$ is the unique element of the set $B_0$. Then $V_{g_0}=M_{r_0}-M_{r_0}\ne\emptyset$ and $g_0[E_{g_0}]=\mu_{r_0}[E_{\mu_{r_0}}]=M^\circ_{r_0}$. By Lemma~\ref{l:extend}, there exists a full metric $f_0$ such that $g_0\subseteq f_0$, $V_{f_0}=V_{g_0}$, $f_0[E_{f_0}\setminus E_{g_0}]\subseteq F_0$ and the function $f_0{\restriction}_{E_{f_0}\setminus E_{g_0}}$ is injective. Therefore, the graph metrics $g_0$ and $f_0$ satisfy the inductive conditions (1)--(6). The condition (7) for the full metric $f_0$ is established in the following claim.

\begin{claim} $\forall r\in\IR_+\;\forall x,y\in V_{f_0}\;\;(\Sf_{f_0}(x;M_r)=\{x,y\}\Leftrightarrow \Sf_{f_0}(y;M_r)=\{x,y\})$.
\end{claim}

\begin{proof} Fix any $r\in\IR_+$ and $x,y\in V_{f_0}=V_{g_0}=M_{r_0}-M_{r_0}$. The equivalence  $\Sf_{f_0}(x;M_r)=\{x,y\}\Leftrightarrow \Sf_{f_0}(y;M_r)=\{x,y\}$ is trivial if $x=y$. So, we assume that $x\ne y$. 

Suppose that $\Sf_{f_0}(x;M_r)=\{x,y\}$. If $xy\in E_{g_0}$, then $\{x\}\cup \Sf_{g_0}(x;g_0(xy))\subseteq \Sf_{g_0}(x;M_r)\subseteq \Sf_{f_0}(x;M_r)$ implies  $|\Sf_{f_0}(x;M_r)|\ge 1+|\Sf_{g_0}(x;g_0(xy))|\ge 3$ and hence $\Sf_{f_0}(x;M_r)\ne\{x,y\}$, which contradicts our assumption. This contradiction shows that $xy\notin E_{g_0}$ and hence $f_0(xy)\in f_0[E_{f_0}\setminus E_{g_0}]\subseteq F_0$. The condition (2) of Lemma~\ref{l:Falpha} ensures that $F_0\cap M_{r_0}\subseteq F_0\cap \IQ r_0=\emptyset$ and hence $f_0(xy)\in M_r\setminus M_{r_0}$ and $M_r\cap M_{r_0}=\{0\}$.  It is clear that $\{y,x\}\subseteq \Sf_{f_0}(y;M_r)$. Assuming that $\Sf_{f_0}(y;M_r)\ne\{x,y\}$, we can find a point $z\in \Sf_{f_0}(y;M_r)\setminus\{x,y\}$. It follows from $f_0(yz)\in M^\circ_r$ and $M^\circ_r\cap M_{r_0}=\emptyset$ that $yz\notin E_{g_0}$ and hence $f_0(yz)\in F_0\setminus\{f_0(yx)\}$ by the injectivity of $f_0{\restriction}_{E_{f_0}\setminus E_{g_0}}$.  The condition (2) of Lemma~\ref{l:Falpha} ensures that $M_r=M_{f_0(yz)}\ne M_{f_0(xy)}=M_r$, which is a contradiction witnessing that $\Sf_{f_0}(y;M_r)=\{x,y\}$. By analogy we can prove that $\Sf_{f_0}(y;M_r)=\{x,y\}$ implies $\Sf_{f_0}(x;M_r)=\{x,y\}$.
\end{proof}

In the following claim we show that the full metric $f_0$ satisfies the inductive condition (8).

\begin{claim}  For every $x\in V_{f_0}$ and $r\in\IR_+$ we have $|\Sf_{f_0}(x;r)|\le 2$ and $f_0[[\Sf_{f_0}(x;r)]^2]\subseteq\{2r\}$.
\end{claim}

\begin{proof} If $r\in M_{r_0}$, then $|\Sf_{g_0}(c;r)|=|\Sf_{\mu_{r_0}}(c;r)|=2$ by the definition of the graph metric $\mu_{r_0}=g_0$. The condition (2) of Lemma~\ref{l:Falpha} ensures that $F_0\cap M_{r_0}\subseteq F_0\cap \IQ r_0=\emptyset$. Then $f_0[E_{f_0}\setminus E_{g_0}]\subseteq F_0\setminus M_{r_0}\subseteq F_0\setminus\{r\}$ implies that $\Sf_{f_0}(x;r)=\Sf_{g_0}(x;r)$ and hence $|\Sf_{f_0}(x;r)|=|\Sf_{g_0}(x;r)|=2$ and $f_0[[\Sf_{f_0}(x;r)]^2]=g_0[[S_{g_0}(x;r)]^2]=\{2r\}$. 
If $r\notin M_{r_0}$, then for every $y\in \Sf_{f_0}(x;r)$ the doubleton $\{x,y\}$ does not belong to $E_{g_0}$ and the injectivity of the restriction $f_0{\restriction}_{E_{f_0}\setminus E_{g_0}}$ ensures that $|\Sf_{f_0}(x;r)|\le 1$. In this case $f_0[[\Sf_{f_0}(x;r)]^2]=\emptyset\subseteq\{2r\}$.
\end{proof}
Therefore, the functions $f_0$ and $g_0$ satisfy the inductive conditions (1)--(10).
\smallskip

Assume that for some nonzero ordinal $\beta<\kappa$ we have constructed sequences $(f_\alpha)_{\alpha\in\beta}$ and $(g_\alpha)_{\alpha\in \beta}$  such that the conditions (1)--(10) are satisfied.

If $\beta$ is a limit ordinal, then let $f_\alpha\defeq \bigcup_{\alpha\in\beta}f_\alpha$ and $g_\beta\defeq\bigcup_{\alpha\in\beta}g_\alpha$. The inductive conditions (2), (3) and the equality $B_\beta=\bigcup_{\alpha<\beta}B_\alpha$ ensure that $f_\alpha=g_\alpha$ is a well-defined full metric satisfying the inductive conditions (1)--(10).

Next, assume that $\beta$ is a successor ordinal and find a unique ordinal $\alpha$ such that $\beta=\alpha+1$. Consider the set
 $$P_{\beta}\defeq\{\langle x,r\rangle:x\in V_{f_\alpha}\times\big(f_\alpha[E_{f_\alpha}]\cup\bigcup_{s\in B_\beta}M_s^\circ\big):|\Sf_{f_\alpha}(x;r)|=|\Sf_{f_\alpha}(x;M^\circ_r)|\le 1\}.$$ 
  Choose an indexed family of graph metrics $(p_{x,r})_{\langle x,r\rangle\in P_\beta}$ satisfying the inductive conditions (10i)--(10iv). Lemmas~\ref{l:floppy-union} and \ref{l:floppy} imply that $g_\beta\defeq f_\alpha\cup\bigcup_{\langle x,r\rangle\in P_\beta}p_{x,r}$ is a graph metric.
 The countability of the sets $f_\alpha$ and $B_\alpha$ implies the countability of the sets $P_\beta$ and $g_\beta$.
 
\begin{claim}\label{cl:g-floppy} The graph metric $g_\beta$ is floppy.
\end{claim}

\begin{proof} By Lemma~\ref{l:floppy-union}, it suffices to check that for every $\langle x,r\rangle\in P_\beta$ and vertices $v\in V_{p_{x,r}}\setminus V_{f_\alpha}$ and $u\in V_{f_\alpha}\setminus V_{p_{x,r}}$ the real numbers
$$
\begin{aligned}
&\Lambda(v)\defeq\inf\{\hat p_{x,r}(a,v)+\hat p_{x,r}(y,b)-\hat p_{x,r}(a,b):a,b\in V_{p_{x,r}}\cap V_{f_\alpha}\}\mbox{ \ and}\\
&\Lambda(u)\defeq\inf\{f_\alpha(au)+f_\alpha(ub)-f_\alpha(ab):a,b\in V_{f_\alpha}\cap V_{p_{x,r}}\}
\end{aligned}
$$
are positive.  The inclusion $\langle x,r\rangle\in P_\beta$ implies that $|\Sf_{f_\alpha}(x;r)|=|\Sf_{f_\alpha}(x;M^\circ_r)|\le 1$ and the condition (10ii) implies that $V_{p_{x,r}}\cap V_{f_\alpha}=\Sf_{f_\alpha}(x;M_r)=\{x\}\cup\Sf_{f_\alpha}(x;M_r^\circ)$ and hence $1\le|V_{p_{x,r}}\cap V_{f_\alpha}|\le 2$. If $|V_{p_{x,r}}\cap V_{f_\alpha}|=1$, then $\Lambda(v)=2\hat p_{x,r}(x,v)>0$ and $\Lambda(u)=2 f_\alpha(xu)>0$ because $\hat p_{x,r}$ is a metric, $f_\alpha$ is a full metrics and $u\ne x\ne v$.

It remains to consider the case when $V_{p_{x,r}}\cap V_{f_\alpha}=\Sf_{f_\alpha}(x;M_r)=\{0\}\cup\Sf_{f_\alpha}(x;r)=\{x,y\}$ where $y$ is the unique point of the sphere $\Sf_{f_\alpha}(x;r)$. Let $\gamma\le \alpha$ be the smallest ordinal such that $xy\in E_{f_\gamma}$. Observe that $\Sf_{f_\gamma}(x;r)=\Sf_{f_\gamma}(x;M^\circ_r)=\{y\}$ and hence $\langle x,r\rangle \in P_{\gamma+1}$.  Assuming that $\gamma<\alpha$ and applying the inductive condition (10i), we conclude that $|\Sf_{f_\alpha}(x;r)|\ge |\Sf_{g_{\gamma+1}}(x;r)|\ge |\Sf_{p_{x,r}}(x;r)|\ge 2$ and hence $|\Sf_{f_\alpha}(x;M^\circ _r)|\ge |\Sf_{f_\alpha}(x;r)|\ge 2$, which contradicts the inclusion $\langle x,r\rangle\in P_\beta$. This contradiction shows that $\gamma=\alpha$.

\begin{claim}\label{cl:notEg} $xy\notin E_{g_\alpha}$.
\end{claim}

\begin{proof} To derive a contradiction, assume that $xy\in E_{g_\alpha}$.
If $\alpha=0$, then $xy\in E_{g_0}$ and $g_0\subseteq f_\alpha$ imply $r=f_0(xy)=g_0(xy)\in M_{r_0}$ and $\Sf_{g_0}(x;r)\subseteq \Sf_{f_\alpha}(x;r)=\{y\}$, which contradicts the choice of the graph metric $g_0=\mu_{r_0}$ whose nonempty spheres contain exactly two points. 
This contradiction shows that $\alpha>0$. The minimality of $\gamma$ and the inductive condition (9) ensure that the nonzero ordinal $\gamma=\alpha$ is successor and hence $\alpha=\gamma=\delta+1$ for some ordinal $\delta$. The minimality of $\gamma$ ensures that $xy\notin E_{f_\delta}$. Then $xy\in E_{g_{\alpha}}\setminus E_{f_\delta}=\bigcup_{\langle z,s\rangle\in P_\alpha}E_{p_{z,s}}$ and hence $xy\in E_{p_{z,s}}$ for some $\langle z,s\rangle\in P_\alpha$. It follows from $p_{z,s}\subseteq g_\alpha\subseteq f_\alpha$ and $xy\in E_{p_{z,s}}$ that $r=f_\alpha(xy)=p_{z,s}(xy)\in M^\circ_s$ and hence $M^\circ_s=M^\circ_r$. By the inductive condition (10i), the graph metric $p_{z,s}$ is isomorphic to the graph metric $\mu_s$, which implies $\w\le |\Sf_{p_{z,s}}(x;M^\circ_s)|=|\Sf_{p_{z,s}}(x;M^\circ_r)|\le  |\Sf_{f_\alpha}(x;M^\circ_r)|$ and contradicts $|\Sf_{f_\alpha}(x;M^\circ_r)|=1$. This contradiction completes the proof of Claim~\ref{cl:notEg}.
\end{proof}

Claim~\ref{cl:notEg} implies that $xy\in E_{f_\alpha}\setminus E_{g_\alpha}$ and hence $r=f_\alpha(xy)\in F_\alpha\subseteq \ddot M$, by the inductive condition (4) and the choice of $F_\alpha\subseteq\ddot M$. Since the graph metric $p_{x,r}$ is isomorphic to the graph metric $\mu_r$, Lemma~\ref{l:Lambda} implies that $\Lambda(v)>0$ for every $v\in V_{p_{x,r}}\setminus V_{f_\alpha}=V_{p_{x,r}}\setminus\{x,y\}$.

\begin{claim} For every $u\in V_{f_\alpha}\setminus V_{p_{x,r}}=V_{f_\alpha}\setminus\{x,y\}$ we have
$$\Lambda(u)=\min\{2f_\alpha(ux),2f_\alpha(uy),f_\alpha(xu)+f_\alpha(uy)-f_\alpha(xy)\}>0.$$
\end{claim}

\begin{proof} Fix any point $u\in V_{f_\alpha}\setminus\{x,y\}$. Since $f_\alpha$ is a graph metric, $\min\{f_\alpha(ux),f_\alpha(uy)\}>0$. It remains to prove that $f_\alpha(xu)+f_\alpha(uy)-f_\alpha(xy)>0$. The triangle inequality for the full metric $f_\alpha$ implies $f_\alpha(xu)+f_\alpha(uy)-f_\alpha(xy)\ge 0$. 
By the inductive conditions (4) and (5), 
\begin{multline*}
\{f_{\alpha}(xu),f_\alpha(uy)\}\subseteq (F_\alpha\setminus \{f_\alpha(xy)\})\cup g_\alpha[E_{g_\alpha}]\\
\subseteq (F_\alpha\setminus\{r\})\cup\bigcup_{s\in B_\alpha}M^\circ_s\cup\bigcup_{\gamma<\alpha}\bigcup_{s\in F_\gamma}M^\circ_s\subseteq (F_\alpha\setminus\{r\})\cup \IQ B_\alpha\cup\bigcup_{\gamma<\alpha}\IQ F_\gamma.
\end{multline*}
 The condition (2) of Lemma~\ref{l:Falpha} ensures that $r=f_\alpha(xy)\ne f_\alpha(xu)+f_\alpha(uy)$ and hence $f_\alpha(xu)+f_\alpha(yu)-f_\alpha(xy)>0$.
\end{proof}
By Lemma~\ref{l:floppy-union}, the graph metric $g_\beta$ is floppy, which completes the proof of Claim~\ref{cl:g-floppy}.
\end{proof}
Applying Lemma~\ref{l:extend}, extend to floppy graph metric $g_\beta$ to a full metric $f_\beta$ that satisfies the inductive conditions (3) and (4). The countablity of the set $g_\beta$ implies the countability of the set $E_{g_\beta}$. The connecteness of the graph $E_{g_\beta}$ implies the countability of the set $V_{g_\beta}=V_{f_\beta}$ and the countability of the set $f_\beta$. Therefore, the graph metrics $g_\beta$ and $f_\beta$ satisfy the inductive conditions (1)--(4).

 In the following claim we prove that the graph metric $g_\beta$ satisfies the inductive condition (5).

\begin{claim}\label{cl:15.15} $g_\beta[E_{g_\beta}]\subseteq  \bigcup _{r\in B_\beta}M^\circ_r\cup \bigcup_{\gamma\in\beta}\bigcup_{r\in F_\gamma}M^\circ_r$.
\end{claim}

\begin{proof} 
%
The inductive conditions (10), (4) and (5) ensure that
$$
\begin{aligned}
g_\beta[E_{g_\beta}]&=f_{\alpha}[E_{f_\alpha}]\cup\bigcup_{\langle x,r\rangle\in P_\beta}p_{x,r}[E_{p_{x,r}}]\subseteq f_\alpha[E_{f_\alpha}\setminus E_{g_\alpha}] \cup g_\alpha[E_{g_\alpha}]\cup\bigcup_{r\in f_\alpha[E_{f_\alpha}]\cup B_\beta}M^\circ_r\\
&\subseteq F_\alpha\cup g_\alpha[E_{g_\alpha}]\cup\bigcup_{r\in F_\alpha\cup g_\alpha[E_{g_\alpha}]\cup B_\beta}M^\circ_r\subseteq \bigcup_{r\in F_\alpha\cup B_\beta}M_r^\circ\cup g_\alpha[E_{g_\alpha}]\cup\bigcup_{r\in g_\alpha[E_{g_\alpha}]}M_r^\circ\\
&\subseteq \bigcup_{r\in F_\alpha\cup B_\beta}M^\circ_r\cup\Big( \bigcup_{r\in B_\alpha}M^\circ_r\cup \bigcup_{\gamma<\alpha}\bigcup_{r\in F_\gamma}M^\circ_r\Big)\cup \bigcup_{r\in g_\alpha[E_{g_\alpha}]}M_r^\circ\\&=\bigcup_{r\in B_\beta}M^\circ_r\cup\bigcup_{\gamma<\beta}\bigcup_{r\in F_\gamma}M^\circ_r.
\end{aligned}
$$
\end{proof}

In the following claim we prove that the function $g_\beta$ satisfies the inductive condition (6).

\begin{claim}\label{cl:15.14b} For all $x\in V_{f_\alpha}$ and $r\in \bigcup_{s\in B_\beta}M_s^\circ$ we have $|\Sf_{g_\beta}(x,r)|\ge 2$.
\end{claim}

\begin{proof} If $\Sf_{f_\alpha}(x;M_r^\circ)=\emptyset$, then the pair $\langle x,r\rangle$ belongs to the set $P_\beta$. If $\Sf_{f_\alpha}(x;M_r^\circ)=\{y\}$ for some $y\in V_{f_\alpha}$, then for the real number $t\defeq f_\alpha(xy)\in M_r^\circ$ we have $\Sf_{f_\alpha}(x;t)=\Sf_{f_\alpha}(x;M_t^\circ)=\{y\}$ and $\langle x,t\rangle\in P_\beta$. In both cases we can find a real number $t\in M_r^\circ$ such that $\langle x,t\rangle\in P_\beta$. In this case the inductive condition (10i) implies $\Sf_{p_{x,t}}(x;r)\subseteq \Sf_{g_\beta}(x;r)$ and hence $2=|\Sf_{p_{x,t}}(x;r)|\le|\Sf_{g_\beta}(x;r)|$ because the graph metric $p_{x,t}$ is isomorphic to the graph metric $\mu_t$ and $r\in M_t$.

It remains to consider the case of $|\Sf_{f_\alpha}(x;M_r^\circ)|\ge 2$. Let $\gamma\le\alpha$ be the smallest ordinal such that $|\Sf_{f_\gamma}(x;M_r^\circ)|\ge 2$.

If $\gamma=0$, then $\Sf_{f_0}(x;M_r^\circ)\ne\emptyset$ implies $x\in V_{f_0}=V_{g_0}$. If $r\in M_{r_0}=g_0[E_{g_0}]$, then $|\Sf_{g_0}(x;r)|=|\Sf_{\mu_{r_0}}(x;r)|=2$ and we are done. So, we assume that $r\notin M_{r_0}$ and hence $M_r\cap g_0[E_{g_0}]=M_r\cap M_{r_0}^\circ=\emptyset$. Since $|\Sf_{f_0}(x;M_r^\circ)|\ge 2$, there exist distinct points $x,y\in \Sf_{f_0}(x;M_r^\circ)$. Then $\{f_0(xz),f_0(yz)\}\subseteq M_r^\circ \setminus g_0[E_{g_0}]$ and hence $\{xz,yz\}\subseteq E_{f_0}\setminus E_{g_0}$. The inductive condition (4) implies that $f_0(xz),f_0(yz)$ are two distinct points of the set $F_0\cap M_s^\circ$, which contradicts Lemma~\ref{l:Falpha}(3). This contradiction shows that $\gamma\ne0$.

The minimality of $\gamma>0$ and the inductive condition (9) imply that $\gamma$ is a succesor ordinal and hence $\gamma=\delta+1$ for some ordinal $\delta$. The minimality of $\gamma$ ensures that $|\Sf_{f_\delta}(x;M_r^\circ)|\le 1$. 

If $\Sf_{f_\delta}(x;M_r^\circ)$ contains some point $y$, then for the number $t\defeq f_\delta(xy)\in f_\delta[E_{f_\delta}]$ we have $|\Sf_{f_\delta}(x;t)|=|\Sf_{f_\delta}(x;M_t^\circ)|=1$ and hence $\langle x,t\rangle\in P_{\delta+1}=P_\gamma$.
In this case the inductive condition (10i) ensures that $\Sf_{p_{x,t}}(x;r)\subseteq \Sf_{g_{\gamma}}(x;r)\subseteq \Sf_{g_\beta}(x;r)$ and hence $|\Sf_{g_\beta}(x;r)|\ge|\Sf_{p_{x,t}}(x;r)|\ge 2$ as the graph metric $p_{x,t}$ is isomorphic to the graph metric $\mu_t$ and $r\in M_t^\circ$. 

Next, assume $\Sf_{f_\delta}(x;M_r^\circ)=\emptyset$.  Since $|\Sf_{f_\gamma}(x;M_r^\circ)|\ge 2$, there exist distinct points $y,z\in \Sf_{f_\gamma}(x;M_r^\circ)$. Then $\{xy,xz\}\subseteq E_{f_\gamma}\setminus E_{f_\delta}$. Assuming that $xy\in E_{g_\gamma}$, we conclude that $xy\in E_{g_\gamma}\setminus E_{f_\delta}\subseteq \bigcup_{\langle u,s\rangle \in P_\gamma}E_{p_{u,s}}$ and hence $xy\in E_{p_{u,s}}$ for some pair $\langle u,s\rangle \in P_\gamma$. Then $f_\gamma(xy)=p_{u,s}(xy)\in M_r^\circ\cap M_s^\circ$ and hence $r\in M_s^\circ$. Since the graph metric $p_{u,s}$ is isomorphic to the graph metric $\mu_s$ and $r\in M_s^\circ$, the sphere $\Sf_{p_{u,s}}(x;s)$ has cardinality $2$ and $|\Sf_{g_\beta}(x;r)|\ge|\Sf_{g_\gamma}(x;r)|\ge|\Sf_{p_{u,s}}(x;r)|=2$. 
By analogy we can show that $|\Sf_{f_\beta}(x,r)|\ge 2$ if $xy\in E_{g_\gamma}$. 

It remains to consider the case $xy,xz\notin E_{g_\gamma}$. In this case $\{xy,xz\}\subseteq E_{f_\gamma}\setminus E_{g_\gamma}$ and by the inductive condition (4), $\{f_{\gamma}(xy),f_\gamma(xz)\}$ are two distinct points of the set $F_\gamma\cap M_r^\circ$, which contradicts Lemma~\ref{l:Falpha}(3).
\end{proof}

In the following claim we prove that the full metric $f_\beta$ satisfies the inductive condition (7).

\begin{claim}   $\forall r\in \IR_+\;\forall x,y\in V_{f_\beta}\;\;\big(\Sf_{f_\beta}(x;M_r)=\{x,y\}\Leftrightarrow \Sf_{f_\beta}(y;M_r)=\{y,x\}\big)$.
\end{claim}

\begin{proof} Take any $r\in\IR_+$ and $x,y\in V_{f_\beta}$. Assuming that $\Sf_{f_\beta}(x;M_r)=\{x,y\}$, we shall prove that $\Sf_{f_\beta}(y;M_r)=\{y,x\}$. If $x=y$, then $\Sf_{f_\beta}(y;M_r)=\Sf_{f_\beta}(x;M_r)=\{y,x\}$ and we are done. So, assume that $x\ne y$.


Let $s\defeq f_\beta(xy)\in M_r^\circ$ and observe that $M_s^\circ=M_r^\circ$.  If $xy\in E_{f_\alpha}$, then $\{x,y\}\subseteq \Sf_{f_\alpha}(x;M_r)\subseteq\Sf_{f_\beta}(x;M_r)=\{x,y\}$ implies $\Sf_{f_\alpha}(x;M_r)=\{x,y\}$. Taking into account that $s=f_\beta(xy)=f_\alpha(xy)\in f_\alpha[E_{f_\alpha}]$ and $\Sf_{f_\alpha}(x;s)=\Sf_{f_\alpha}(x;M^\circ_s)=\Sf_{f_\alpha}(x;M_r^\circ)=\{y\}$, we conclude that $\langle x,s\rangle\in P_\beta$ and the set $$\Sf(f_\beta;M_r)=\Sf_{f_\beta}(x;M_s)\supseteq \Sf_{g_\beta}(x;M_s)\supseteq \Sf_{p_{x,s}}(x;M_s)$$is infinite, which contradicts the equality $\Sf_{f_\beta}(x;M_r)=\{x,y\}$. This contradiction shows that $xy\notin E_{f_\alpha}$. 

If $xy\in E_{g_\beta}$, then $xy\in E_{g_\beta}\setminus E_{f_\alpha}\subseteq\bigcup_{\langle z,t\rangle\in P_\beta}E_{p_{z,t}}$ and hence $xy\in E_{p_{z,t}}$ for some $\langle z,t\rangle\in P_\beta$. Then $s=f_\beta(xy)=p_{z,t}(xy)\in p_{z,t}[E_{p_{z,t}}]=\mu_t[E_{\mu_t}]=M_t^\circ$ and $M_r^\circ =M^\circ_s=M_t^\circ$.
Taking into account that the graph metric $p_{z,t}$ is isomorphic to the graph metric $\mu_t$ and $\Sf_{\mu_t}(c;M_t^\circ)$ is infinite for every $c\in V_{\mu_t}$, we conclude that the sphere
$$\Sf_{f_\beta}(x;M_r)\supseteq\Sf_{g_\beta}(x;M_t)\supseteq \Sf_{p_{z,t}}(x;M_t)$$
is infinite, which contradicts or assumption. This contradiction shows that $xy\notin E_{g_\beta}$ and then $s=f_\beta(xy)\in f_\beta[E_{f_\beta}\setminus E_{g_\beta}]\subseteq F_\beta$ by the inductive assumption (4). The condition (2) of Lemma~\ref{l:Falpha} ensures that $F_\beta$ is disjoint with the set $\IQ B_\beta\cup\bigcup_{\gamma<\beta}\IQ F_\gamma$ and hence $\IQ s\cap(\IQ B_\beta\cup\bigcup_{\gamma<\beta}\IQ F_\gamma)=\{0\}$. On the other hand, Claim~\ref{cl:15.15} guarantees that
$$g_\beta[E_{g_\beta}]\subseteq \bigcup _{t\in B_\beta}M^\circ_t\cup\bigcup_{\gamma<\beta}\bigcup_{t\in F_\gamma}M_t^\circ$$and hence $\IQ r\cap g_\beta[E_{g_\beta}]=\IQ s\cap g_\beta[E_{g_{\beta}}]=\emptyset$. 

It is clear that $\{y,x\}\subseteq\{y\}\cup \Sf_{f_\beta}(y;s)\subseteq \Sf_{f_\beta}(y;M_r)$.  Assuming that $\Sf_{\beta}(y; M_r)\ne\{y,x\}$, we can find a point $z\in\Sf_\beta(y;M_r)\setminus\{y,z\}$. It follows from $M_r\cap g_\beta[E_{g_\beta}]\subseteq \IQ r\cap g_\beta[E_{g_\beta}]=\emptyset$ that $\{yz,yx\}\subseteq E_{f_\beta}\setminus E_{g_\beta}$. By the inductive condition (4), $f_\beta(yx)$ and $f_\beta(yz)$ are two distinct points of the set $F_\beta\cap M_r^\circ$, which contradicts Lemma~\ref{l:Falpha}(3). This contradiction shows that $\Sf_{f_\beta}(y;M_r)=\{x,y\}$.

By analogy we can prove that $\Sf_{f_\beta}(y;M_r)=\{y,x\}$ implies $\Sf_{f_\beta}(x;M_r)=\{x,y\}$.
\end{proof}

In the following claim we prove that the full metric $f_\beta$ satisfies the inductive condition (8).

\begin{claim} For every $x\in V_{f_\beta}$ and $r\in \IR_+$ we have $|\Sf_{f_\beta}(x;r)|\le 2$ and $f_\beta[[\Sf_{f_\beta}(x;r)]^2]\subseteq\{2r\}$.
\end{claim}

\begin{proof} 
If $|\Sf_{f_\beta}(x;r)|\le 1$, then $[\Sf_{f_\beta}(x;r)]^2=\emptyset$ and $f_\beta[[\Sf_{f_\beta}(x;r)]^2]=\emptyset\subseteq\{2r\}$. So, we assume that 
$|\Sf_{f_\beta}(x;r)|\ge 2$. Then $r\in f_\beta[E_{f_\beta}]=f_\beta[E_{f_\beta}\setminus E_{g_\beta}]\cup g_\beta[E_{g_\beta}]$.  The condition (2) of Lemma~\ref{l:Falpha} and Claim~\ref{cl:15.15} imply that 
$$F_\beta\cap g_\beta[E_{g_\beta}]\subseteq F_\beta\cap \Big(\bigcup_{r\in B_\beta}M^\circ_r\cup\bigcup_{\gamma\in\beta}\bigcup_{r\in F_\gamma}M^\circ_r\Big)=\emptyset.$$
If $r\notin g_\beta[E_{g_\beta}]$, then $r\in f_\beta[E_{f_\beta}\setminus E_{g_\beta}]\subseteq F_\beta\setminus g_\beta[E_{g_\beta}]$ and $|\Sf_{f_\beta}(x;r)|\le 1$, by the inductive condition (4). This is a contradiction showing that $r\in g_\beta[E_\beta]\setminus F_\beta$. It follows from $f_\beta[E_{f_\beta}\setminus E_{g_\beta}]\subseteq F_\beta$  and the inductive condition (10)  that $$\Sf_{f_\beta}(x;r)=\Sf_{g_\beta}(x;r)=\Sf_{f_\alpha}(x;r)\cup\bigcup_{\langle z,t\rangle\in P_\beta}\Sf_{p_{z,t}}(x;r)=\Sf_{f_\alpha}(x;r)\cup\bigcup_{\langle z,t\rangle\in P'_\beta}\Sf_{p_{z,t}}(x;r),$$
where $$P_\beta'\defeq\{\langle z,t\rangle\in P_\beta:(x\in V_{p_{z,t}})\;\wedge\;( r\in M^\circ_t)\}.$$ For every $\langle z,t\rangle\in P_\beta'$ we have $M^\circ_t=M^\circ_r$.

If $P_\beta'=\emptyset$ and $x\notin V_{f_\alpha}$, then the sphere $\Sf_{f_\beta}(x;r)$ is empty and hence $f_\beta[[\Sf_{f_\beta}(x;r)]^2]=\emptyset\subseteq\{2r\}$.

If $P'_\beta=\emptyset$ and $x\in V_{f_\alpha}$, then $\Sf_{f_\beta}(x;r)=\Sf_{f_\alpha}(x;r)$. In this case the inductive assumption (8) ensures that
$$|\Sf_{f_\beta}(x;r)|=|\Sf_{f_\alpha}(x;r)|\le 2\quad\mbox{and}\quad f_\beta[[\Sf_{f_\beta}(x;r)]^2]=f_\alpha[[\Sf_{f_\alpha}(x;r)]^2]\subseteq\{2r\}.$$ 

It remains to consider the case of $P'_\beta\ne\emptyset$. In this case we shall prove that the set $\{p_{\langle z,t\rangle}:\langle z,t\rangle\in P_\beta'\}$ is a singleton.

To derive a contradiction, assume that $\langle y,s\rangle,\langle z,t\rangle$ two pairs in $P_\beta'$ such that $p_{y,s}\ne p_{z,t}$. Since $M_s^\circ=M_r^\circ=M_t^\circ$, the inductive conditions (10iii) and (10iv) ensure that $\Sf_{f_\alpha}(y;M_r)\ne \Sf_{f_\alpha}(z;M_r)$, $y\ne z$, and $x\in V_{p_{y,s}}\cap V_{p_{z,t}}= \Sf_{f_\alpha}(y;M_s)\cap\Sf_{f_\alpha}(z;M_t)$. The choice of $\langle y,s\rangle,\langle z,t\rangle\in P_\beta$ guarantees that $|\Sf_{f_\alpha}(y;s)|=|\Sf_{f_\alpha}(y;M_s^\circ)|\le 1$ and  $|\Sf_{f_\alpha}(z;t)|=|\Sf_{f_\alpha}(z;M_t^\circ)|\le 1$.

If $x\ne y$, then $x\in \Sf_{f_\alpha}(y;M_s)$ and $|\Sf_{f_\alpha}(y;M^\circ_s)|\le 1$ implies $\Sf_{f_\alpha}(y;M_r)=\Sf_{f_\alpha}(y;M_s)=\{y,x\}$. By the inductive condition (7), $\Sf_{f_\alpha}(x;M_r)=\{x,y\}$. If $z=x$, then $\Sf_{f_\alpha}(y;M_s)=\{x,y\}=\Sf_{f_\alpha}(x;M_r)=\Sf_{f_\alpha}(z;M_t)$, which is a contradiction showing that $z\ne x$ and hence $\Sf_{f_\alpha}(z;M_r)=\{z,x\}$. By the inductive condition (7), $\{x,y\}=\Sf_{f_\alpha}(x;M_r)=\{x,z\}$ and hence $y=z$, which is a final contradiction showing that the set $\{p_{z,t}:\langle z,t\rangle\in P_\beta'\}$ is a singleton. 

Fix any pair $\langle z,t\rangle\in P'_\beta$ and observe that
$$\Sf_{f_\alpha}(x;r)=\Sf_{f_\alpha}(x;r)\cup\Sf_{p_{z,t}}(x;r).$$

\begin{claim}\label{cl:last} $\Sf_{f_\alpha}(x;r)\subseteq \Sf_{p_{z,t}}(x;t)$.
\end{claim}

\begin{proof} The claim holds trivially if $\Sf_{f_\alpha}(x;r)=\emptyset$. So, we assume that $\Sf_{f_\alpha}(x;r)\ne\emptyset$. The choice of $\langle z,t\rangle\in P_\beta'$ and the inductive condition (10ii) guarantee that $x\in V_{f_\alpha}\cap V_{p_{z,t}}=\Sf_{f_\alpha}(z;M_t)$ and $|\Sf_{f_\alpha}(z;M_t)|=|\{x\}\cup\Sf_{f_\alpha}(z;M_t^\circ)|\le 2$. Also $\langle z,t\rangle\in P_\beta'$ implies $r\in M^\circ_t$. If $\Sf_{f_\alpha}(z;M_t^\circ)=\emptyset$, then $x\in \Sf_{f_\alpha}(z;M_t)=\{z\}$ and $\emptyset\ne \Sf_{f_\alpha}(x;r)\subseteq\Sf(z;M^\circ_t)=\emptyset$, which is a contradiction showing that $\Sf_{f_\alpha}(z;M^\circ_t)$ is a singleton.

If $z\ne x$, then $\Sf_{f_\alpha}(z;M_t)=\{z,x\}$. The inductive condition (7) guarantees that $\Sf_{f_\alpha}(x;M_t)=\{x,z\}$. Then $\emptyset\ne\Sf_{f_\alpha}(x;r)=\Sf_{f_\alpha}(x,M^\circ_t)=\{z\}$. The inductive condition (10ii), ensures that $\{x,z\}=\Sf_{f_\alpha}(z;M_t)=V_{p_{z,t}}\cap V_{f_\alpha}$. On the other hand, the inclusion $r\in M_t$ ensures that $r=f_\alpha(xz)=p_{z,t}(xz)$ and hence $\Sf_{f_\alpha}(x;r)=\{z\}\subseteq \Sf_{p_{z,t}}(x;r)$.

It remains to consider the case $z=x$. In this case $r\in M_t$ implies 
$\emptyset\ne \Sf_{f_\alpha}(x;r)\subseteq\Sf_{f_\alpha}(x;M_t^\circ)=\Sf_{f_\alpha}(z;M_t^\circ)$. The choice of $\langle z,t\rangle\in P_\beta$ ensures that the nonempty set $\Sf_{f_\alpha}(z;M_t^\circ)$ is a singleton and hence $\Sf_{f_\alpha}(x;r)=\Sf_{f_\alpha}(x;M_t^\circ)=\{y\}$ for some $y\in V_{f_\alpha}$. Then $\{x,y\}=\{z,y\}=\Sf_{f_\alpha}(z;M_t)=V_{p_{z,t}}\cap V_{f_\alpha}$, by the inductive condition (10ii). On the other hand, the inclusion $r\in M_t$ ensures that $r=f_\alpha(xy)=p_{z,t}(xy)$ and hence $\Sf_{f_\alpha}(x;r)=\{y\}\subseteq \Sf_{p_{z,t}}(x;r)$. This completes the proof of Claim~\ref{cl:last}.
\end{proof}

By Claim~\ref{cl:last},
$$\Sf_{f_\beta}(x;r)=\Sf_{f_\alpha}(x;r)\cup\Sf_{p_{z,t}}(x;r)=\Sf_{p_{z,t}}(x;r).$$
Taking into account that the graph metric $p_{z,t}$ is isomorphic to the graph metric $\mu_t$, we conclude that
$|\Sf_{f_\beta}(x;r)|=|\Sf_{p_{z,t}}(x;r)|\le 2$ and $f_\beta[[\Sf_{f_\beta}(x;r)]^2]=f_\alpha[[\Sf_{p_{z,t}}(x;r)]^2]\subseteq\{2r\}$.
\end{proof}
This completes the proof of Lemma~\ref{l:step}. 
\end{proof}
The inductive conditions (2) and (3) of Lemma~\ref{l:step} imply that $f\defeq\bigcup_{\alpha\in\kappa}f_\alpha=\bigcup_{\alpha\in\kappa}g_\alpha$ is a well-defined full metric.

Consider the metric space $X\defeq V_{f}$ endowed with the metric $\dist:X\times X\to\IR$ defined by $$\dist(x,y)=\begin{cases}f(xy)&\mbox{if $x\ne y$};\\
0&\mbox{if $x=y$}.
\end{cases}
$$

It remains to prove that $(X,\dist)$ is a Banakh space and $\dist[X^2]=M$. To see that $\dist[X^2]\subseteq M$, observe that 
$$\dist[X^2]=\{0\}\cup f[E_f]=\{0\}\cup\bigcup_{\beta\in\kappa}\big(g_\beta[E_{g_\beta}]\big)
\subseteq \{0\}\cup\bigcup_{\beta\in\kappa}\Big(\bigcup_{r\in B_\beta}M^\circ_r\cup\bigcup_{\gamma\in\beta}\bigcup_{r\in F_\gamma}M^\circ_r\Big)\subseteq \bigcup_{r\in M}M_r=M,
$$by the inductive conditions (5) and (9) of Lemma~\ref{l:step}.

To see that $M\subseteq\dist[X^2]$, take any $x\in X$ and $r\in M\setminus \{0\}$ and find an ordinal $\alpha\in\kappa$ such that $x\in V_{f_\alpha}$ and $r\in B_\alpha$. The inductive condition (6) ensures $|\Sf_f(x;r)|\ge|\Sf_{g_{\alpha+1}}(x;r)|\ge 2$ and hence $r\in f[E_f]\subseteq\dist[X^2]$.

To see that $(X,\dist)$ is a Banakh space, take any point $c\in X$ and real number $r\in \dist[X^2]=M$. We have to prove that $\Sf(c;r)=\{x,y\}$ for some points $x,y\in X$ with $\dist(x,y)=2r$. If $r=0$, then the points $x=c=y$ have the desired property. So, assume that $r>0$. Since $c\in V_f=\bigcup_{\alpha\in\kappa}V_{f_\alpha}$ and $M\setminus\{0\}=\bigcup_{\alpha\in\kappa}B_\alpha$, there exists an ordinal $\alpha\in\kappa$ such that $c\in V_{f_\alpha}$ and $r\in B_\alpha$.    The inductive condition (6) ensures $|\Sf_{g_{\alpha+1}}(c;r)|\ge 2$ and hence the sphere $\Sf_{g_{\alpha+1}}(c;r)$ contains two distinct points $x,y$. The inductive condition (8) ensures that for every ordinal $\beta$ with $\alpha<\beta<\kappa$, we have $|\Sf_{f_\beta}(c;r)|\le 2$ and $f_\beta[[\Sf_{f_\beta}(c;r)]^2]\subseteq \{2r\}$. This implies that $\Sf_{f_\beta}(c;r)=\{x,y\}$ and $f_{\beta}(xy)=2r$. Consequently, $$\Sf(c;r)=\Sf_f(c;r)=\bigcup_{\beta=\alpha+1}^\kappa\Sf_{f_\beta}(c;r)=\{x,y\}$$and $\dist(x,y)=f(xy)=f_{\alpha+1}(xy)=2r$, witnessing that $(X,\dist)$ is a Banakh space. \hfill $\square$


\section{Acknowledgements}

The authors expresses his sincere thanks to 
\begin{itemize}
\item Will Brian who resolved author's question \cite{TB} at \cite{WB} thus suggesting an idea of the proof of Theorems~\ref{t:main1}  and \ref{t:WB};
\item Pavlo Dzikovskyi for valuable discussions on Banakh spaces and suggesting the proofs of Lemma~\ref{l:Dzik1} and Theorem~\ref{t:Dzik2};
\item Pietro Mayer for his answer \cite{PM} to the question \cite{TB}, which suggested a crucial idea in constructing exotic Banakh groups in Theorems~\ref{t:H1} and \ref{t:H2}, and also for his (now deleted) answer to the question \cite{TB2} and suggesting a crucial idea that allowed to elaborate the technique of game extensions of floppy graph metrics \cite{BM}, which was exploited in the proof of Theorem~\ref{t:exotic};
\item the {\tt MathOverflow} user {\tt user42355} for his comment under the answer of Pietro Mayer \cite{PM}; the comment contained a crucial idea realized in the proof  of Lemma~\ref{l:1}.
\end{itemize}

\end{document}